\documentclass[svgnames]{amsart}
\usepackage[english]{babel}
\usepackage[utf8]{inputenc}

\usepackage{hyperref}
\hypersetup{colorlinks,citecolor=Black,linkcolor=Black,urlcolor=Blue}

\usepackage{tabularx, ifthen}
\usepackage{amssymb, amsfonts, amsmath, amsthm, amscd, amsxtra, latexsym, mathabx, mathtools}
\usepackage{enumerate, enumitem}

\usepackage{tikz}\usetikzlibrary{cd}
\theoremstyle{plain}
\newtheorem{thm}{Theorem}[section] 
\newtheorem{lemma}[thm]{Lemma}
\newtheorem{cor}[thm]{Corollary}
\newtheorem{prop}[thm]{Proposition}
\newtheorem{thmintro}{Theorem}

\newtheorem{corintro}[thmintro]{Corollary}
\newtheorem{propintro}[thmintro]{Proposition}
\newtheorem{claim}[thm]{Claim}

\theoremstyle{definition}
\newtheorem{defn}[thm]{Definition}
\newtheorem{rem}[thm]{Remark}
\newtheorem{ex}[thm]{Example}
\newtheorem{questintro}[thmintro]{Question}
\newtheorem{remintro}[thmintro]{Remark}
\newtheorem{conj}[thmintro]{Conjecture}

\newtheorem{notation}[thm]{Notation}
\newtheorem{assumption}[thm]{Standing assumption}
\newtheorem{ax}{Axiom}

\newenvironment{claimproof}{\begin{proof}}{\end{proof}}

\newcommand{\Z}{\mathbb{Z}}
\newcommand{\R}{\mathbb{R}}
\newcommand{\Cay}[2]{\operatorname{Cay}\left(#1,#2\right)}

\newcommand{\frakS}{\mathfrak S}
\newcommand{\diam}{\mathrm{diam}}
\newcommand{\C}{\mathcal C}
\newcommand{\ov}[1]{\overline{#1}}
\newcommand{\U}{\mathcal{U}}

\newcommand{\dist}{\mathrm{d}}
\newcommand{\nest}{\sqsubseteq}
\newcommand{\propnest}{\sqsubsetneq}
\newcommand{\orth}{\bot}
\newcommand{\transverse}{\pitchfork}
\newcommand{\link}{\operatorname{Lk}}

\newcommand{\Stab}[2]{\operatorname{Stab}_{#1}(#2)}

\newcommand{\duaug}[2]{{#1}^{+{#2}}}
\newcommand{\Sat}{\mathrm{Sat}}
\newcommand{\Act}{\operatorname{Act}}

\newcommand{\N}{\mathcal{N}}
\newcommand{\B}{\mathcal{B}}
\newcommand{\W}{\mathcal{W}}
\newcommand{\calS}{\mathcal{S}}
\newcommand{\gate}{\mathfrak{g}}
\newcommand{\p}{\mathfrak{p}}

\newcommand{\h}{\widehat}

\newcommand{\MCG}{\mathcal{MCG}^\pm}

\newcommand{\Squid}{\operatorname{Cone}}

\newcommand{\tsh}[1]{\left\{\kern-.7ex\left\{#1\right\}\kern-.7ex\right\}}

\renewcommand{\angle}{\sphericalangle}

\usepackage{epigraph}

\title[Short HHG II]{Short hierarchically hyperbolic groups II:\\
Quotients and the Hopf property for Artin groups}

\author[G. Mangioni]{Giorgio Mangioni}
    \email{gm2070@hw.ac.uk}

\author[A. Sisto]{Alessandro Sisto}
    \address{(Alessandro Sisto) Maxwell Institute and Department of Mathematics, Heriot-Watt University, Edinburgh, UK}
    \email{a.sisto@hw.ac.uk}

\begin{document}

\begin{abstract}
We prove that most Artin groups of large and hyperbolic type are Hopfian, meaning that every self-epimorphism is an isomorphism. The class covered by our result is generic, in the sense of Goldsborough-Vaskou. Moreover, assuming the residual finiteness of certain hyperbolic groups with an explicit presentation, we get that all large and hyperbolic type Artin groups are residually finite. We also show that ``most'' quotients of the five-holed sphere mapping class group are hierarchically hyperbolic, up to taking powers of the normal generators of the kernels. 
\\
The main tool we use to prove both results is a Dehn-filling-like procedure for short hierarchically hyperbolic groups (these also include e.g. non-geometric 3-manifolds, and triangle- and square-free RAAGs). 
\end{abstract}
\maketitle


\setcounter{tocdepth}{1}
\tableofcontents

\section*{Introduction}

By work of Osin \cite{Osin_Dehn_Fill}, and independently Groves and Manning \cite{GM}, if a group $G$ is hyperbolic relative to a subgroup $P$, and $N\unlhd P$ is a ``sufficiently deep'' normal subgroup, then $G/\langle\langle N\rangle\rangle$ is hyperbolic relative to $P/N$. This quotient procedure is called \emph{Dehn filling}, as it abstracts Thurston's Dehn filling procedure for cusped hyperbolic 3-manifolds \cite{Thurston_topology3m}. In this paper, we develop an analogue Dehn filling technique for \emph{short hierarchically hyperbolic groups}, first introduced in \cite{Short_HHG:I}. This family includes the mapping class group of the five-holed sphere, extra-large Artin groups, fundamental groups of non-geometric 3-manifolds, right-angled Artin groups without triangles and squares in their defining graphs, and many others. Our tools allow us to construct hyperbolic quotients of these groups; in turn this has several applications, the main one being related to the Hopf property for Artin groups, as we now explain.

\subsection*{Hopf property for Artin groups}
The motivation of our Dehn filling machinery is to “lift” properties from the quotients to the groups themselves. In fact, we study when a group $G$ within our class is \emph{Hopfian}, which means that every surjective homomorphism $G\to G$ is an isomorphism. This property is very useful to produce automorphisms of $G$, because, while it is usually easy to check that a map is surjective (it is enough that a generating set is contained in the image), proving injectivity can be quite fiddly.

Within our class, we prove the Hopf property for “most” Artin groups of \emph{large and hyperbolic type}, by which we mean that all edge labels in the defining graph are at least $3$ and there is no triangle whose labels are all $3$. The exact statement is Theorem~\ref{thm:XL_Hopfian}, from which we extract some special cases. First, recall that an Artin group is \emph{even} if all edge labels in its defining graph are even. 

\begin{thmintro}\label{thmintro:even_hopf}
    Let $A_\Gamma$ be an even large-type Artin group. Then $A_\Gamma$ is Hopfian.
\end{thmintro}

For the second application, we say that an Artin group has a \emph{single odd component} if every two vertices in the defining graph are connected by a combinatorial path with odd labels. This is a natural notion, as two generators of an Artin group are conjugate if and only if the corresponding vertices are connected by an odd path \cite{paris_artin_groups}. 

\begin{thmintro}\label{thmintro:odd_hopf}
    Let $A_\Gamma$ be an Artin group of large and hyperbolic type with a single odd component. Then $A_\Gamma$ is Hopfian.
\end{thmintro}

As explained in Remark~\ref{rem:generic_artin}, the aforementioned class of Artin groups is generic in the sense of \cite{GV_random}, so we obtain:

\begin{corintro}\label{corintro:generic}
    A generic Artin group is Hopfian.
\end{corintro}

The Corollary was also very recently obtained in \cite{BMV}, where, by completely different means, the authors identify two other classes which turn out to be generic (see Remark~\ref{rem:bmv} for a comparison between the families).
\par\medskip
In \cite{barak2024equational}, Barak proved that all HHG satisfying a technical condition are equationally Noetherian, hence Hopfian (see e.g. \cite[Corollary 3.14 and Theorem D]{Groves_Hull:eq_noetherian} for the implication). The further requirement however, is not satisfied by our groups, as discussed in Remark~\ref{rem:no_we_cant_barak}.

\subsubsection*{Towards residual finiteness}
It is a widespread belief that all Artin groups are residually finite \cite[Problem 8.05]{aim}. Residual finiteness, which would imply the Hopf property, has only ben established for certain classes of Artin groups; in fact, to the best of our knowledge, it is not known whether generic Artin groups are residually finite (see Remark~\ref{rmk:resfin} for an overview). In Figure~\ref{fig:comparison} we even exhibit a four-generated Artin group which is Hopfian by Theorem~\ref{thmintro:odd_hopf} but is not known to be residually finite.

However, our Theorem~\ref{thm:XL_Hopfian} is proven by constructing “sufficiently many” hyperbolic quotients, of which an explicit presentation is given in Remark~\ref{rmk:Shephard}. All hyperbolic groups are Hopfian by e.g. \cite[Corollary 6.13]{ReinfeldWeidmann}, and if our quotients were in fact residually finite, then our Artin groups would be as well. We summarise this in the following:

\begin{thmintro}\label{thmintro:artin_reshyp}
    Let $A_\Gamma$ be an Artin group of large and hyperbolic type. Then $A_\Gamma$ is residually hyperbolic. Furthermore, if all the hyperbolic groups from Remark~\ref{rmk:Shephard} are residually finite, then $A_\Gamma$ is residually finite.
\end{thmintro}

In the same Remark~\ref{rmk:Shephard}, we relate residual finiteness of $A_\Gamma$ to residual finiteness of certain \emph{Shephard groups}. This class, named after G. Shephard \cite{shephard}, generalises both Coxeter and Artin groups, as one is allowed to require that some of the generators in an Artin presentation have finite order. 

\begin{remintro}\label{rem:comparison with goldman}
In \cite{goldman20242dimensionalshephardgroups} Goldman proved that any large and hyperbolic type Artin group $A_\Gamma$ is residually hyperbolic. Our work and Goldman's are independent, use different tools, and were carried out simultaneously. However, we both obtain the following intermediate result: for a suitable choice of a positive integer $N$, the Shephard group obtained by killing the $N$-th power of all standard generators of $A_\Gamma$ is hyperbolic relative to the images of the dihedral subgroups. As a consequence, this group admits hyperbolic ``Dehn filling'' quotients, either by our Theorem \ref{thmintro:dehnfill} below or by the ``standard'' relatively hyperbolic Dehn filling Theorem. Furthermore, under the further assumption that $\Gamma$ is triangle-free, Goldman was able to obtain residual finiteness \cite[Corollary G]{goldman20242dimensionalshephardgroups}. 
\end{remintro}

\subsection*{Quotients of $\MCG(S_{0,5})$} Only a handful of mapping class  groups of finite-type surfaces are short HHG, including that of the five-holed sphere, which we focus on here. There are various ways to take quotients of mapping class groups that yield hierarchically hyperbolic groups, including quotients by powers of pseudo-Anosovs \cite{hhs_asdim} and quotients by powers of Dehn twists \cite{BHMS}. These quotients have been further studied by exploiting hierarchical hyperbolicity to obtain quasi-isometric and algebraic rigidity results \cite{rigidity_mcg_mod_dt, mangioni2023rigidityresultslargedisplacement}. They have also notably been used to relate famous questions about profinite properties of mapping class groups and profinite rigidity of certain 3-manifolds to residual finiteness of certain hyperbolic groups \cite{BHMS, Wilton}. 

In this context, we previously asked, roughly, whether given any finite collection of elements of a mapping class group one can quotient by suitable powers and obtain a hierarchically hyperbolic group, see \cite[Question 3]{rigidity_mcg_mod_dt}. We provide an almost complete answer for $\MCG(S_{0,5})$:

\begin{thmintro}[{see Theorem~\ref{thm:S5_quotient}}]\label{thmintro:mcg}
    Let $S=S_{0,5}$, and let $g_1,\dots,g_l\in \MCG(S)$. 
    Suppose that every partial pseudo-Anosov $g_i$ has no hidden symmetries. Then there exists $N\in\mathbb{N}-\{0\}$ such that, for all $K_1,\dots, K_l\in \mathbb{Z}-\{0\}$ we have that $\MCG(S)/\langle\langle \{g_i^{K_i N}\}\rangle\rangle$ is hierarchically hyperbolic.
\end{thmintro}

Here, not having hidden symmetries is a technical condition which only partial pseudo-Anosovs can satisfy, see Definition~\ref{defn:hiddensymm_general}. Unfortunately we needed this additional requirement due to fine algebraic reasons, but we do not think that it is necessary, and in fact we believe that there is now enough evidence to upgrade our question to a conjecture:

\begin{conj}
    Let $S$ be any surface of finite type, and let $g_1,\dots,g_l\in \MCG(S)$. There exists $N\in\mathbb{N}-\{0\}$ such that, for all $K_1,\dots, K_l\in \mathbb{Z}-\{0\}$, the group $\MCG(S)/\langle\langle \{g_i^{K_i N}\}\rangle\rangle$ is hierarchically hyperbolic.
\end{conj}
To the best of our knowledge, the conjecture is open even for quotients by suitable powers of non-separating Dehn Twists. What makes the conjecture interesting is that tackling it should lead to a more complete theory of Dehn fillings for hierarchically hyperbolic groups, and in turn this should have many applications beyond those presented in this paper and \cite{BHMS, Wilton}.

\subsection*{Dehn fillings of short HHG}
\subsubsection*{Main result}
Roughly, a short HHG contains specified subgroups which are $\mathbb Z$-central extensions of hyperbolic groups. We call \emph{cyclic directions} the kernels of these extensions (see Subsection~\ref{defn:short_HHG} for the full definition of a short HHG). For the five-holed sphere mapping class group, the specified extensions are curve stabilisers, and cyclic directions are generated by Dehn twists. 

Short HHG are examples of \emph{hierarchically hyperbolic groups}, first introduced in \cite{HHS_I}. The only consequence of hierarchical hyperbolicity that the reader should bear in mind for this Introduction is that a short HHG is hyperbolic if the list of specified subgroups is empty. Hence, in order to make a short HHG “more hyperbolic”, the idea is to take a generator of a cyclic direction and quotient by a power of it, as such an element has “large” centraliser and in a hyperbolic group this is allowed only if the element has finite order. With this in mind, we now present the main technical result of this paper, which should be thought of as an analogue of the relatively hyperbolic Dehn filling theorem:

\begin{thmintro}[{see Theorem~\ref{thm:quotient_short_HHG}}]\label{thmintro:dehnfill}
Let $G$ be a short HHG, and let $g_1,\dots,g_n$ be generators of some of its cyclic directions. Then there exists $M\in\mathbb{N}-\{0\}$ such that, for all choices $k_i\in \Z-\{0\}$, the quotient $\ov{G}= G/\langle\langle \{g_i^{k_iM}\}\rangle\rangle$ is a short HHG.
\end{thmintro}

In fact, the Theorem also gives a natural short HHG structure on $\ov G$, where the cyclic directions are images of the cyclic directions of $G$ that do not contain conjugates of the $g_i$. Thus, a particularly interesting case is where we take quotients by powers of generators of all cyclic directions, hence obtaining a hyperbolic group. A refinement of the above yields the following, which readily implies Theorem \ref{thmintro:artin_reshyp}. Recall that a group $G$ is \emph{fully residually $P$} for some property $P$ if, for every finite subset $F\subset G$, there exists a quotient $G\to \ov G$ where $F$ injects, and such that $\ov G$ enjoys $P$.

\begin{corintro}[{see Corollary~\ref{cor:fullres_hyp}}]\label{corintro:fullres}
    Short HHG are fully residually hyperbolic.
\end{corintro}

\subsubsection*{Tools and techniques}
The proof of Theorem~\ref{thmintro:dehnfill} combines two approaches. Firstly, as explored in \cite{Short_HHG:I}, the hierarchical structure of a short HHG can be modified by constructing suitable quasimorphisms on the specified $\Z$-central extensions (see Subsection~\ref{subsec:squid} for further details). In Subsection~\ref{subsec:preparation} we take advantage of this flexibility to make the structure of a short HHG ``as compatible as possible'' with the quotient projection. Secondly, we adapt the machinery of \emph{rotating families}, first introduced in \cite{dahmani:rotating}, to short HHG. Mimicking arguments from \cite{dfdt} and \cite{BHMS}, these tools allow one to lift certain combinatorial configurations from the quotient $\ov G$ to the original group $G$. Each HHG axiom for $\ov G$ follows from the corresponding statement for $G$, which is already hierarchically hyperbolic. The most novel and difficult part of this construction, compared with \cite{BHMS} and other papers, is that the HHS structure of the quotients involves quasilines coming from the aforementioned quasimorphisms; therefore, new ideas are required to construct suitable retractions onto those, as detailed in Subsection~\ref{subsubsec:qi_emb_quot}.

\subsubsection*{Comparison with known Dehn Filling results}
In \cite{BHMS}, the authors study quotients of mapping class groups of arbitrary finite-type surfaces by suitably large powers of \emph{all} Dehn twists, which as mentioned above should be thought of as the equivalent of what we call central directions. On the one hand, our Theorem~\ref{thmintro:dehnfill} applies to quotients by suitably large powers of \emph{any collection} of central directions. On the other, our techniques fail if we consider the quotient of a non-short mapping class group by large enough powers of some Dehn twists, such as only those around non-separating curves. The problem is roughly as follows: in any reasonable candidate HHS structure for the quotient, the image of a curve stabiliser will be a product region, and therefore hierarchically hyperbolic itself. Now, if a curve does not lie in the chosen collection, then the quotient image of its stabiliser will be a $\Z$-central extension, and by \cite[Corollary 4.3]{HRSS_3manifold} a necessary condition for it to be hierarchically hyperbolic is that its Euler class is \emph{bounded}. In the case of the five-holed sphere, the quotient extension has hyperbolic base, and so is bounded by \cite{neumannreeves}; however this is not necessarily true for surfaces of higher complexity. 

Motivated by this setup, in future work we will explore under which conditions a quotient of a $\Z$-central extension of a group $G$ remains bounded \cite{inpreparation}.

\subsubsection*{A criterion for hopficity}
The fundamental tool in our study of the Hopf property for short HHG is the following criterion (stated here in slightly simplified form), whose proof is straightforward but for which we could not find a suitable reference:
\begin{propintro}[{see Lemma~\ref{lem:enough}}]
Suppose that $G$ has \emph{enough Hopfian quotients}, meaning that, for every surjective homomorphism $\phi: G\to G$ and $g_0\in G-\{1\}$, there exists a quotient $H$ of $G$, say with quotient map $q$ such that:
\begin{itemize}
 \item $q(g_0)\neq 1$,
 \item $H$ is Hopfian,
 \item $\phi$ induces a homomorphism $\psi\colon H\to H$.
\end{itemize}
Then $G$ is Hopfian.
\end{propintro}

Now let $G$ be a short HHG, and let $\phi$, $g_0$ be as above. If $\phi$ maps central directions to central directions, then it induces a map of some hyperbolic (hence Hopfian \cite[Corollary 6.13]{ReinfeldWeidmann}) Dehn filling quotient $H$, obtained by annihilating suitable powers of all central directions, and by Corollary~\ref{corintro:fullres} we can also assume that $g_0$ survives in $H$. This is not always the case; however, $\phi$ often preserves “enough” central directions, and the Dehn filling quotient by those directions will still be “hyperbolic enough” to be Hopfian. More precisely, we shall make use of hopficity of certain relatively hyperbolic groups \cite{Groves_Hull:eq_noetherian} and the following fact, which we highlight as it is of independent interest.
\begin{thmintro}[{see Theorem \ref{thm:rel_hyp_hopf}}]
Let $G$ be hyperbolic relative to $\Z$-central extensions of hyperbolic groups (including the case that $G$ itself is such an extension). Then $G$ is Hopfian.
\end{thmintro}

\subsection*{Future directions}
It is natural to ask whether, using a version of our argument or otherwise, one can in fact show that all short HHG are Hopfian:

\begin{questintro}
    Are all short HHG Hopfian? More ambitiously, are all HHG Hopfian?
\end{questintro}

But in fact, this is already a major question in the context of Artin groups:

\begin{questintro}
    Are all large hyperbolic type Artin groups Hopfian?
\end{questintro}

Going one step further, given the connection to residual finiteness of certain hyperbolic groups, the following question could have important consequences:

\begin{questintro}\label{quest:artin_rf}
    Are all large hyperbolic type Artin groups residually finite?
\end{questintro}

As explained in Theorem~\ref{thmintro:artin_reshyp}, a positive answer to Question~\ref{quest:artin_rf} would follow from residual finiteness of an explicit collection of hyperbolic groups, so we ask:

\begin{questintro}
    Are all hyperbolic groups from Remark~\ref{rmk:Shephard} residually finite?
\end{questintro}

Finally, we believe that there are many more examples of short HHG, and in fact we believe that the following question has a positive answer:

\begin{questintro}
    Is a random quotient of a short HHG, for a suitable notion of randomness, again a short HHG?
\end{questintro}

\subsection*{Outline}
Section~\ref{sec:what_is_CHHS} contains background on combinatorial hierarchically hyperbolic groups, of which short HHG are an instance. In Section~\ref{sec:shortHHG}, we describe the family of short HHG, and collect results about them from \cite{Short_HHG:I}. The class of Dehn filling quotients we consider is then made precise in Subsection~\ref{sec:large_dt}. 
\par\medskip

In Section~\ref{sec:rotating} we adapt the machinery of rotating families from \cite{dahmani:rotating, dfdt} to our quotients of interest. We then construct a short HHG structure for the quotients in Section~\ref{sec:quotient_is_short}: see in particular Theorem~\ref{thm:quotient_short_HHG}, which is Theorem~\ref{thmintro:dehnfill} from the Introduction. As a by-product, in Subsection~\ref{subsec:reshyp} we prove residual hyperbolicity of short HHG: see Corollary~\ref{cor:fullres_hyp}, which is Corollary~\ref{corintro:fullres} above.

\par\medskip
Section~\ref{sec:hopf} develops tools to study self-epimorphisms of short HHG. An example of how they are put into practice is presented in Subsection~\ref{subsec:HNN}, where we prove the Hopf property for certain HNN extensions of the direct product of $\Z$ and a free group: see Example~\ref{example:zxF}. The same techniques are then pushed further in Section~\ref{subsec:example_artin}, where we prove that many Artin groups of large and hyperbolic type are Hopfian: see Theorem~\ref{thm:XL_Hopfian} for the full statement, which encompasses both Theorems~\ref{thmintro:even_hopf} and~\ref{thmintro:odd_hopf}. 
\par\medskip
In Section~\ref{sec:mcg_quot} we prove hierarchical hyperbolicity of many quotients of the five-holed sphere mapping class group: see Theorem~\ref{thm:S5_quotient}, which is Theorem~\ref{thmintro:mcg}. 

\subsection*{Acknowledgements}
We are grateful to Oli Jones for suggesting an enlightening way to shorten some proofs of the Hopf property. We also thank Giovanni Sartori and Nicolas Vaskou for answering many questions about Artin groups. Special thanks go to Martín Blufstein and Alexandre Martin for fruitful discussions. Finally, we thank the referee for useful comments, both on the mathematical content and on the exposition.

\section{Combinatorial HHS}\label{sec:what_is_CHHS} 
In this section we recall the definition of a combinatorial HHS and its hierarchically hyperbolic structure, as first introduced in \cite{BHMS}. The reader might want to refer to \cite[Section 1]{BHMS}, which contains discussion of all the various notions we recall below. Also, the reader might find \cite{converse} useful, as in there it is explained how to create a combinatorial HHS structure from an HHS structure (in many cases).

\begin{defn}[Induced subgraph]
    Let $X$ be a simplicial graph. Given a subset $S\subseteq X^{(0)}$ of the set of vertices of $X$, the subgraph \emph{spanned} by $S$ is the complete subgraph of $X$ with vertex set $S$.
\end{defn}

\begin{defn}[Join, link, star]\label{defn:join_link_star}
Given disjoint simplices $\Delta,\Delta'$ of $X$, we let $\Delta\star\Delta'$ denote the simplex spanned by $\Delta^{(0)}\cup\Delta'^{(0)}$, if it exists. For each simplex $\Delta$, the \emph{link} $\link(\Delta)$ is the union of 
all simplices $\Sigma$ of $X$ such that $\Sigma\cap\Delta=\emptyset$ and $\Sigma\star\Delta$ is a simplex of $X$.  Observe that $\link(\Delta)=\emptyset$ if and only if $\Delta$ is a maximal simplex. The \emph{star} of $\Delta$ is $\operatorname{Star}(\Delta)\coloneq \link(\Delta)\star\Delta$, i.e. the union of all simplices of $X$ that contain $\Delta$.
\end{defn}

\begin{defn}[$X$--graph, $\W$--augmented graph]\label{defn:X_graph}
An \emph{$X$--graph} is a graph $\W$ whose vertex set is the set of all maximal simplices of $X$.

For a simplicial graph $X$ and an $X$--graph $\W$, the \emph{$\W$--augmented graph} $\duaug{X}{\W}$ is the graph defined as follows:
\begin{itemize}
     \item the $0$--skeleton of $\duaug{X}{\W}$ is $X^{(0)}$;
     \item if $v,w\in X^{(0)}$ are adjacent in $X$, then they are adjacent in $\duaug{X}{\W}$; 	
     \item if two vertices in $\W$ are adjacent, then we consider $\sigma,\rho$, the associated maximal simplices of $X$, and in $\duaug{X}{\W}$ we connect each vertex of $\sigma$ to each vertex of $\rho$.
\end{itemize}
We equip $\W$ with the usual path-metric, in which each edge has unit length, and do the same for $\duaug{X}{\W}$.
\end{defn}

\begin{defn}[Equivalence between simplices, saturation]\label{defn:simplex_equivalence}
For $\Delta,\Delta'$ simplices of $X$, we write $\Delta\sim\Delta'$ to mean $\link(\Delta)=\link(\Delta')$. We denote the $\sim$--equivalence class of $\Delta$ by $[\Delta]$.
Let $\Sat(\Delta)$ denote the set of vertices $v\in X$ for which there exists a simplex $\Delta'$ of $X$ such that $v\in\Delta'$ and $\Delta'\sim\Delta$, i.e.
$$\Sat(\Delta)=\left(\bigcup_{\Delta'\in[\Delta]}\Delta'\right)^{(0)}.$$
We denote by $\frakS$ the set of $\sim$--classes of non-maximal simplices in $X$.
\end{defn}

\begin{defn}[Complement, link subgraph]\label{defn:complement}
Let $W$ be an $X$--graph.  For each simplex $\Delta$ of $X$, let $\C(\Delta)$ be the induced subgraph of $\duaug{X}{W}$ spanned by $\link(\Delta)^{(0)}$, which we call the \emph{augmented link} of $\Delta$. Also, let $Y_\Delta$ be the subgraph of $\duaug{X}{W}$ induced by the set of vertices $X^{(0)}-\Sat(\Delta)$.
\end{defn}

Note that $\C  (\Delta)=\C  (\Delta')$ whenever $\Delta\sim\Delta'$. (We emphasise that we are taking links in $X$, not in $\duaug{X}{\W}$, and then considering the subgraphs of $Y_\Delta$ induced by those links.)

\begin{defn}[Combinatorial HHS]\label{defn:combinatorial_HHS}
A \emph{combinatorial HHS} $(X,\W)$ consists of a simplicial graph $X$ and an $X$--graph $\W$ satisfying the following conditions:
\begin{enumerate}
    \item \label{item:chhs_flag} There exists $n\in\mathbb N$ such that any chain $\link(\Delta_1)\subsetneq\dots\subsetneq\link(\Delta_i)$, where each $\Delta_j$ is a simplex of $X$, has length at most $n$;
    \item \label{item:chhs_delta} There is a constant $\delta$ so that for each non-maximal simplex $\Delta$, the subgraph $\C  (\Delta)$ is $\delta$--hyperbolic and $(\delta,\delta)$--quasi-isometrically embedded in $Y_\Delta$, where $Y_\Delta$ is as in Definition~\ref{defn:complement};
    \item \label{item:chhs_join} Whenever $\Delta$ and $\Sigma$ are non-maximal simplices for which there exists a non-maximal simplex $\Gamma$ such that $\link(\Gamma)\subseteq\link(\Delta)\cap \link(\Sigma)$, and $\diam(\C   (\Gamma))\geq \delta$, then there exists a simplex $\Pi$ which extends $\Sigma$ such that $\link(\Pi)\subseteq \link(\Delta)$, and all $\Gamma$ as above satisfy $\link(\Gamma)\subseteq\link(\Pi)$;
    \item \label{item:C_0=C} If $v,w$ are distinct non-adjacent vertices of $\link(\Delta)$, for some simplex $\Delta$ of $X$, contained in $\W$-adjacent maximal simplices, then they are contained in $\W$-adjacent simplices of the form $\Delta\star\Sigma$.
\end{enumerate}
\end{defn}

\begin{defn}[Nesting, orthogonality, transversality]\label{defn:nest_orth}
Let $X$ be a simplicial graph.  Let $\Delta,\Delta'$ be non-maximal simplices of $X$.  Then:
\begin{itemize}
     \item $[\Delta]\nest[\Delta']$ if $\link(\Delta)\subseteq\link(\Delta')$;
     \item $[\Delta]\orth[\Delta']$ if $\link(\Delta')\subseteq \link(\link(\Delta))$.
\end{itemize}
If $[\Delta]$ and $[\Delta']$ are neither $\orth$--related nor $\nest$--related, we write 
$[\Delta]\transverse[\Delta']$.
\end{defn}

\begin{defn}[Projections]\label{defn:projections}
Let $(X,\W)$ be a combinatorial HHS.  

Fix $[\Delta]\in\frakS$ and define a map $\pi_{[\Delta]}:\W\to 2^{\C  (\Delta)}$ as follows. Let
$$p:Y_\Delta\to2^{\C  (\Delta)}$$
be the coarse closest point projection, i.e. 
$$p(x)=\{y\in\C  (\Delta):\dist_{Y_\Delta}(x,y)\le\dist_{Y_\Delta}(x,\C  (\Delta))+1\}.$$

Suppose that $w$ is a vertex of $\W$, so $w$ corresponds to a unique simplex $\Sigma_w$ of $X$. Now, \cite[Lemma 1.15]{BHMS} states that the intersection $\Sigma_w\cap Y_\Delta$ is non-empty and has diameter at most $1$. Define $$\pi_{[\Delta]}(w)=p(\Sigma_w\cap Y_\Delta).$$

We have thus defined  $\pi_{[\Delta]}:\W^{(0)}\to 2^{\C  (\Delta)}$. If $v,w\in \W$ are joined by an edge $e$ of $\W$, then $\Sigma_v,\Sigma_w$ are joined by edges in $\duaug{X}{\W}$, and we let
$$\pi_{[\Delta]}(e)=\pi_{[\Delta]}(v)\cup\pi_{[\Delta]}(w).$$

Now let $[\Delta],[\Delta']\in\frakS$ satisfy $[\Delta]\transverse[\Delta']$ or $[\Delta']\propnest [\Delta]$. Let $$\rho^{[\Delta']}_{[\Delta]}=p(\Sat(\Delta')\cap Y_\Delta),$$ where $p:Y_\Delta\to\C  (\Delta)$ is coarse closest-point projection. 

Let $[\Delta]\propnest [\Delta']$. Let $\rho^{[\Delta']}_{[\Delta]}:\C  (\Delta')\to \C  (\Delta)$ be defined as follows.  On $\C  (\Delta')\cap Y_\Delta$, it is the restriction of $p$ to $\C  (\Delta')\cap Y_\Delta$. Otherwise, it takes the value $\emptyset$.
\end{defn}

We are finally ready to state the main theorem of \cite{BHMS}:
\begin{thm}[HHS structures for $X$--graphs]\label{thm:hhs_links}
Let $(X,\W)$ be a combinatorial HHS. The above data defines a \emph{hierarchically hyperbolic space} (HHS) structure on $\W$, in the sense of \cite{HHS_II}.

Moreover, let $G$ be a group acting on $X$ with finitely many orbits of subcomplexes of the form $\link(\Delta)$, where $\Delta$ is a simplex of $X$. Suppose moreover that the action on maximal simplices of $X$ extends to an action on $\W$, which is metrically proper and cobounded. Then $G$ is a \emph{hierarchically hyperbolic group} (HHG), in the sense of \cite{HHS_II}.
\end{thm}

\begin{defn}\label{def:combHHG}
We will say that a group $G$ satisfying the assumptions of Theorem~\ref{thm:hhs_links} is a \emph{combinatorial} HHG. The hierarchically hyperbolic structure comes with a constant $E$, called the \emph{HHS constant}, which we can assume to be larger than both $n$ and $\delta$.
\end{defn}

\section{Definition and results on short HHG}\label{sec:shortHHG}
This Section recaps the definition and properties of short HHG, as introduced in \cite{Short_HHG:I}. 

\begin{defn}[Blowup graph]\label{defn:blowup} Let $\ov X$ be a simplicial graph, whose vertices are labelled by graphs $\{L_v\}_{v\in\ov{X}^{(0)}}$.
The \emph{blowup} of $\ov{X}$, with respect to the collection $\{L_v\}$, is the graph $X$ obtained from $\ov{X}$ by replacing every vertex $v$ with the \emph{cone} $\Squid(v)=v * (L_v)^{(0)}$. Two cones $\Squid(v)$ and $\Squid(w)$ span a join in $X$ if and only if $v,w$ are adjacent in $\ov{X}$, and are disjoint otherwise. In particular, there is a Lipschitz retraction $p\colon X\to \ov{X}$ mapping every $\Squid(v)$ to its tip $v$.
\end{defn}

\begin{figure}[htp]
    \centering
    \includegraphics[width=\textwidth, alt={The cones of two adjacent vertices span a join}]{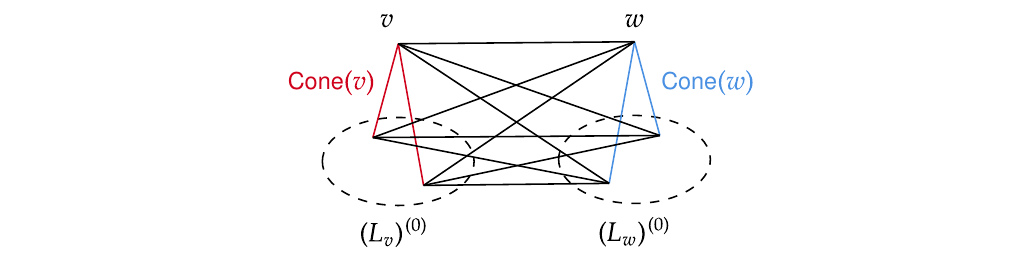}
    \caption{The blowup of two adjacent vertices of $\ov{X}$.}
    \label{fig:blowup}
\end{figure}

For every simplex $\Delta\subset X$, let $\ov\Delta=p(\Delta)$, which we call the \emph{support} of $\Delta$, and for every $v\in\ov\Delta$ let $\Delta_v=\Delta\cap \Squid(v)$. When describing a simplex $\Delta$, we shall put vertices belonging to the same $\Delta_v$ in parentheses: for example, if the vertices of $\Delta$ are $\{v,w,x\}$, where $v,w\in\ov X^{(0)}$ and $x\in (L_v)^{(0)}$, then we denote $\Delta$ by $\{(v,x),(w)\}$. Then, by inspection of the definition, one gets the following:
\begin{lemma}\label{cor:bounded_links}
Suppose that $\ov{X}$ is triangle-free, and that no component of $\ov{X}$ is a single point. Then, given a simplex $\Delta$ of $X$, one of the following holds:
\begin{enumerate}
\item \label{cor:bounded_links_emptyset} $\Delta=\emptyset$, and $\link_X(\Delta)=X$;
 \item \label{cor:bounded_links_edge} (Edge-type) $\Delta=\{(v,x)\}$, where $v\in \ov{X}^{(0)}$ and $x\in L_v$, and $\link_X(\Delta)=p^{-1}\link_{\ov X}(v)$;
 \item \label{cor:bounded_links_almostmax} (Triangle-type)  $\Delta=\{(v,x),(w)\}$, where $v,w\in \ov{X}^{(0)}$ are $\ov X$-adjacent and $x\in L_v$, and $\link_X(\Delta)=(L_w)^{(0)}$;
 \item \label{cor:bounded_links_maximal} $\Delta=\left\{(v,x), (w,y)\right\}$ is a maximal simplex, where $v,w\in \ov{X}^{(0)}$ are adjacent, $x\in L_v$ and $y\in L_w$, and $\link_X(\Delta)=\emptyset$. 
  \item \label{cor:bounded_links_bounded} $\link_X(\Delta)$ is a vertex, or a non-trivial join. In particular $\diam(\link_X(\Delta))\le 2$.
\end{enumerate}
\end{lemma}

\begin{figure}[htp]
    \centering
    \includegraphics[width=\textwidth, alt={A simplex of edge type connects a vertex with a point in its cone. A simplex of triangle type is a triangle whose link is the base of a cone.}]{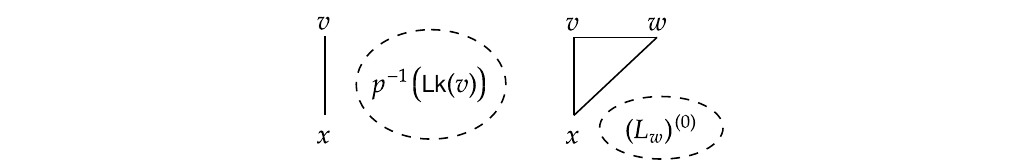}
    \caption{A simplex of edge-type (on the left) and a simplex of triangle-type (on the right). Their links are represented by the dashed areas.}
    \label{fig:links}
\end{figure}

\subsection{Definition}\label{defn:short_HHG}
Let $G$ be a combinatorial HHG, whose structure comes from the action on the combinatorial HHS $(X,\W)$. We say that $G$ is \emph{short} if it satisfies Axioms \eqref{short_axiom:graph}-\eqref{short_axiom:extension}-\eqref{short_axiom:cobounded} below.
\begin{ax}[Underlying graph]\label{short_axiom:graph} $X$ is obtained as a blowup of some graph $\ov{X}$, which is triangle- and square-free and such that no connected component of $\ov X$ is a point. Moreover, $\ov X$ is a $G$-invariant subgraph of $X$. 
\end{ax}

The above Axiom implies that the $G$-action on $X$ restricts to a $G$-action on $\ov X$ with finitely many $G$-orbits of edges.

\begin{ax}[Vertex stabilisers are cyclic-by-hyperbolic]\label{short_axiom:extension}
    For every $v\in\ov{X}^{(0)}$ there is an extension
        $$\begin{tikzcd}
        0\ar{r}&Z_v\ar{r}&\Stab{G}{v}\ar{r}{\p_v}&H_v\ar{r}&0
        \end{tikzcd}$$
        where $H_v$ is a hyperbolic group and $Z_v$ is a cyclic, normal subgroup of $\Stab{G}{v}$ which acts trivially on $\link_{\ov{X}}(v)$. We call $Z_v$ the \emph{cyclic direction} associated to $v$.
        
        Moreover, one requires that the family of such extensions is equivariant with respect to the $G$-action by conjugation; in particular, $Z_{gv}=gZ_vg^{-1}$ for every $v\in \ov{X}^{(0)}$ and $g\in G$.
\end{ax}

\begin{notation}\label{notation:lv_Uv}
    For every $v\in \ov{X}^{(0)}$, let $\ell_v$ be the domain associated to any triangle-type simplex whose link is $(L_v)^{(0)}$, and let $\U_v$ be the domain associated to any edge-type simplex supported on $v$.
\end{notation}

\begin{ax}[(Co)bounded actions]\label{short_axiom:cobounded}
For every $v\in \ov{X}^{(0)}$, the cyclic direction $Z_v$ acts geometrically on $\C\ell_v$ and with uniformly bounded orbits on $\C\ell_w$ for every $w\in\link_{\ov X}(v)$. In particular, $\C\ell_v$ is a quasiline if $Z_v$ is infinite cyclic, and uniformly bounded otherwise.
\end{ax}

We will denote a short HHG, together with its short structure, by $(G, \ov X, \W)$. The reader can find examples of short HHG in \cite[Sections 2.3 and 6]{Short_HHG:I}, among which figure the mapping class group of a five-punctured sphere, Artin groups of large and hyperbolic type, non-geometric graph manifold groups, and RAAGs on triangle- and square-free graphs.

\begin{defn}\label{defn:colourable}
    A short HHG $(G, \ov X, \W)$ is \emph{colourable} if there exists a partition of the vertices of $\ov{X}$ into finitely many colours, such that no two adjacent vertices share the same colour, and the $G$-action on $\ov X$ descends to an action on the set of colours.
\end{defn}
It is easily seen that the above property coincides with the more general notion of \emph{colourability} of a HHG, which requires the existence of a finite, $G$-invariant colouring \emph{of the whole domain set} (see e.g. \cite[Definition 4.15]{barak2024equational}).

\subsection{Properties}

\begin{rem}\label{rem:unbounded_dom_short_hhg}
    As a consequence of Lemma~\ref{cor:bounded_links}, if a simplex $\Delta\subseteq X$ is such that the associated coordinate space $\C (\Delta)$ is unbounded, then $$[\Delta]\in\{S\}\cup\{\ell_v\}_{v\in\ov{X}^{(0)}}\cup \{\U_v\}_{v\in\ov{X}^{(0)},\,|\link_{\ov X}(v)|=\infty}$$
    where $S=[\emptyset]$ is the $\nest$-maximal domain, and $\ell_v$, $\U_v$ are defined as in Notation~\ref{notation:lv_Uv}.
    
    Furthermore, by how nesting and orthogonality are defined in a combinatorial HHS (Definition~\ref{defn:nest_orth}), we see that:
    \begin{itemize}
        \item $\ell_v\orth \ell_w$ whenever $v\neq w$ are adjacent in $\ov X$, and are transverse otherwise; 
        \item $\ell_v\orth \U_v$; 
        \item $\ell_v\nest \U_w$ whenever $v\neq w$ are adjacent in $\ov X$;
        \item If $v$ has valence greater than one in $\ov X$, and $\dist_{\ov X}(v,w)\ge 2$, then $\U_v\transverse \ell_w$.
        \item If both $v$ and $w$ have valence greater than one in $\ov X$, and $\dist_{\ov X}(v,w)\ge 2$, then $\U_v\transverse \U_w$.
    \end{itemize}
    We avoided describing the slightly more complicated relations involving $\U_v$ when $v$ has valence one in $\ov X$, as we shall not need them.
\end{rem}

\begin{rem}\label{rem:coord_spaces_wrt_lines}
 The main coordinate space $\C S=X^{+\W}$ $G$-equivariantly retracts onto the \emph{augmented support graph} $\ov X^{+\W}$, obtained from $\ov{X}$ by adding an edge between $v$ and $w$ if they belong to $\W$-adjacent maximal simplices of $X$. In other words, $\C S$ is $G$-equivariantly quasi-isometric to a graph with vertex set $\ov X^{(0)}$, which contains $\ov X$ as a (non-full) $G$-invariant subgraph.
 
Similarly, for every $v\in \ov{X}^{(0)}$, $\C \U_v$ is $\Stab{G}{v}$-equivariantly quasi-isometric to the \emph{augmented link} $\link_{\ov{X}}(v)^{+\W}$, on which $Z_v$ acts trivially.
  \end{rem}

By \cite[Lemmas 2.10 and 2.11]{Short_HHG:I}, which are special cases of a more general phenomenon for combinatorial HHG, a short HHG satisfies several strengthened versions of the bounded geodesic image axiom, which we recall here.
\begin{notation}\label{notation:rho}
    Set $\rho^w_{w'}\coloneq \rho^{\ell_w}_{\ell_w'}$, which is defined as $\rho^{[\Delta]}_{[\Delta']}$ for any two simplices of triangle-type $\Delta=\left\{(v,x), (w)\right\}$ and $\Delta'=\left\{(v',x'), (w')\right\}$. This is a subset of diameter bounded by the HHS constant $E$. 
\end{notation}

\begin{lemma}[Strong BGI for projections on quasilines]\label{lem:strong_bgi_for_short}
    Whenever $u,v,w\in \ov X^{(0)}$, if both $\rho^u_w$ and $\rho^v_w$ are defined and at least $2E$-apart in $\C \ell_w$, then every geodesic $[u,v]\subset \ov X^{+\W}$ must pass through $\operatorname{Star}_{\ov{X}}(w)$.

    Similarly, whenever $u,v,w\in\link_{\ov{X}}(z)$, if both $\rho^u_w$ and $\rho^v_w$ are defined and at least $2E$-apart in $\C\ell_w$, then every geodesic $[u,v]\subset \link_{\ov{X}}(z)^{+\W}$ must pass through $w$.
\end{lemma}

\begin{notation}
    For every $u,v,w\in\ov X^{(0)}$ such that $u\neq w$ and $v\neq v$, if $w$ has valence greater than one in $\ov X$ set 
$$\dist_{\link_{\ov X}(w)^{+\W}}(u, v)=\dist_{\link_{\ov X}(w)^{+\W}}\left( p(\rho^{\ell_u}_{\U_w}), p(\rho^{\ell_v}_{\U_w})\right),$$
where $p\colon X\to \ov X$ is the retraction. If instead $w$ has valence one then $\link_{\ov X}(w)$ is a point, and we set $\dist_{\link_{\ov X}(w)^{+\W}}(u, v)=0$.
\end{notation}

\begin{lemma}[Strong BGI for projections on links]\label{lem:strong_bgi_for_short_edge_type}
    Let $w\in\ov X^{(0)}$. For every $u,v\in \ov X^{(0)}-\{w\}$, if $\dist_{\link_{\ov X}(w)^{+\W}}(u, v)\ge 2E$, then every geodesic $[u,v]\subset \ov X^{+\W}$ must pass through $w$.
\end{lemma}

Finally, recall that, given a hyperbolic group $H$ and an infinite cyclic subgroup $Q\le H$, there exists a maximal virtually cyclic subgroup $K_{H}(Q)$ containing $Q$.
\begin{lemma}[{\cite[Lemma 2.17]{Short_HHG:I}}]\label{lem:Hv_hyp_1}
       Let $(G, \ov X, \W)$ be a short HHG, $v\in \ov{X}^{(0)}$, and let $W$ be a collection of $\Stab{G}{v}$-orbit representatives of vertices in $\link_{\ov X}(v)$ with infinite cyclic directions. Then $H_v$ is hyperbolic relative to $\{K_{H_v}(\p_v(Z_w))\}_{w\in W}$. 
\end{lemma}

\subsection{Blowup materials}\label{subsec:squid}
The main technical result of \cite{Short_HHG:I} is that short HHG can be equivalently defined as those groups equipped with certain data, called \emph{blowup materials}. We first set up some notation.
\begin{defn}
    A finitely generated group $G$ is \emph{weakly hyperbolic} relative to a collection of conjugacy classes of subgroups if, given any finite generating set $S$ for $G$, and any choice of a subgroup $H_i$ in each conjugacy class, the coned-off Cayley graph $\Cay{G}{S\bigcup_i H_i}$ is hyperbolic.
\end{defn}
\begin{notation}[Product regions]\label{notation:pre_def_squid}
In what follows, $G$ is a group acting on a simplicial graph $\ov X$ with finitely many orbits of edges. Fix $V=\{v_1,\ldots, v_k\}$ a set of representatives of the $G$-orbits of vertices in $\ov{X}$. For every $v_i\in V$ and every $g\in G$, set $P_{gv_i}=g\Stab{G}{v_i}$, which we call the \emph{product region} associated to $gv_i$ (with respect to this choice of orbit representatives). We always see $P_v$ as a metric subspace of $G$, where the latter is equipped with the word metric with respect to any finite generating set. 
\end{notation}

The following definition is equivalent to \cite[Definition 3.9]{Short_HHG:I}, which we rephrase for improved clarity.
\begin{defn}[Blowup materials]\label{defn:squid_material}
Let $G$ be a finitely generated group, acting on a simplicial graph $\ov X$ with finitely many orbits of edges. For every $v\in \ov X^{(0)}$ let $Z_v\unlhd \Stab{G}{v}$ be a cyclic, normal subgroup; furthermore, whenever $Z_v$ is infinite, let $\phi_v\colon E_v\to \R$ be a homogeneous quasimorphism, where $E_v\unlhd \Stab{G}{v}$ is some finite-index, normal subgroup centralising $Z_v\cap E_v$. Let $V$ be a set of representatives for the $G$-orbits of vertices in $\ov X$, and for every $v\in \ov X^{(0)}$ let $P_v$ be the associated product region with respect to $V$, as in Notation~\ref{notation:pre_def_squid}.  

The above data defines \emph{blowup materials} on $G$ if the following conditions hold:
    \begin{enumerate}[start=0]
        \item\label{squid_material:equivariance} All subgroups and quasimorphisms are equivariant with respect to the $G$-action by conjugation.
        \item\label{squid_material:graph} $\ov{X}$ is triangle- and square-free, and no connected component of $\ov{X}$ is a single point.
        \item\label{squid_material:extensions}
        For every $v\in \ov X^{(0)}$, the quotient $H_v\coloneq \Stab{G}{v}/Z_v$ is hyperbolic.
        \item\label{squid_material:edge_stab} Whenever $e=\{v,w\}$ is an edge of $\ov{X}$, $Z_v$ fixes $w$, and $\Stab{G}{v}\cap\Stab{G}{w}$ contains $\langle Z_v, Z_w\rangle $ as a subgroup of finite index. Moreover, $\p_v(Z_w)$ is quasiconvex in $H_v$.
        \item\label{squid_material:big_papa} $G$ is weakly hyperbolic relative to $\{\Stab{G}{v}\}_{v\in \ov X}$.
        \item\label{squid_material:quasimorphisms} For every $v\in \ov X^{(0)}$ with $Z_v$ infinite, the quasimorphism $\phi_v$ is unbounded on $Z_{v}\cap E_{v}$ and trivial on $Z_w\cap E_{v}$ for every vertex $w\in\link_{\ov X}(v)$. 
        \item\label{squid_material:gates}
        There exist a constant $B\ge 0$ and, for every $v_i\in V$, a $B$-coarsely Lipschitz, $B$-coarse retraction $\gate_{v_i}\colon G\to 2^{E_{v_i}}$,
        with the following properties:
        \begin{enumerate}[label=(\roman*)]
            \item If $w\in\link_{\ov X}(v_i)$ then $\gate_{v_i}(P_w)\subseteq N_B(P_w)$.
            \item For every $i$ there exists a finite collection $\{w_i^1,\ldots, w_i^r\}\subseteq\link_{\ov X}(v_i)$ such that, whenever $\dist_{\ov X}(v_i, u)\ge 2$,  there exist $g\in\Stab{G}{v_i}$ and $1\le j\le r$ such that $\gate_{v_i}(P_u)\subseteq N_B\left(gZ_{w_i^j}\right)$.
        \end{enumerate} 
    \end{enumerate}
    We call $\ov X$ the \emph{support graph}, each $Z_v$ the \emph{cyclic direction} associated to $v$, and each $\gate_{v_i}$ the \emph{gate} onto $E_{v_i}$.
\end{defn}

\begin{thm}\label{thm:squidification}(\cite[Theorem 3.10]{Short_HHG:I})
    Let $G$ be a finitely generate group admitting blowup materials, with support graph $\ov X$. Then $G$ admits a short HHG structure, where $\ov X$ is the support graph and the cyclic direction associated to each $v\in\ov{X}^{(0)}$ is (a finite-index subgroup of) $Z_v$.
\end{thm}

\begin{prop}
\label{prop:short_is_squid}
(\cite[Proposition 4.1]{Short_HHG:I})
    A short HHG $(G, \ov X, \W)$ admits blowup materials, whose support graph is $\ov X$ and whose extensions $Z_v\to\Stab{G}{v}\to H_v$ are those from Axiom~\ref{short_axiom:extension}.
\end{prop}

In particular, we get:
\begin{cor}[Edge groups]
    Let $(G, \ov X, \W)$ be a short HHG. Whenever $\{v,w\}$ is an edge of $\ov X$, the \emph{edge group} $\Stab{G}{v}\cap\Stab{G}{w}$ contains $\langle Z_v, Z_w\rangle$ as a subgroup of finite index.
\end{cor}

\subsection{Subgroups generated by large cyclic directions}\label{sec:large_dt}

\begin{notation}[Kernel of the projection]\label{notation:kernel}
Let $(G, \ov{X},\W)$ be a colourable short HHG. Let $\B=\{s_1,\ldots,s_r\}\subset\ov{X}^{(0)}$ be a (possibly non-maximal) collection of vertices belonging to pairwise different $G$-orbits. For every $i=1,\ldots,r$ choose a non-zero natural number $M_i\in \mathbb{N}-\{0\}$, and for every $v\in G\{s_i\}$ let $\Gamma_v=M_i Z_{v}$.

From now on, we will focus on all normal subgroups of the form
$$\N\coloneq\langle\langle M_1Z_{s_1},\ldots, M_rZ_{s_r}\rangle\rangle=\langle \Gamma_v\rangle_{v\in G\B}.$$
$\B$ will be called the \emph{base} of $\N$.
\end{notation}

\begin{defn}[Deep enough quotient]\label{defn:deep_enough}
Let $\N=\langle\langle M_1Z_{s_1},\ldots, M_rZ_{s_r}\rangle\rangle$ be as in Notation~\ref{notation:kernel}. We will say that a property of $\N$, or of $G/\N$, holds \emph{if $\N$ is deep enough} if there exists $D\in \mathbb{N}-\{0\}$ such that the property holds whenever every $M_i$ is a multiple of $D$. 
\end{defn}

\begin{rem}\label{rem:initial_D}
    Notice that, if $\N$ is deep enough, one can always assume that:
    \begin{itemize}
        \item Every finite $\Gamma_v$ is trivial, or in other words $\B$ does not contain vertices with finite cyclic direction;
        \item Whenever $v,w$ are $\ov{X}$-adjacent, $Z_v$ commutes with $\Gamma_w$ (this is because the centraliser of $Z_v$ in $\Stab{G}{v}$ has index at most two);
        \item $\N$ lies in the normal subgroup $G_0$ of $G$ which acts trivially on the set of colours of $\ov X$.
    \end{itemize} 
\end{rem}

\section{Rotating families and projection graphs, revisited}\label{sec:rotating}
In this Section, we prove that the machinery of \emph{composite projection systems}, devised in \cite{dahmani:rotating} and then developed further in \cite{dfdt}, can be adapted to study quotients of colourable short HHG by powers of cyclic directions. This allows us to describe the quotient graphs $\ov X/\N$ and $\ov X^{+\W}/\N$, and the vertex stabilisers for the induced $G/\N$-action on $\ov X/\N$. The reader who is only interested in short HHG is advised to skip to the consequences which are gathered in Subsection~\ref{subsec:rotating_for_short}. 

\subsection{Strong composite projection graphs}\label{subsec:SCPG}
We first recall some definitions from \cite{dahmani:rotating}.

\begin{defn}[{\cite[Definition 1.2]{dahmani:rotating}}]\label{def;CPS}
Let $\mathbb{Y}_*$ be the disjoint union of finitely many countable sets $\mathbb{Y}_1, \dots, \mathbb{Y}_m$, called \emph{colours}, and for every $y\in \mathbb{Y}_*$ let $i(y)$ be the index such that $y\in \mathbb{Y}_{i(y)}$. A \emph{composite projection system} on $\mathbb{Y}_*$ consists of
\begin{itemize} 
\item a constant $\theta\ge 0$;
    \item a family of subsets  $\Act(y)\subset   \mathbb{Y}_*$ for $ y\in\mathbb{Y}_*$  (the \emph{active set} for $y$) such that $\mathbb{Y}_{i(y)} \subset\Act(y)$, and such that  $x\in \Act(y)$ if and only if $y\in \Act(x)$  (\emph{symmetry in action}), 
    \item and a family of functions $\dist^\pi_y : ( \Act(y)- \{y\} )^2 \to \R_+$, satisfying: 
    \begin{itemize} 
        \item \textbf{Symmetry}: $\dist^\pi_y(x,z)=\dist^\pi_y(z,x)$ for $x,z\in\Act(y)-\{y\}$;
        \item \textbf{Triangle inequality}: $\dist^\pi_y(w,x)\leq \dist^\pi_y(w,z)+\dist^\pi_y(z,x)$ for all $w,x,z\in\Act(y)-\{y\}$;
        \item \textbf{Behrstock inequality:} $\min\{\dist^\pi_y(x,z),\dist^\pi_z(x,y)\}\leq\theta$ whenever both quantities are defined; 
        \item \textbf{Properness:} $|\{y \in \mathbb{Y}_i, \dist^\pi_y(x,z)\geq \theta \}|<\infty$ for all $x,z\in \mathbb{Y}_i$;
			\item  \textbf{Separation:} $\dist^\pi_y(z,z)<\theta$ for $z\in\Act(y)-\{y\}$;
			\item  \textbf{Closeness in inaction:} if   $x\notin \Act(z)$ then, for
			all $y \in \Act(x) \cap \Act(z)$, we have $ \dist_y^\pi(x, z)\leq \theta$; 
			\item \textbf{Finite filling:}  for all $\mathcal{Z}\subset \mathbb{Y}_*$, there
			is a finite collection $x_j\in\mathcal{Z}$ such that $\cup_j
			\Act(x_j)$ covers  $\cup_{ x\in \mathcal{Z}} \Act(x)$.
		\end{itemize}
	\end{itemize}
\end{defn}

We will also require the following “uniform” version of the properness axiom:
\begin{defn}\label{defn:uniform_properness_CPS}
    A composite projection system on $\mathbb{Y}_*$ is \emph{uniformly proper} if $|\{y \in \mathbb{Y}_*, \dist^\pi_y(x,z)\geq \theta \}|<\infty$ for all $x,z\in \mathbb{Y}_*$.
\end{defn}

\begin{defn}\label{defn:group_action_on_CPS}
    A group $G$ acts on a composite projection system $\mathbb{Y}_*$ by isomorphisms if:
    \begin{itemize}
        \item $G$ acts on $\mathbb{Y}_*$, and the action permutes the colours;
        \item for all $g\in G$ and $y\in\mathbb{Y}_*$, $\Act(gy)=g\Act(y)$;
        \item for all $g\in G$ and $x,y,z\in \mathbb{Y}_*$ such that $\dist^\pi_y(x,z)$ is defined, $\dist^\pi_{gy}(gx,gz)=\dist^\pi_y(x,z)$.
    \end{itemize}  
\end{defn}

\begin{rem}\label{rem:d^angle} We shall follow the constructions and arguments from \cite{dfdt} as closely as possible, in order to recover their results without altering most of the proofs. To this extent, arguing as in the paragraph below \cite[Definition 1.1]{dfdt}, we first replace each distance $\dist^\pi_y$ by a modified function $\dist_y^\angle\colon (\Act(y) -\{y\})^2 \to \mathbb{R}_+$, constructed as in \cite[Theorem 3.3]{BBF}. The new functions $\dist_y^\angle$ differ from $\dist^\pi$ by a uniformly bounded amount depending only on $\theta$, and they satisfy almost all properties from Definitions~\ref{def;CPS} and~\ref{defn:uniform_properness_CPS} with $\theta$ replaced by $\Theta=4\theta+1$; the only difference is that the triangle inequality is ``coarsified'', meaning that there exists a constant $\kappa\ge 0$ such that, for all $w,x,z\in\Act(y)-\{y\}$, $\dist^\angle_y(w,x)\le \dist^\angle_y(w,z)+\dist^\angle_y(z,x)+\kappa$. Furthermore, if $G$ acts on $(\mathbb{Y}_*,\dist^\pi_*)$ by isomorphisms, then the new functions are $G$-equivariant, meaning that for all $g\in G$ and $x,y,z\in \mathbb{Y}_*$ such that $\dist^\angle_y(x,z)$ is defined, $\dist^\angle_{gy}(gx,gz)=\dist^\angle_y(x,z)$.
    \end{rem}

\begin{defn}(Composite rotating family)\label{def;CRF}
	Consider a composite projection system $(\mathbb{Y}_*, \theta,\dist^\pi_*)$, endowed with an action of a group $G$ by isomorphisms. Let $\Theta=4\theta+1$, and let $\dist^\angle_*$ be the modified functions from Remark~\ref{rem:d^angle}. 
    
    A \emph{composite rotating family} on  $(\mathbb{Y}_*,G)$, with \emph{rotation control} $\Theta_{rot}>0$, is a family of subgroups $\left\{\Gamma_v\right\}_{v\in \mathbb{Y}_*}$ such that the following holds: 
	\begin{itemize}
		\item For all $x\in \mathbb{Y}_*$, $\Gamma_x <  \Stab{G}{x}$ is an infinite group;
		\item whenever $y=x$ or $y\not\in\Act(x)$, the subgroup $\Gamma_x$ fixes $y$ and $\dist^\angle_y$;
		\item for all $g\in G$, and all $x\in \mathbb{Y}_*$, we have $\Gamma_{gx}= g\Gamma_x g^{-1}$;
		\item if $x\notin \Act(z)$ then $\Gamma_x$ and $\Gamma_z$ commute;
        \item $\Gamma_x$ acts with \textbf{proper isotropy}, meaning that for all $R>0$ and $y \in \Act(x)$, the set  $\{\gamma \in \Gamma_x, \dist_x^\angle (y, \gamma y) <R\}$ is finite;
		\item for all $i\leq m$ and for all $x,y, z \in \mathbb{Y}_i$, if $\dist^\angle_y(x,z)\leq \Theta$, then $$\dist^\angle_y(x,gz) \geq \Theta_{rot}$$ for all
		$g\in \Gamma_y-\{1\}$.
	\end{itemize}
    Set $\N$ as the normal subgroup of $G$ generated by $\bigcup_{v\in \mathbb{Y}_*}\Gamma_v$.
\end{defn}

\begin{defn}[SCPG]\label{dfn:scpg} Let $G$ be a finitely generated group. A \emph{strong} $G$-composite projection graph is the data of:
\begin{itemize}
    \item A hyperbolic graph $\calS$ on which $G$ acts by simplicial automorphisms, with finitely many orbits of vertices;
    \item A structure of a uniformly proper composite projection system on $\calS^{(0)}$, with constant $\theta$, on which the induced action of $G$ is by isomorphisms;
    \item A $G$-invariant subset $\mathbb{Y}_*\subset \calS^{(0)}$, which inherits the structure of a composite projection system;
    \item A composite rotating family $\{\Gamma_v\}_{v\in\mathbb{Y}_*}$ with rotation control $\Theta_{rot}$.
    \item \emph{Strong Bounded Geodesic Image}: a constant $C\in\mathbb{R}_+$ so that the following holds. For each $x, y, s \in \mathcal S$ so that $\dist^\angle_s(x, y)$ is defined and larger than $C$, every geodesic $[x, y]\subset \calS$ contains a vertex $w$ such that $\dist_{\calS}(w,s)\le C$, and either $w=s$ or $w\not\in \Act(s)$. In particular, $w$ is fixed by $\Gamma_s$. 
\end{itemize}
We shall denote a SCPG by the tuple $(\mathcal S, G, \mathbb{Y}_*, \{\Gamma_v\})$.
\end{defn}

\begin{rem}[Comparison with \cite{dfdt}]
The above is a slight variant of the main Definition in \cite[Section 2]{dfdt}. There are three notable differences:
\begin{itemize}
    \item In \cite{dfdt} the authors require that, for every $s\in \calS^{(0)}$, $\Gamma_s$ fixes $\link_{\mathcal S}(s)$, which is in general not true for short HHG (see Remark~\ref{rem:why_not_dfdt}). 
    \item The version of the bounded geodesic image appearing in \cite{dfdt} asks that, whenever $x, y$ project far to $s$, on any $\calS$-geodesic $[x, y]$ there exists a vertex $w$ which is $\calS$-adjacent to $s$. However, what is actually needed to run the arguments from \cite{dfdt} is the existence of some $w\in [x,y]$ which is fixed by $\Gamma_s$ and lies within uniform $\calS$-distance from $s$. Both these requirements are part of our strong bounded geodesic image assumption.
    \item Though the definition in \cite{dfdt} involves a sub-projection complex $\mathbb{Y}^\tau_*$ on which the composite rotating family is defined, all proofs there implicitly assume that $\mathbb{Y}^\tau_*$ coincides with the whole projection complex. Nonetheless, this does not invalidate the consequences the authors get in \cite[Section 5]{dfdt} about mapping class groups of surfaces, since in that setting the composite rotating family is defined on the whole projection complex.
    \\
    In our setting, instead, it is relevant that the composite rotating family $\mathbb{Y}_*$ is \emph{not} defined on the whole $\calS^{(0)}$, because we want to be able to quotient by any collection of cyclic directions. We establish a way to pass from the whole projection complex $\calS^{(0)}$ to the sub-complex $\mathbb{Y}_*$ in the Transfer-like Lemma~\ref{lem:transfer_like} below, whose main ingredients are uniform properness (Definition~\ref{defn:uniform_properness_CPS}) and the fact that $G$ acts cofinitely on $\calS^{(0)}$.
\end{itemize}
\end{rem}

\begin{assumption}
    For the rest of the Subsection, let $\calS$ be a hyperbolic graph, let $(\calS^{(0)},\theta,\dist^\pi_*)$ be a uniformly proper composite projection system, and let $(\mathcal S, G, \mathbb{Y}_*, \{\Gamma_v\})$ be a SCPG with rotation control $\Theta_{rot}$. Let $\N$ be the subgroup of $G$ generated by $\{\Gamma_v\}_{v\in\mathbb{Y}_*}$.
\end{assumption}

\begin{lemma}[Transfer-like Lemma]\label{lem:transfer_like}
There exists a constant $B\ge 0$, not depending on $\Theta_{rot}$, with the following property. For every $x\in \calS^{(0)}$ and every $i=1,\ldots,m$ there exists $t_i(x)\in\mathbb{Y}_i$ such that, for every $y\in \Act(x)\cap \Act(t_i(x))$, we have that $\dist^\angle_y(x, t_i(x))\le B$. 
\end{lemma}

\begin{proof}
Let $x_1,\ldots, x_l\in \mathbb{Y}_*$ be a finite collection of representatives of the $G$-orbits. For each of these points $x_j$ and for each colour $i=1,\ldots, m$ choose a point $t_i(x_j)\in\mathbb{Y}_i$. By uniform properness (Definition~\ref{defn:uniform_properness_CPS}, together with Remark~\ref{rem:d^angle}), there exists a constant $B$ such that
$$\max_{i,j}\sup_{y\in\Act(x_j)\cap \Act(t_i(x_j))} \dist^\angle_y(x_j, t_i(x_j))\le B$$
Now, for every $x\in \mathbb{Y}_*$ there exists $j(x)\le l$ and an element $g\in G$ such that $x=g(x_{j(x)})$. Moreover, since $G$ acts on the set of colours, for every $i=1,\ldots, m$ there exists a unique $i'\in\{1,\ldots, m\}$ such that $gt_{i'}(x_{j(x)})\in \mathbb{Y}_{i}$, and we can set $t_i(x)=gt_{i'}(x_{j(x)})$. 
Notice that, since $G$ preserves projection distances, 
$$\sup_{y\in\Act(x)\cap \Act(t_i(x))} \dist^\angle_y(x, t_i(x))=\sup_{y\in\Act(x_j)\cap \Act(t_{i'}(x_j))} \dist^\angle_y(x_j, t_{i'}(x_j))\le B,$$
and this concludes the proof.
\end{proof}

\begin{cor}[{cf. \cite[Corollary 1.8]{dfdt}}]\label{cor:spostamento}
    The following holds if the rotation control $\Theta_{rot}$ is large enough. Let $v\in\calS^{(0)}$ and $w\in \mathbb{Y}_*$. If $v$ is $w$-active, then $\dist^\angle_w(v, \gamma_w v)>\theta$ for every $\gamma_w\in\Gamma_w-\{0\}$. In particular, $\gamma_w v\neq v$.
\end{cor}

\begin{proof}
    Let $t=t_{i(w)}(v)$. Since now $t$ and $w$ lie in the same colour, we have that $\dist^\angle_w(t, \gamma_w t)>\Theta_{rot}$. Then the coarse triangular inequality yields that $\dist^\angle_w(v, \gamma_w v)>\Theta_{rot}-2B-2\kappa$, and we can choose the rotation control so that the latter quantity is greater than $\theta$.
\end{proof}

The following Lemma combines some properties from \cite[Section 3]{dfdt}, which only use that $ \{\Gamma_v\}_{v\in\mathbb{Y}_*}$ is a composite rotating family on a composite projection complex $\mathbb{Y}_*$.
\begin{lemma}\label{lem:reduction_for_original_rotating_fam}
    There exist a constant $\aleph\in\mathbb{R}$ depending only on $\mathbb{Y}_*$, a well-order of $\N$ which we call \emph{complexity}, and an indexing $i\colon \N\to\{1,\ldots, m\}$ such that the following holds. For every $\gamma\in \N-\{1\}$ and every $t\in \mathbb{Y}_{i(\gamma)}$, there exist $s \in \mathbb{Y}_*$ and $\gamma_s\in\Gamma_s$ such that:
    \begin{itemize}
        \item $\gamma_s\gamma$ has strictly lower complexity than $\gamma$;
        \item $\dist^\angle_s(t, \gamma t)\ge \Theta_{rot}/2-\aleph$.
    \end{itemize}
\end{lemma}

\begin{proof}
    Set $\aleph=\Theta_0/2 + \kappa$, where $\kappa$ is the constant from Remark~\ref{rem:d^angle} and $\Theta_0$ is defined in \cite[Remark 1.2 and Standing assumption 1.4]{dfdt} as
    $\Theta_0=8\theta+2+3\kappa.$
    We stress that both $\theta$ and $\kappa$ only depend on the composite projection system on $\mathbb{Y}_*$ and on the modified distance function from Remark~\ref{rem:d^angle}.
    The well-order is the lexicographic order $(\alpha(\gamma), n(\gamma))$, where the ordinal $\alpha$ and the integer $n$ are defined in \cite[Theorem 3.1 and Definitions 3.2 and 3.3]{dfdt}. The index $i(\gamma)\coloneq  i(\alpha(\gamma))$ is defined in \cite[Theorem 3.1]{dfdt}. The statement about $\gamma$ and $t$ follows by combining the second and the fourth bullet of \cite[Theorem 3.5]{dfdt}.
\end{proof}

Now set $\Theta_{short} =\Theta_{rot}/2-\aleph-2B-2\kappa$, where $B$ is the constant from the Transfer-like Lemma~\ref{lem:transfer_like}.

\begin{lemma}[Shortening pair]\label{lem:short_less_complex}
For all $\gamma\in \N-\{1\}$, and all $x\in \calS^{(0)}$, there exist a \emph{shortening pair} $(s,\gamma_s)$ (here $s\in\mathbb{Y}_*$ and $\gamma_s\in\Gamma_s$)  so that  $\gamma_s\gamma$ has strictly lower complexity than $\gamma$, and either
\begin{enumerate}
	 \item at least one of $x$ and $\gamma x$ is $s$-inactive, or
	 \item $\dist^\angle_s(x,\gamma x)\ge \Theta_{short}$.	
\end{enumerate}
\end{lemma}

\begin{proof}
Let $t=t_{i(\gamma)}(x)$, defined as in the Transfer-like Lemma~\ref{lem:transfer_like}. By Lemma~\ref{lem:reduction_for_original_rotating_fam}, there exist $s \in \mathbb{Y}_*$ and $\gamma_s\in\Gamma_s$ such that $\gamma_s\gamma$ has strictly lower complexity, and $\dist^\angle_s(t, \gamma t)\ge \Theta_{rot}/2-\aleph$. Now, if one between $x$ and $\gamma x$ is $s$-inactive we are done. Otherwise, the (coarse) triangle inequality yields
$$\dist^\angle_s(x,\gamma x)\ge \dist^\angle_s(t,\gamma (t))-2 B-2\kappa \ge \Theta_{short},$$
where we used that 
$$\dist^\angle_s(x,t)\le \sup_{v\in\Act(x)\cap \Act(t)}\dist^\angle_v(x,t)\le B,$$
and therefore also 
$$\dist^\angle_s(\gamma x,\gamma t)\le \sup_{v\in\Act(\gamma x)\cap \Act(\gamma t}\dist^\angle_v(\gamma x,\gamma t)=\sup_{v'\in\Act(x)\cap \Act(t)}\dist^\angle_{v'}(x,t)\le B.$$
\end{proof}

\subsubsection{Lifting and projecting}
From now on, let $\calS/\N$ be the graph whose vertices and edges are $\N$-orbits of vertices and edges in $\calS$, and let $q\colon \calS\to \calS/\N$ be the quotient projection. Notice that, a priori, $\calS/N$ might not be a simplicial graph.

The next lemmas are the analogues of the results in \cite[Section 4]{dfdt}, whose proofs can be run verbatim in our setting. Indeed, they all rely only on \cite[Corollary 3.6]{dfdt} (which is our Lemma~\ref{lem:short_less_complex}), and the fact that, whenever $x,y\in\calS^{(0)}$ have sufficiently large projection on some $s\in \calS^{(0)}$, then every geodesic $[x,y]$ contains a point $w$ which is fixed by $\Gamma_s$ (which for us is a consequence of the strong bounded geodesic image assumption). 

\begin{lemma}\label{lem:simplicial_quotient_graph}
 If the rotation control $\Theta_{rot}$ is sufficiently large, then $\calS/\N$ is a simplicial graph.
\end{lemma}

\begin{proof}
    \textbf{No double edges.} Let $e=\{v,w\}$ be any edge in $\calS$, and suppose there exist $n,n'\in \N$ such that $e'=\{nv,n'w\}$ is also an edge. Up to replacing $e'$ by $(n')^{-1}e'$, assume that $n'=1$. We now prove that $e$ and $e'$ must lie in the same $\N$-orbit, and therefore $\calS/\N$ has no double edges.
    
    We argue by induction on the complexity of $n$. If $nv=v$, then the two edges coincide, and we are done. Otherwise, by Lemma~\ref{lem:short_less_complex} there exists a shortening pair $(s,\gamma_s)$. If $v$ is $s$-inactive, then we can apply $\gamma_s$ to both edges, and we get a new concatenation of translates of $e$ and $e'$ with endpoints $v$ and $\gamma_s n v$. Then we can conclude by induction, as $\gamma_s n$ has strictly lower complexity than $n$. Similarly, if $nv$ is $s$-inactive, we simply replace $n$ by $\gamma_s n$, and again we get a concatenation of two edges with endpoints $v$ and $\gamma_s n v$. If $w$ is $s$-inactive, we can apply $\gamma_s$ to $e'$; this yields a new concatenation of translates of $e$ and $e'$ since $\gamma_s$ fixes $w$, and again we conclude by induction. 

If none of the above happens, we have that $\dist^\angle_s(v, nv)\ge \Theta_{short}$, and by the coarse triangle inequality one of $\dist^\angle_s(v, w)$ and $\dist^\angle_s(w, nv)$ is at least $\Theta_{short}/2-\kappa$. Notice that we can always choose $\Theta_{rot}$ large enough that the latter quantity is greater than the constant $C$ from the strong BGI assumption. But this is a contradiction, because then one of $v,w,nv$ would have to be fixed by $\gamma_s$.

\textbf{No loops.} Towards a contradiction, suppose that there exists an edge $e$ of $\calS$ with endpoints $v,nv$ for some $n\in \N$, and choose $n$ of minimal complexity among all such edges. Since $n\neq 1$, there exists a shortening pair $(s,\gamma_s)$, as in Lemma~\ref{lem:short_less_complex}. If $v$ is $s$-inactive, then we can replace $e$ by $\gamma_s e=\{v,\gamma_s nv\}$, thus violating minimality of $n$ since $\gamma_s n$ has strictly lower complexity. Similarly, if $nv$ is $s$-inactive, then $e=\{v,\gamma_snv\}$, again contradicting minimality of $n$. Otherwise $\dist^\angle_s(v, nv)\ge \Theta_{short}$, which we can assume to be bigger than the constant $C$ from the strong BGI assumption. This is a contradiction, since then one f $v$ and $nv$ would have to be $s$-inactive.
\end{proof}

For the next lemma, by a \emph{combinatorial path} in a simplicial graph we mean a concatenation of edges.

\begin{lemma}\label{lem:geodesic_lifts} For each combinatorial path $\ov\gamma$ in $\calS/\N$ starting at $\ov x$, and any point $x$ in the preimage of $\ov x$ (henceforth: a lift of $\ov x$), there exists a combinatorial path in $\calS$ so that $q\circ\gamma=\ov \gamma$, which we call a lift of $\ov\gamma$. Moreover, if $\ov\gamma$ is a geodesic, then so is $\gamma$.
\end{lemma}

\begin{proof}
    This follows from the fact that $\N$ acts simplicially on $\calS$, and that $q$ is $1$-Lipschitz.
\end{proof}

\begin{defn}\label{def:kgon_generic}
    A $k$-gon in a simplicial graph $\Lambda$ is a subgraph of the form
    $$Q=\bigcup_{i=0}^{k-1}[x_i, x_{i+1}],$$
    where $x_0,\ldots,x_{k}\in\Lambda^{(0)}$, $x_{k}=x_0$ and each $[x_i, x_{i+1}]$ is a geodesic path.
\end{defn}

\begin{lemma}[{cf. \cite[Proposition 4.3]{dfdt}}]\label{lem:rotating_lift_mgons}
    For every $k\in\mathbb{N}$ there exists a constant $\Theta_{rot}(k)$ such that the following holds if the rotation control $\Theta_{rot}$ is larger than $\Theta_{rot}(k)$. For every $k$-gon $\ov{Q}\subset \calS/\N$ there exists a $k$-gon $Q\subset \calS$ such that $q(Q)=\ov Q$, which we call a \emph{lift} of $\ov Q$.
\end{lemma}

\begin{proof}
Let $\ov Q=\bigcup_{i=0}^{k-1}[\ov x_i, \ov x_{i+1}]$. 
Lift all segments of $\ov Q$ to a (possibly non-closed) chain $\bigcup_{i=0}^{k-1}[x_i, x_{i+1}]$. Let $\gamma\in \N$ be such that $\gamma x_0=x_k$. We proceed by induction on the complexity of $\gamma$.
If $\gamma=1$ then $x_0=x_k$ and the chain is already a $k$-gon. Otherwise, by Lemma~\ref{lem:short_less_complex} there exists a shortening pair $(s,\gamma_s)$. If $x_0$ is $s$-inactive, then we can apply $\gamma_s$ to the whole chain, and we get a new chain with endpoints $x_0'=\gamma_s x_0=x_0$ and $x_k'=\gamma_s\gamma x_0$. Then we can conclude by induction, as $\gamma_s\gamma$ has strictly lower complexity than $\gamma$. Similarly, if $\gamma x_0$ is $s$-inactive, we can apply $\gamma_s^{-1}$ to the whole chain, and we get a new chain with endpoints $x_0'=\gamma_s^{-1}x_0$ and $x_k'=\gamma_s\gamma x_0=\gamma x_0$. Again, the endpoints of the chain are such that $x_k'=\gamma_s\gamma x_0'$, and we conclude by induction.

Otherwise, we have that $\dist^\angle_s(x_0, \gamma x_0)\ge \Theta_{short}$. If at least one $x_i$ is $s$-inactive, we can apply $\gamma_s$ to the chain “after” $x_i$ (meaning, to all geodesics $[x_j,x_{j+1}]$ where $i\le j$), and again conclude by induction. Otherwise, by the coarse triangle inequality we get that there exists some $i\in\{0,\ldots, k-1\}$ such that $\dist^\angle_s(x_i, \gamma x_{i+1})\ge \Theta_{short}/k-k\kappa$. Notice that we can always choose $\Theta_{rot}$ large enough that the latter quantity is greater than the constant $C$ from the bounded geodesic image assumption. Therefore, there exists $w\in [x_i,x_{i+1}]$ such that either $w=s$ or $w$ is $s$-inactive. In both cases $\gamma_s$ fixes $w$, and therefore one can change the lift of the chain “after” $w$ (meaning, one can apply $\gamma_s$ to $[w,x_{i+1}]$ and to all geodesics $[x_j,x_{j+1}]$ where $i+1\le j$). Again, the conclusion then follows by induction.
\end{proof}

With the same techniques, one can prove the following two results:

\begin{lemma}[{cf. \cite[Proposition 4.3]{dfdt}, “moreover” part}]\label{lem:quadranglelift_if_thetarot/10_ge} Let $\ov Q$ be a geodesic quadrangle in $\calS/\N$. If the geodesics $[\ov v_1, \ov w_1]$, $[\ov v_2, \ov w_2]$ of Q have lifts $[v_i, w_i]$ so that
$\dist^\angle_s(v_i, w_i) \le\Theta_{rot}/10$ whenever the quantity is defined, then there exists a lift $Q$ of $\ov Q$ such that the lift $[v'_i, w'_i]$ of
$[\ov v_i, \ov w_i]$ contained in $Q$ is an $\N$-translate of $[v_i, w_i]$.
\end{lemma}

\begin{lemma}[{cf. \cite[Lemma 4.4]{dfdt}}]\label{lem:isom_proj}
    Suppose that $x,y\in\calS^{(0)}$ have the property that $\dist^\angle_s(x,y)\le \Theta_{rot}/10$ whenever the quantity is defined. Then $q|_{[x,y]}$ is isometric, for any geodesic $[x, y]\subset\calS$.
\end{lemma} 
We can finally prove the following:

\begin{thm}[{cf. \cite[Theorem 2.1]{dfdt}}]\label{thm:rotating}
    Let $(\mathcal S, G, \mathbb{Y}_*,  \{\Gamma_v\})$ be a SCPG, and let $\N$ be the subgroup of $G$ generated by $\bigcup_{v\in \mathbb{Y}_*}\Gamma_v$. If the rotation control $\Theta_{rot}$ is sufficiently large, then:
    \begin{enumerate}
        \item\label{item_hyp} $\mathcal S/\N$ is $\delta$-hyperbolic, where $\delta$ is any hyperbolicity constant for $\mathcal S$;
        \item If the action of $G$ on $\mathcal S$ admits a loxodromic element (resp. loxodromic WPD), then so does the action of $G/\N$ on $\mathcal S/\N$. More precisely, for every loxodromic (resp. loxodromic WPD) element $g\in G$ there exists a constant $\Theta_{loxo}(g)$ such that, if $\Theta_{rot}>\Theta_{loxo}(g)$, then $\ov g\in G/\N$ is loxodromic (resp. loxodromic WPD).
        \item\label{item:nonelm} If the action of $G$ on $\mathcal S$ is non-elementary, then so is the action of $G/\N$ on $\mathcal S/\N$.
    \end{enumerate}
\end{thm}

\begin{proof}
    Item~\eqref{item_hyp} follows from the fact that geodesic triangles in $\calS/\N$ lift to geodesic triangles in $\calS$, and that the projection $\calS\to \calS/\N$ is $1$-Lipschitz. Next, fix a basepoint $x_0\in \calS$, and let $g\in G$ act loxodromically on $\calS$. 
    
    \begin{claim}\label{claim:loxo_stay_loxo}
    There exists $\Theta_{loxo}(g)>0$ such that $\dist^\angle_s(x_0,g^kx_0)\le \Theta_{loxo}(g)/10$ for every $k\in \mathbb Z$ and every $s$ for which the projection is defined. 
    \end{claim}
    If the Claim is true, then Lemma~\ref{lem:isom_proj} implies that the projection map is an isometry on $\{g^kx_0\}_{k\in \mathbb N}$ whenever $\Theta_{rot}>\Theta_{loxo}(g)$, and in particular $\ov g$ is loxodromic. If moreover $g$ is WPD, then $\ov g$ is WPD as well by \cite[Proposition 4.6]{dfdt}, whose proof runs verbatim once one replaces \cite[Proposition 4.3]{dfdt} by Lemma \ref{lem:quadranglelift_if_thetarot/10_ge}.

    \begin{claimproof}[Proof of Claim~\ref{claim:loxo_stay_loxo}]
    Let $\mu=\mu(g,\delta)$ be such that each geodesic $[x_0,g^kx_0]$ lies at Hausdorff distance at most $\mu$ from the quasigeodesic sequence $x_0 , gx_0,\ldots, g^kx_0$. Fix $s\in \calS^{(0)}$. By strong BGI, either $\dist^\angle_s(x_0, g^kx_0)\le C$, or there exists some vertex on $[x_0,g^kx_0]$ within distance $C$ from $s$. Hence there exist integers $p, q$ with $1\le p \le q \le k$ and $|p-q|$ bounded by some constant depending only on $g$, so that $s$ is not $C$-close to any vertex of $[x_0 , g^i x_0 ]$ or $[g^i x_0 , g^k x_0 ]$ unless $p \le i \le q$. Thus, by strong BGI and the coarse triangle inequality, we get $$\dist^\angle_s (x_0 , g^k x_0 ) \le \dist^\angle_s (g^px_0 , g^q x_0 ) + 2C+2\kappa=\dist^\angle_{g^{-p}s} (x_0 , g^{q-p} x_0 ) + 2C+2\kappa.$$ 
    Now notice that the quantity $\sup_{s\in \mathbb Y_*}\dist^\angle_{g^{-p}s} (x_0 , g^{q-p} x_0 )=\sup_{s\in \mathbb Y_*}\dist^\angle_{s} (x_0 , g^{q-p} x_0 )$ is bounded above by uniform properness. Moreover, the difference $q-p$ can assume finitely many values, so we see that $\dist^\angle_s (x_0 , g^k x_0 )$ is uniformly bounded.
    \end{claimproof}

    We finally turn to Item \eqref{item:nonelm}. Let $g,h$ be independent loxodromic elements, and suppose that $\Theta_{rot}\ge \max\{\Theta_{loxo}(g),\Theta_{loxo}(h)\}$, so that both $\ov g$ and $\ov h$ are loxodromic in the quotient. 

     \begin{claim}\label{claim:ind_stay_ind}
    There exists $\Theta_{ind}(g,h)>0$ such that $\dist^\angle_s(g^kx_0,h^jx_0)\le \Theta_{ind}(g,h)/10$ for every $k,l\in \mathbb Z$ and every $s$ for which the projection is defined. 
    \end{claim}
    
    \begin{claimproof}
        Since $g,h$ are independent, there exists $\mu$ such that, for all integers $k,l$, any geodesic $[g^k x_0 , h^l x_0 ]$ lies at Hausdorff distance at most $\mu$ from the sequence $\{g^k x_0 , g^{k-1} x_0 , \ldots , x_0 \} \cup \{x_0 , h x_0 , \ldots , h^l x_0 \}$. We can now argue as in Claim~\ref{claim:loxo_stay_loxo}, replacing $\{x_0 , gx_0 , \ldots , g^k x_0 \}$ by $\{g^k x_0 , g^{k-1} x_0 , \ldots , x_0 \} \cup \{x_0 , h x_0 , \ldots , h^l x_0 \}$.
    \end{claimproof}

    By combining Claim~\ref{claim:ind_stay_ind} and Lemma~\ref{lem:isom_proj} we get that the projection map is an isometry on $\langle g\rangle x_0\cup \langle h\rangle x_0$, and therefore $\ov g$ and $\ov h$ are again independent.
\end{proof}

\subsubsection{Other consequences}
We gather here a collection of facts which are not used in the above proof, but are relevant for our arguments on short HHG. Firstly, the following theorem gives a presentation of $\N$ under the (fairly restrictive) assumption that any two vertices are active. This is stated in the language of the more general \cite[Theorem 2.2]{dahmani:rotating}, though it already follows from \cite[Theorem 5.3.(a)]{DGO}.

\begin{thm}\label{lem:structure_of_kernel}
    Let $\{\Gamma_v\}_{v\in \mathbb{Y}_*}$ be a composite rotating family on $(\mathbb{Y}_*,G)$. Suppose that $\Act(z)=\mathbb{Y}_*$ for every $z\in \mathbb{Y}_* $. If the rotation control $\Theta_{rot}$ is sufficiently large, there exists a (possibly infinite) subset $\mathcal J\subset \mathbb{Y}_*$ such that
    $$\N \cong \bigast_{v\in \mathcal J}\Gamma_v.$$
\end{thm}

Next, we describe the elements of the kernel $\N$ fixing a given vertex;
\begin{lemma}[{cf. \cite[Proposition 4.8]{dfdt}}] \label{lem:stabiliser_in_scpg}
    If the rotation control $\Theta_{rot}$ is sufficiently large, then for any vertex $v\in \mathcal S$ we have
     $$\Stab{G}{v}\cap \N=\langle\Gamma_w\cap\Stab{G}{v}\rangle_{w\in \mathbb{Y}_*}= \langle\Gamma_w\rangle_{w\in \mathbb{Y}_*-(\Act(v)-v)}.$$
\end{lemma}

\begin{proof}
    One can follow the proof of \cite[Proposition 4.8]{dfdt}, which only uses the existence of shortening pairs (which is Lemma~\ref{lem:short_less_complex}) and \cite[Corollary 1.8]{dfdt} (which is our Corollary~\ref{cor:spostamento}).
\end{proof}

We conclude with two lemmas allowing us to inject certain finite configurations in the quotient. First, a consequence of Lemma~\ref{lem:isom_proj}:
\begin{cor}\label{lem:finite_vertices_inject}
    For any two distinct vertices $v,w\in\calS^{(0)}$ there exists a constant $\Theta_{rot}(v,w)$ such that, if the rotation control $\Theta_{rot}$ is greater than $\Theta_{rot}(v,w)$, then $\ov v\neq \ov w$.  As a consequence, for every finite set $\mathcal{W}\subseteq \calS^{(0)}$, there exists a constant $\Theta_{rot}(\mathcal W)$ such that, if the rotation control $\Theta_{rot}$ is greater than $\Theta_{rot}(\mathcal W)$, then $\mathcal{W}$ injects inside $\calS^{(0)}/\N$.
\end{cor}

In the same spirit, we show that we can inject finite subsets of $G$ in the quotient group:

\begin{lemma}\label{lem:separating_rotating_fam}
    For every $g\in G-\{1\}$ there exists a constant $\Theta_{rot}(g)$ such that, if the rotation control $\Theta_{rot}$ is greater than $\Theta_{rot}(g)$, then $g\not\in\N$. 

    As a consequence, for every finite subset $F\subset G$ there exists a constant $\Theta_{rot}(F)$ such that, if $\Theta_{rot}\ge\Theta_{rot}(F)$, then $F$ injects in $G/\N$.
\end{lemma}

\begin{proof}
    Take an element $y_i\in\mathbb{Y}_i$ for every colour $i=1,\ldots,m$. By uniform properness, there exists a constant $M$ such that
$$\max_{i=1,\ldots,m}\sup_{s\in\Act(y_i)\cap\Act(g y_i)} \dist^\angle_s(y_i,g y_i)\le M.$$
Choose $\Theta_{rot}(g)$ such that $\Theta_{rot}(g)/2-\aleph> M$, where $\aleph$ is the constant from Lemma~\ref{lem:reduction_for_original_rotating_fam} which only depends on $\mathbb{Y}_*$. Now let $\N$ have rotation control greater than $\Theta_{rot}(g)$. If by contradiction $g\in \N$, then we can define its index $i(g)\in\{1,\ldots,m\}$. By Lemma~\ref{lem:reduction_for_original_rotating_fam} we can then find $s\in \mathbb{Y}_*$ such that
$$M<\Theta_{rot}/2-\aleph\le \dist^\angle_s(y_{i(g)}, g y_{i(g)})\le M,$$
and this yields a contradiction.
\end{proof}

\subsection{From a short HHG to a SCPG}\label{rem:Short_is_SCPG}
In this Subsection we show that short HHG fit the framework of strong composite projection graphs. We break the argument into a series of Propositions.

Let $(G, \ov{X},\W)$ be a colourable short HHG, with HHS constant $E$. For every $v\in \ov{X}^{(0)}$ set $\Act(v)=\ov{X}^{(0)}-\link_{\ov{X}}(v)$, which contains all vertices with the same colour as $v$ since no two $\ov X$-adjacent vertices have the same colour. Furthermore, set $\dist^\pi_v(x,y)=\diam_{\ell_v}(\rho^x_v\cup\rho^y_v)$ whenever the quantity is defined. 

\begin{prop}
    There exists $\theta\ge 0$ such that $(\ov{X}^{(0)}, \theta, \dist^\pi_*)$ is a uniformly proper composite projection system.
\end{prop}

\begin{proof}
     The proof is similar to that of \cite[Lemma 8.21]{BHMS}. Symmetry and the triangle inequality are clear, since $\dist^\pi_v$ is define as a diameter of a union. The Behrstock inequality is implied by the existence of a HHG structure, as it follows by combining the partial realisation axiom \cite[Definition 1.1.(8)]{HHS_II} with the first consistency inequality \cite[Definition 1.1.(4)]{HHS_II}. Uniform properness follows from the Distance Formula \cite[Theorem~4.5]{HHS_II}, again together with \cite[Definition 1.1.(8)]{HHS_II}. Separation corresponds to the fact that, for every two non-adjacent $v,w\in \ov X^{(0)}$, $\diam\rho^w_v\le E$. 

    Symmetry in action is also clear, since $x\in\Act(y)$ if and only if $x\not \in \link_{\ov X}(y)$. Moving to closeness in inaction, if $x,y,z\in \ov X^{(0)}$ are such that $x$ and $y$ are $\ov X$-adjacent while $z$ is disjoint from both, then $\ell_x\orth\ell_y$ and both are transverse to $\ell_z$. Then $\dist^\pi_z(x,y)$ is uniformly bounded by \cite[Lemma 1.5]{DHS}.

    We finally prove finite filling. Let $\mathcal Z\subseteq \ov X^{(0)}$. By definition, $y\in \bigcup_{x\in \mathcal Z}\Act(x)$ if and only if $y\not \in \bigcap_{x\in \mathcal Z}\link_{\ov X}(x)$, so it is enough to show that there does not exists an infinite sequence $\{x_i\}_{i\in \mathbb{N}}$ in $ \ov X^{(0)}$ such that $$\bigcap_{i\le k}\link_{\ov X}(x_i)\subsetneq \bigcap_{i\le k=1}\link_{\ov X}(x_i)$$ for all $k$. This is because, if $x_1$ and $x_2$ are $\ov X$-adjacent, then $\link_{\ov X}(x_1)\cap \link_{\ov X}(x_2)=\emptyset$, since $\ov X$ is triangle-free; if instead $x_1$ and $x_2$ are not $\ov X$-adjacent, then $\link_{\ov X}(x_1)\cap \link_{\ov X}(x_2)$ consists of at most one vertex, since $\ov X$ is square-free. 
\end{proof}

Now, $G$ acts by isomorphisms on this composite projection system, because it acts by isomorphisms on $\ov X$, the action permutes the colours, and projections in the HHG structure are $G$-equivariant. 

As in Notation~\ref{notation:kernel}, let $\B=\{s_1,\ldots,s_r\}$ be a base, such that every $Z_{v_i}$ is infinite. Choose non-trivial integers $M_1,\ldots,M_r$, and set $\Gamma_{gv_i}=g(M_iZ_{v_i})g^{-1}$ for all $g\in G$ and $i=1,\ldots, r$. Set $\mathbb{Y}_*=G\B$, with the induced set of colours. 

\begin{prop}
    For every $\Theta_{rot}>0$ there exists $M\in \mathbb{N}-\{0\}$ such that, if every $M_i$ is a non-trivial multiple of $M$, then $\{\Gamma_v\}_{v\in \mathbb Y_*}$ is a composite rotating family on $(\mathbb Y_*, G)$ with rotation control $\Theta_{rot}$.
\end{prop}
With the notation from Definition~\ref{defn:deep_enough}, the above Proposition states that every rotation control can be achieved if $\N$ is deep enough.

\begin{proof}
    By construction, every $\Gamma_v$ is infinite, and it acts trivially on the star of $v$ in $\ov X$ by Axiom~\ref{short_axiom:extension}. The collection $\{\Gamma_v\}$ is equivariant by construction, and if $M$ is even then $\Gamma_x$ commutes with $\Gamma_y$ whenever $x\in \link_{\ov X}(y)$. 
    Proper isotropy, as well as the last axiom, both follow from the fact that $Z_v$ acts metrically properly and coboundedly on the quasiline $\C\ell_v$.
\end{proof}

Finally, let $\calS=\ov X^{+\W}$ be the augmented support graph, which is hyperbolic as it is $G$-equivariantly quasi-isometric to the main coordinate space $\C S$ (see Remark~\ref{rem:coord_spaces_wrt_lines}). Notice that, by Lemma~\ref{lem:strong_bgi_for_short}, every $\ov X^{+\W}$-geodesic with large projection to some $\ell_v$ contains a vertex in the $\ov X$-star of $v$; hence the strong Bounded Geodesic Image property from Definition \ref{dfn:scpg} holds in this setting. Summing up, we proved the following:

\begin{prop} For every $\Theta_{rot}>0$ there exists $M\in \mathbb{N}$ such that, if every $M_i$ is a non-trivial multiple of $M$, the collection $(\ov X^{+\W}, G, G\B, \{\Gamma_v\}_{v\in G\B})$ is a strong composite projection graph with rotation control $\Theta_{rot}$.
\end{prop}

\begin{rem}\label{rem:why_not_dfdt}
Notice that we could not have used the original definition from \cite{dfdt}. Indeed, in our setting the hyperbolic graph is $\ov X^{+\W}$ (not $\ov X$); furthermore, $\link_{\ov X^{+\W}}(s)$ contains $\link_{\ov X}(s)$ but they might not coincide, and this means that there might be vertices of $\link_{\ov X^{+\W}}(s)$ which are \emph{not} fixed by $\Gamma_s$. \end{rem}

\begin{rem}[SCPG for augmented links]\label{rem:scpg_local_version}
For every $v\in\ov{X}^{(0)}$ set $\mathbb{Y}^v_*=G\B \cap \link_{\ov X}(v)$. Arguing as above, one can check that the collection $$\left(\link_{\ov X}(v)^{+\W}, \Stab{G}{v}, \mathbb{Y}^v_*,  \{\Gamma_w\}_{w\in \mathbb{Y}^v_*}\right),$$ defines a strong composite projection graph, whose rotation control can be made arbitrarily large (for all vertices at once) by choosing $\N$ is deep enough. We stress that, since any two vertices $w,w'\in \mathbb{Y}^v_*$ are always disjoint in $\link_{\ov X}(v)$, one has that $\Act(w)=\mathbb{Y}^v_*$ for every $w\in \mathbb{Y}^v_*$. This is relevant as then, by Lemma~\ref{lem:structure_of_kernel}, the subgroup $\langle \Gamma_w\rangle _{w\in \mathbb{Y}^v_*}$, which we shall later denote by $\N_v$, is a free product of (some of) the $\Gamma_w$.
\end{rem}

\subsection{Summary of the consequences for short HHG}\label{subsec:rotating_for_short}
We gather here all implications of the above discussion to our setting. Recall that we have a colourable short HHG $(G, \ov X, \W)$, with HHS constant $E$, and we are considering the quotient by a subgroup $\N$, as in Notation~\ref{notation:kernel}.

\subsubsection{The quotient support graph}\label{subsec:support_for_quotient} 
Let $(\Lambda, N)$ be either $(\ov X^{+\W},\N)$ or $(\link_{\ov X}(v)^{+\W}, \N_v)$ for some $v\in\ov{X}^{(0)}$. Define $\Lambda/N$ as the graph whose vertices and edges are $N$-orbits of vertices and edges of $\Lambda$, and let $q\colon \Lambda\to \Lambda/N$ be the quotient projection. We shall denote the $N$-orbit of a vertex $v\in \Lambda$ by $[v]$. 
\begin{cor}[of Lemma~\ref{lem:simplicial_quotient_graph}]\label{cor:simplicial_x^w/n}
    $\Lambda/N$ is simplicial.
\end{cor}
Moreover, a plethora of subgraphs of $\Lambda/N$ lift isometrically to $\Lambda$. For example: 
\begin{cor}[of Lemma~\ref{lem:geodesic_lifts}]\label{lem:lift_geodesics_for_short}
For each combinatorial path $\ov\gamma$ in $\Lambda/N$ starting at $[v]$, and any point $v$ in the preimage of $[v]$ (henceforth: a lift of $[v]$), there exists a combinatorial path in $\Lambda$ so that $q\circ\gamma=\ov \gamma$, which we call a lift of $\ov\gamma$. Moreover, if $\ov\gamma$ is a geodesic, then so is $\gamma$.
\end{cor}

Furthermore, we can also lift \emph{geodesic $k$-gons}:
\begin{defn}\label{defn:kgon}
    A $k$-gon in a metric graph $\Lambda$ is a closed combinatorial path $Q$ made of $k$ geodesic segments. In other words,
    $$Q=\bigcup_{i=0}^{k-1}[x_i, x_{i+1}],$$
    where $x_0,\ldots,x_{k}\in\Lambda^{(0)}$, $x_{k}=x_0$ and each $[x_i, x_{i+1}]$ is a geodesic path.
\end{defn}

\begin{cor}[of Lemma~\ref{lem:rotating_lift_mgons}]\label{lem:lifting}
For every $k\in\mathbb{N}$ the following holds if $N$ is deep enough. For every $k$-gon $\ov{Q}$ inside $\Lambda/N$ there exists a $k$-gon $Q$ inside $\Lambda$ such that $q(Q)=\ov Q$. We say that $Q$ is a \emph{lift} of $\ov{Q}$.
\end{cor}

Now we focus our attention on $\Lambda=\ov X^{+\W}$. Recall that $\ov X$ is a (non-full) $G$-invariant subgraph of $\ov X^{+\W}$, as pointed out in Remark~\ref{rem:coord_spaces_wrt_lines}, so the quotient projection restricts to a map $\ov X\to \ov X/\N$, which we shall still call $q$. As $\ov X/\N$ is then a $G/\N$-invariant subgraph of $\ov X^{+\W}/\N$, it is itself simplicial; furthermore, Corollary~\ref{lem:lift_geodesics_for_short} implies that, for every $k\in\mathbb{N}$, we can find $\N$ deep enough that every closed combinatorial path $\ov\gamma\subset \ov X/\N$ of length at most $k$ has a lift \emph{inside $\ov X$}. Then we summarise the properties of $\ov X/\N$ below:
\begin{lemma}\label{lem:ovX/N_is_good}
    If $\N$ is deep enough, then:
    \begin{itemize}
        \item $\ov X/\N$ is a triangle- and square-free simplicial graph, and none of its connected components is a point.
        \item For every $\N$-orbit $[v]$ of a vertex $v\in\ov X^{(0)}$, $\link_{\ov X/\N}([v])=q(\link_{\ov X}(v)).$
    \end{itemize} 
\end{lemma}

\begin{proof}
    We already noticed that $\ov X/\N$ is simplicial. Furthermore, any non-degenerate triangle or square in $\ov X/\N$ would lift to $\ov X$ if $\N$ is deep enough (notice that the lift would again be non-degenerate, as its edges would have different quotient projections), and every vertex of $\ov X/\N$ belongs to at least one edge, because this is true in $\ov X$.
    Moving to the second bullet, Corollary~\ref{lem:lift_geodesics_for_short} implies that every edge $\{[v], [w]\}$ lifts to an edge $\{v,w\}$. 
\end{proof}

\subsubsection{Large rotations stabilising a vertex}
Given $v\in\ov{X}^{(0)}$, we have the following description of the elements of $\N$ which fix $v$:

\begin{cor}[of Lemma~\ref{lem:stabiliser_in_scpg}]\label{cor:N_cap_Stabv}
     If $\N$ is deep enough, then for every $v\in \ov{X}^{(0)}$ we have that
     $\N\cap\Stab{G}{v}=\langle \Gamma_w\rangle_{w\in\operatorname{Star}_{\ov X}(v)\cap G\B}.$
 \end{cor}

As $\Gamma_v$ acts trivially on $\link_{\ov X}(v)$, we get that
$$\link_{\ov X}(v)/\langle \Gamma_w\rangle_{w\in\operatorname{Star}_{\ov X}(v)\cap G\B}=\link_{\ov X}(v)/\N_v,$$
where $$\N_v\coloneq \langle \Gamma_w\rangle_{w\in\link_{\ov X}(v)\cap G\B}.$$
Moreover, no two elements of $\link_{\ov X}(v)$ are adjacent, as $\ov X$ is triangle-free, so we get:
 \begin{cor}[of Lemma~\ref{lem:structure_of_kernel}]\label{cor:Nv_free_prod}
     If $\N$ is deep enough, then for any $v\in \ov{X}^{(0)}$ there exists $\mathcal J\subseteq\link_{\ov X}(v)^{(0)}$ such that $\N_v$ has a free presentation
     $$\N_v=\bigast_{w\in \mathcal J} \Gamma_w.$$
 \end{cor}

\subsubsection{Uniform hyperbolicity of the quotients}
By choosing $\N$ to be deep enough, we can ensure that the quotient of the main coordinate space of $(X,\W)$ remains hyperbolic, \emph{with the same hyperbolicity constant} (which we bound by the HHS constant $E$):
\begin{cor}[of Theorem~\ref{thm:rotating}, global version]\label{cor:hyp_quotients_from_rotating} If $\N$ is deep enough, then $\ov X^{+\W}/\N$ is $E$-hyperbolic. 
Furthermore,  $G/\N$ acts acylindrically on $\ov X^{+\W}/\N$, and if $G$ acts non-elementarily on $\ov X^{+\W}$ then so does $G/\N$ on $\ov X^{+\W}/\N$.
\end{cor}

By Remark~\ref{rem:scpg_local_version}, we can also establish a “local version” of the above result:
\begin{cor}[of Theorem~\ref{thm:rotating}, local version]\label{cor:local_hyp_quotients_from_rotating} If $\N$ is deep enough, then for every $v\in \ov{X}^{(0)}$ we have that $\link_{\ov X}(v)^{+\W}/\N_v$ is $E$-hyperbolic.
\end{cor}

\subsubsection{Preserving finite data in the quotient}
\begin{cor}[of Lemmas~\ref{lem:finite_vertices_inject} and~\ref{lem:separating_rotating_fam}]\label{cor:separating_filling}
    Let $\{w_1,\ldots,w_k\}\subset \ov X^{(0)}$ be a finite collection of vertices, and let $F\subset G-\{1\}$ be a finite subset.  If $\N$ is deep enough then
    \begin{itemize}
        \item $\{w_1,\ldots,w_k\}$ injects in the quotient $\ov X^{(0)}/\N$;
        \item $F\cap \N=\emptyset$, so that $F$ injects in the quotient $G/\N$.
    \end{itemize}
\end{cor}

Combining the Corollary with the description of $\N_v$, we get the following characterisation of which cyclic directions survive in the quotient:

\begin{lemma}\label{lem:NV_cap_Z_v} The following holds whenever $\N$ is deep enough. For any $v\in \ov{X}^{(0)}-G\B$, we have that $\N\cap Z_v=\{0\}$.
\end{lemma}

\begin{proof} 
First, let $F=Z_{x_1}\cup\ldots \cup Z_{x_l},$ where $x_1,\ldots,x_l$ are representatives of the $G$-orbits of vertices with finite cyclic directions. By Corollary~\ref{cor:separating_filling}, we can choose $\N$ deep enough so that it does not contain any conjugates of elements of $F$, so that $\N$ intersects trivially every finite cyclic direction.

Thus, it remains to show that $\N\cap Z_v=\{0\}$ whenever $v\not\in G\B$ and $Z_v$ is infinite. Notice first that $\N\cap Z_v\le \N\cap\Stab{G}{v}=\N_v$, as $v$ does not belong to $G\B$. Since we assumed $\N$ to be deep enough to satisfy Remark~\ref{rem:initial_D}, we have that $Z_v$ commutes with $\Gamma_w$ for every $w\in\link_{\ov X}(v)$, so $\N\cap Z_v$ must lie in the centre of $\N_v$. Then Corollary~\ref{cor:Nv_free_prod} tells us that $\N_v$ is a free product of (some) $\Gamma_w$s, and in particular it has non-trivial centre if and only if $\N_v=\Gamma_w$ for some $w\in\link_{\ov X}(v)$. In this case, it suffices to notice that $Z_v\cap\Gamma_w\le Z_v\cap Z_w$ must be trivial, since $Z_v$ acts geometrically on the quasiline $\C\ell_v$ while $Z_w$ acts on it with uniformly bounded orbits.
\end{proof}

\subsubsection{The quotient extensions}\label{subsec:quotient_ext}
Recall that, for every $v\in \ov{X}^{(0)}$, we defined $[v]\in \ov{X}^{(0)}/\N$ as its $\N$-orbit. Let $Z_{[v]}=Z_v/(Z_v\cap \N)$ and $H_{[v]}\coloneq H_v/\p_v(\N_v)$.
\begin{lemma}\label{lem:squid_for_quot_ext.1}
For every $v\in \ov{X}^{(0)}$ there is a commutative diagram of group extensions, where the vertical arrows are the restrictions of the quotient projection $G\to G/\N$:
$$\begin{tikzcd}
    0\ar{r}&Z_v\ar{r}\ar[d]&\Stab{G}{v}\ar{r}{\p_v}\ar{d}&H_v\ar{r}\ar{d}&1\\
    0\ar{r}&Z_{[v]}\ar{r}&\Stab{G}{[v]}\ar{r}{\p_{[v]}}&H_{[v]}\ar{r}&0.
\end{tikzcd}$$ 
Consequently, $Z_{[v]}$ is a cyclic, normal subgroup of $\Stab{G}{[v]}$ acting trivially on $\link_{\ov X/\N}([v])$, and the collection of quotient extensions is equivariant with respect to the $G/\N$-action by conjugation.
\end{lemma}

\begin{proof} It is easy to see that $\Stab{G}{[v]}$ is the quotient projection of $\Stab{G}{v}$. Furthermore, define a map $\p_{[v]}\colon\Stab{G}{[v]}\to H_{[v]}$ by sending the coset $g(\N\cap \Stab{G}{v})$ to the coset $\p_v(g)\p_v(\N_v)$, for every $g\in \Stab{G}{v}$. This map is well-defined, as 
$$\p_v(\N\cap \Stab{G}{v})=\p_v\left(\langle \Gamma_w\rangle_{w\in\{\operatorname{Star}_{\ov X}(v)\}\cap G\B}\right)=\p_v(\N_v),$$
and it is a group homomorphism with kernel $Z_{[v]}$. Then one can easily see that the above diagram commutes, using that both $\N\cap\Stab{G}{v}$ and $Z_v$ are normal in $\Stab{G}{v}$. Furthermore, by Lemma~\ref{lem:ovX/N_is_good}, $\link_{\ov X/\N}([v])=q(\link_{\ov X}(v))$, so $Z_{[v]}$ acts trivially on $\link_{\ov X/\N}([v])$). The remaining properties of the quotient extension follow from the corresponding features of the top row of the diagram.
\end{proof}

Next, we show that the quotient $H_{[v]}$ is hyperbolic:

\begin{lemma}\label{lem:Hv/Nv_rel_hyp}
Let $v\in \ov{X}^{(0)}$, and fix a collection $[W]$ of $\Stab{G}{[v]}$-orbit representatives of vertices in $\link_{\ov X/\N}([v])$ with infinite cyclic direction.

If $\N$ is deep enough, then:
\begin{enumerate}
    \item $H_{[v]}$ is hyperbolic;
    \item The collection $\{\p_{[v]}(Z_{[w]})\}_{[w]\in [W]}$ is \emph{independent} in $H_{[v]}$, that is, for every distinct $[w],[w']\in [W]$ and every $\ov h\in H_{[v]}$, $\p_{[v]}(Z_{[w]})^{\ov h} \cap \p_{[v]}(Z_{[w']})=\{1\}$;
    \item $\p_{[v]}(Z_{[w']})$ is quasiconvex in $H_{[v]}$ for every ${[w']\in \link_{\ov X/\N}([v])}$.
\end{enumerate} 
\end{lemma}

\begin{proof}
    Recall from Lemma~\ref{lem:Hv_hyp_1} that $H_v$ is hyperbolic relative to the collection $\{K_{H_v}(\p_v(Z_w))\}_{w\in W}$, for any collection $W$ of $\Stab{G}{v}$-orbit representatives of vertices in $\link_{\ov X}(v)$ with infinite cyclic direction. We can choose $W$ such that its quotient projection contains $[W]$. Now, since $\p_v(\N_v)=\langle \p_v(\Gamma_w)\rangle_{w\in\link_{\ov X}(v)\cap G\B}$, by the relative Dehn Filling Theorem \cite[Theorem 1.1]{Osin_Dehn_Fill} there exists a finite set $F_v\subset H_v-\{1\}$ such that, if $\p_v(\Gamma_w)\cap F_v=\emptyset$ for every $w\in W$, then $H_{[v]}$ is hyperbolic relative to 
    $$\{Q_{[w]}\coloneq K_{H_v}(\p_v(Z_w))/\p_v(Z_w\cap \N))\}_{w\in W}.$$
    By choosing $\N$ deep enough, we can ensure that $\p_v(\Gamma_w)$ avoids $F_v$ for every $w\in W$. Indeed, $\p_v(\Gamma_w)=M_i\p_v(Z_w)$, where $M_i$ is one of the integers from Notation~\ref{notation:kernel}. Since both $F_v$ and $W$ are finite, we can find $M\in \mathbb{N}-\{0\}$ such that $M\p_v(Z_w)\cap F_V=\emptyset$ for every $w\in W$, and then choose each $M_i$ to be a multiple of $M$. In turn, since there are finitely many $G$-orbits of vertices in $\ov X$, if $\N$ is deep enough then $H_{[v]}$ admits the above relative hyperbolic structure for every $v\in \ov X^{(0)}$.
    
    Now, since peripherals are quotients of virtually cyclic subgroups, $H_{[v]}$ is hyperbolic. Moreover, for every $w\in W$ such that $[w]\in [W]$, we have that $\p_{[v]}(Z_{[w]})\le Q_{[w]}$, so Item (2) in the statement follows from the fact that conjugates of different peripherals have finite intersection in a relatively hyperbolic group. Finally, every infinite $\p_{[v]}(Z_{[w']})$ is conjugate to $\p_{[v]}(Z_{[w]})$ for some $[w]\in [W]$, and the latter is a finite-index subgroup of $Q_{[w]}$. Item (3) then comes from the fact that conjugates of peripherals in a relatively hyperbolic group are quasiconvex.
\end{proof}

\subsubsection{Absence of hidden symmetries}
For later use, we also point out that cyclic directions in the kernel have no \emph{hidden symmetries}, in the following sense.

\begin{defn}[No hidden symmetries]\label{defn:hiddensymm_general}
    Let $0\to \Z\to G\xrightarrow[]{\pi} H\to 1$ be a group extension, with $H$ hyperbolic, and let $C\le G$ be a cyclic subgroup. We say that $C$ has \emph{no hidden symmetries} if $\pi(C)$ is infinite and $C$ contains a finite-index subgroup which is normal in $\pi^{-1}(K_H(\pi(C)))$. We say that an element $g\in G$ has no hidden symmetries if $\langle g\rangle$ has no hidden symmetries.
\end{defn}

\begin{lemma}\label{lem:nohidden_for_quot}
	The following holds if $\N$ is deep enough. For any two adjacent vertices $v,w\in \ov X^{(0)}-G\B$ with infinite cyclic directions, $Z_{[w]}$ has no hidden symmetries in $\Stab{G/\N}{[v]}$, in the sense of Definition~\ref{defn:hiddensymm_general}.
\end{lemma}

\begin{proof}
Let $D_0$ be such that $\N=\langle \langle M_xZ_x\rangle\rangle_{x\in \B}$ satisfies all properties from Subsection~\ref{subsec:SCPG} whenever each $M_x$ is a non-trivial multiple of $D_0$. Furthermore, by Corollary~\ref{cor:Nv_free_prod} applied to the composite rotating family on the whole $\ov X^{(0)}$, there exists $D_1$ such that, for every $v\in \ov X^{(0)}$, $R_v\coloneq\langle M_yZ_y\rangle_{y\in \link_{\ov X}(v)}$ is a (possibly infinitely generated) free group whenever each $M_y$ is a non-trivial multiple of $D_1$ (notice that $R_v$ does not necessarily coincide with $\N_v$, as $\B$ might not contain a representative for every $G$-orbit). Then let $\N=\langle \langle M_xZ_x\rangle\rangle_{x\in \B}$, where each $M_x$ is a multiple of $D_0D_1$. 

Now let $v,w\in \ov X^{(0)}$ be adjacent vertices, and let $\ov g\in \p_{[v]}^{-1}\left(K_{H_{[v]}}(\p_{[v]}(Z_{[w]}))\right)$. We want to show that $\ov g$ normalises a finite-index subgroup of $Z_{[w]}$. Since $\ov g$ projects to $K_{H_{[v]}}(\p_{[v]}(Z_{[w]}))$ in $H_{[v]}$, it must normalise a finite-index subgroup of $\p_{[v]}(Z_{[w]}$, so there exists $m\in \mathbb N-\{0\}$, $\varepsilon\in\{\pm 1\}$, and $k\in \mathbb Z$ such that
$$z_{\ov g[w]}^{mD_1}=\ov gz_{[w]}^{mD_1}\ov g^{-1}=z_{[w]}^{\varepsilon mD_1}z_{[v]}^k.$$
If $k=0$ we are done, so suppose otherwise. Let $g\in \Stab{G}{v}$ project to $\ov g$. By the above equality, there exists $n\in \N$ such that
$$z_{gw}^{mD_1}=z_{w}^{\varepsilon mD_1}z_{v}^k n.$$
In particular $n\in \N\cap \Stab{G}{v}=\N_v$. We can rewrite the above equality as
$$z_v^k=z_{w}^{-\varepsilon 2mD_1}z_{gw}^{2mD_1}n^{-1}.$$
By our choice of how deep $\N$ is, the right hand side belongs to $R_v$, which is a free group. Now, $z_v^2$ commutes with $z_x$ for every $x\in \link_{\ov X}(v)$, so $\langle z_v^{2k}\rangle$, which is an infinite cyclic group because $k\neq 0$, must belong to the centre of $R_v$. The latter is trivial unless $R_v\cong \Z$, which happens exactly when $\link_{\ov X}(v)$ consists of a single vertex. In the latter case we would have that $w=gw$, and therefore
$$z_{[w]}^{mD_1}=z_{\ov g[w]}^{mD_1}=\ov gz_{[w]}^{mD_1}\ov g^{-1},$$
proving that $\ov g$ normalises $\langle z_{[w]}^{mD_1}\rangle$.
\end{proof}

\section{A short structure for the quotient}\label{sec:quotient_is_short}
The main aim of this paper is to prove the following theorem, which roughly states that the class of short HHG is stable under taking quotients by large enough cyclic directions. Recall Notation \ref{notation:kernel}, where we defined normal subgroups of the form $\N=\langle M_iZ_{s_i}\rangle$, that is, generated by cyclic subgroups of the cyclic directions of a short HHG.

\begin{thm}\label{thm:quotient_short_HHG}
    Let $(G, \ov X, \W)$ be a short HHG, as in Definition~\ref{defn:short_HHG}, and let $\N$ be the normal subgroup as in Notation~\ref{notation:kernel}. If $\N$ is deep enough, then $(G/\N, \ov X/\N, \W/\N)$  is a short HHG, where the cyclic direction associated to each $[v]\in(\ov X/\N)^{(0)}$ is (a finite-index subgroup of) $Z_{[v]}\coloneq Z_v/(Z_v\cap \N)$. Furthermore, if $v\not\in G\B$ then $Z_{[v]}\cong Z_v$.
\end{thm}

\begin{proof}[Outline of the proof]
    In Subsection~\ref{subsec:preparation} we modify the short HHG structure for $G$, in such a way that the kernel $\N$ acts with uniformly bounded orbits on the quasilines.
    
    We then set $\h X= X/\N$, which we prove to be the blowup of $\ov X/\N$ in Lemma~\ref{lem:blowup_for_quotient}, and we consider the $\h X$-graph $\h \W$, obtained from the (possibly non simplicial) graph $\W/\N$ by removing loop edges and double edges. Then we prove that $(\h X,\h \W)$ is a combinatorial HHS. Here is where we verify each axiom from Definition~\ref{defn:combinatorial_HHS}:
    \begin{itemize}
        \item Axioms~\eqref{item:chhs_flag} and~\eqref{item:chhs_join} both follow from the fact that $\h X$ is a blowup. 
        \item Axiom~\eqref{item:chhs_delta} is split between Subsections~\ref{subsubsec:hyp_link_quot} and~\ref{subsubsec:qi_emb_quot}.
        \item Axiom~\eqref{item:C_0=C} is Lemma~\ref{lem:edges_in_link_quot}.
    \end{itemize}
    In Subsection~\ref{rem:g_action_quot} we notice that the $G$-action on $X$ induces a $G/\N$-action on $\h X$, with finitely many orbits of links, and a geometric action on $\h W$. Therefore $(\h X, \h \W)$ is a combinatorial HHG structure for $G/\N$. Furthermore, in Lemma~\ref{lem:short_hhg_structure_for_quot} we check that it is actually a short HHG structure, with the required cyclic directions. The “furthermore” part of the statement is simply Lemma~\ref{lem:NV_cap_Z_v}.
\end{proof}

\subsection{Preparing the structure above}\label{subsec:preparation}
We first need to tweak the short HHG structure for $G$, to make it as “compatible” with $\N$ as possible. This will later allow us to define a combinatorial HHG structure for $G/\N$ by taking the quotient by $\N$ of the refined structure.
\par\medskip

First, fix a short HHG structure $(G,\ov X, \W_0)$, coming from the action on the combinatorial HHS $(X_0, \W_0)$. One can use the $G$-action on $\ov X^{+\W_0}$ to build a strong composite projection graph (SCPG), and let $\N_0$ be deep enough to satisfy all properties from Subsection~\ref{subsec:rotating_for_short} \emph{with respect to this SCPG}. Later we shall choose a deeper subgroup $\N\le \N_0$, so to avoid confusion we shall denote by $[v]_0\in (\ov X/\N_0)^{(0)}$ the $\N_0$-orbit of a vertex $v\in \ov X^{(0)}$. 

Now let $v\in \ov X^{(0)}-G\B$ with infinite cyclic direction, so that $Z_{[v]_0}\cong Z_v\cong \Z$ by Lemma~\ref{lem:NV_cap_Z_v}. By Lemma~\ref{lem:Hv/Nv_rel_hyp}, the quotient $H_{[v]_0}=\Stab{G/\N_0}{[v]_0}/Z_{[v]_0}$ is hyperbolic. Furthermore, if we fix a collection $[W]_0$ of $\Stab{G/\N_0}{[v]_0}$-orbits of vertices in $\link_{\ov X/\N_0}([v]_0)$ with infinite cyclic directions, then the subgroups $\{Z_{[w]_0}\}_{[w]_0\in [W]_0}$ have no hidden symmetries by Lemma~\ref{lem:nohidden_for_quot}, and they project to an independent collection in $H_{[w]_0}$ by Lemma~\ref{lem:Hv/Nv_rel_hyp}. Hence, if we denote the centraliser of $Z_{[v]_0}$ in $\Stab{G/\N_0}{[v]_0}$ by $E_{[v]_0}$, Corollary A.4 from the Appendix of \cite{Short_HHG:I}  produces a homogeneous quasimorphism $\psi_{[v]_0}\colon E_{[v]_0}\to \R$ which is unbounded on $Z_{[v]_0}$ and is trivial on $Z_{[u]_0}$ for every $u\in \link_{\ov X}(v)$. 

Next, applying Proposition~\ref{prop:short_is_squid} to $(G,\ov X, \W_0)$ produces blowup materials, and in particular, for every $v\in \ov X^{(0)}$, a subgroup $E_v\le \Stab{G}{v}$ centralising $Z_v$. Whenever $v\in \ov X^{(0)}-G\B$ has infinite cyclic direction, we can precompose the quasimorphism $\psi_{[v]_0}$ constructed above with the quotient projection $E_{v}\to E_{[v]_0}$. This gives a new homogeneous quasimorphism $\psi_{v}\colon E_{v}\to \mathbb{R}$, which is unbounded on $Z_{v}\cap E_{v}$ and trivial on $Z_w\cap E_{v}$ whenever $w\in\link_{\ov X}(v)$. Most importantly, by construction $\psi_{v}$ also vanishes on the intersection $\N_0\cap E_{v}$. If we replace the quasimorphism $\phi_v$ associated to $v$ by $\psi_v$, we get new blowup materials for $G$, and Theorem~\ref{thm:squidification} then yields \emph{another} short HHG structure $(X,\W)$ on $G$, with the same support graph $\ov X$. Using the $G$-action on $\ov X^{+\W}$, one can build \emph{another} strong composite projection graph, and let $\N\le \N_0$ be deep enough to satisfy all properties of Section~\ref{sec:large_dt} \emph{with respect to the new SCPG}.

\begin{rem}[Motivation of the above construction]\label{rem:weird_flex}
    Earlier we mentioned that we wanted the short HHG structure to be “compatible” with the quotient. The improvement when passing from $(X_0, \W_0)$ to $(X, \W)$ is that, for every $v\in \ov X^{(0)}$ with infinite cyclic direction, $\N\cap E_v$ now acts with uniformly bounded orbits on the quasiline $L_v$ associated to $v$, as its construction involved collapsing the “coarse level sets” of the quasimorphism $\psi_v$ (see \cite[Definition 3.12]{Short_HHG:I} for further details). This will be one of the key ingredients of the proof, in particular when we verify that augmented links in the quotient are hyperbolic (see Lemma~\ref{lem:l_[w]} below).
\end{rem}

\subsection{The candidate combinatorial structure}\label{subsec:X_W_quot}
From now on, let $\N$ be deep enough to satisfy all properties of Section~\ref{sec:large_dt}. We now move to the description of the combinatorial HHG structure for $G/\N$. We first recall from \cite{Short_HHG:I} that the underlying graph $X$ of the short HHG structure for $G$, granted by Theorem~\ref{thm:squidification}, is a blowup of $\ov X$, where, given any $v_i\in V$ and any $g\in G$, the vertex $v=gv_i$ is blown up to the cone over $(L_v)^{(0)}=g\Stab{G}{v_i}$. 

Now let $X/\N$ be the graph whose vertices and edges are $\N$-orbits of vertices and edges in $X$. Notice that $\ov X/\N$ is a full, simplicial subgraph of $X/\N$.

\begin{lemma}\label{lem:blowup_for_quotient}
   The graph $X/\N$ is isomorphic to the blowup $\h X$ of $\ov X/\N$ with respect to the family $\{L_{[v]}\coloneq L_v/(\N\cap\Stab{G}{v})\}_{[v]\in\ov X/\N^{(0)}}$. In particular, $\h X$ is simplicial.
\end{lemma}

\begin{proof} It suffices to notice that, for every $v\in \ov X^{(0)}$, $p\in (L_v)^{(0)}$, and $n\in \N$, $np$ is only adjacent to $nv$. In other words, each $\N$-translate of $p$ belongs to a single edge, and all these edges are in the same $\N$-orbit. Then by definition $X/\N$ is the desired blowup.
\end{proof}

Let $\h p\colon \h X\to \ov X/\N$ be the retraction mapping every cone to its apex. Given a simplex $\h \Delta$ of $\h X$, we shall call $\h p(\h \Delta)$ its support. We also denote a maximal simplex of $\h X$ by $\Delta([x],[y])$, where $[x]\in (L_{[v]})$, $[y]\in (L_{[w]})$, and $[v],[w]$ are $\ov X/\N$-adjacent. 

Given a simplex $\Delta$ of $X$, we will denote its projection to $\h X$ as $\h \Delta$, and we will say that $\Delta$ is a \emph{lift} of $\h \Delta$. In particular, a lift of a maximal simplex $\Delta([x],[y])$ is of the form $\Delta(x,y)$, where $x\in [x]$ and $y\in [y]$. Conversely:

\begin{lemma}\label{lem:maximal_to_maximal}
    Let $\Delta(x,y)$ be a maximal simplex of $X$. If $\N$ is deep enough, then $\h \Delta=\Delta([x],[y])$ is a maximal simplex of $\h X$.
\end{lemma}

\begin{proof}
    Let $v=p(x)$ and $w=p(y)$. As $\ov X$ is a $G$-invariant subgraph of $X$, $v$ and $x$ cannot be in the same $\N$-orbit, and similarly for all other pairs of vertices of $\Delta$ where one is in $\ov X$ and the other is not. Moreover, if $x$ and $y$ are in the same $\N$-orbit, then so are $v$ and $w$. Thus, it suffices to exclude that $v$ and $w$ are in the same $\N$-orbit, which is true as they must have different colours, and we assumed $\N$ to be deep enough to preserve each colour.
\end{proof}

As every simplex of $X$ can be completed to a maximal simplex, we get that:

\begin{cor}\label{cor:simplices_inject}
    Every simplex $\Delta$ of $X$ injects inside $\h X$. In particular, if $a,b\in X^{(0)}$ are $X$-adjacent vertices, then their projections $[a], [b]$ are distinct.
\end{cor}

\begin{defn}[$\h\W$-edges]\label{defn:hW}
    Let $\h W$ be the $\h X$-graph where two maximal simplices $\Delta([x],[y])$, $\Delta([x'],[y'])$ are adjacent if and only if they admit lifts $\Delta(x,y)$, $\Delta(x',y')$ which are adjacent in $\W$. 
\end{defn}

In other words, $\h \W$ is obtained from $\W/\N$ after collapsing double edges and loops, in order to get a simplicial graph.

\subsection{Checking the combinatorial HHG axioms}
We now check that the pair $(\h X, \h\W)$ is a combinatorial HHG structure for $G$. The recurring theme will be that one is often allowed to lift combinatorial configurations from $\h X^{+\h \W}$ to $X^{+\W}$. This way, all properties of $(\h X, \h \W)$ can be deduced from the corresponding statements about $(X,\W)$, which we already know to be a combinatorial HHG structure for $G$.
\par\medskip

For convenience, we recall the notion of a shortening pair, which will be a key ingredient for our lifting procedures:
\begin{cor}[of Lemma~\ref{lem:short_less_complex}]\label{cor:shortpair_for_short} The following holds if $\N$ is deep enough. 
Let $(\Lambda, N)$ be either $(\ov X^{+\W},\N)$ or $(\link_{\ov X}(v)^{+\W}, \N_v)$ for some $v\in\ov{X}^{(0)}$. There exists a well-order on $N$, called \emph{complexity}, such that the minimum element is the identity $1$. Moreover, for all $\gamma\in N-\{1\}$ and all $x\in \Lambda^{(0)}$, there exist a \emph{shortening pair} $(s,\gamma_s)$ (here $s\in G\B\cap \Lambda^{(0)}$ and $\gamma_s\in\Gamma_s$) so that $\gamma_s\gamma$ has strictly lower complexity than $\gamma$, and either
\begin{enumerate}
	 \item one between $x$ and $\gamma x$ is fixed by $\Gamma_s$, or
	 \item $\dist_s(x,\gamma x)\ge 100E$.	
\end{enumerate}
\end{cor}

\begin{rem}[Dependence on $\N$]\label{rem:dependence on N}
    From now on, we shall say that a quantity is \emph{depth-resistant} if it does not depend on the choice of powers $\{M_1,\ldots, M_k\}$ used to define $\N$, as in Notation~\ref{notation:kernel}, but only on the fact that each $M_i$ is a multiple of a large enough integer (that is, the quantity is the same for every deep enough $\N$). If we could prove that all constants in the proofs below were depth-resistant, then the combinatorial structure $(\h X, \h\W)$ for $G/\N$ would be \emph{uniformly} hierarchically hyperbolic, i.e. the HHS constant and the uniqueness function would be depth-resistant. This is not the case, but the only exception is that, whenever $w\in G\B$, the diameter of $L_{[w]}$ depends on the index of $\Gamma_w$ inside $Z_w$ (see Lemma~\ref{lem:l_[w]}).
\end{rem}

\subsubsection{Finite chains, and intersection of links}
As $\h X$ is a blowup of a triangle- and square-free graph and none of its connected components is a single point, one can argue exactly as in \cite[Subsection 3.3.2]{Short_HHG:I} to get the first and third requirements of Definition~\ref{defn:combinatorial_HHS}, with depth-resistant constants:

\begin{cor}[Verification of Definition~\ref{defn:combinatorial_HHS}.\eqref{item:chhs_flag}]\label{cor:finite_complexity_HHS_blow_up} Any chain $\link(\Delta_1)\subsetneq\dots\subsetneq\link(\Delta_i)$, where each $\Delta_j$ is a simplex of $X$, has length at most $25$.
\end{cor}

\begin{cor}[Verification of Definition~\ref{defn:combinatorial_HHS}.\eqref{item:chhs_join}]\label{lem:simplicial_wedge_property}
Let $\Sigma,\Delta$ be non-maximal simplices of $X$, and suppose that there exists a non-maximal simplex $\Gamma$ such that $[\Gamma]\nest[\Sigma]$, $[\Gamma]\nest[\Delta]$ and $\diam(\C (\Gamma))\ge 3$. Then there exists a non-maximal simplex $\Pi$ which extends $\Sigma$ such that $[\Pi]\nest[\Delta]$ and all $\Gamma$ as above satisfy $[\Gamma]\nest[\Pi]$.
\end{cor}

\subsubsection{Fullness of links}
\begin{lemma}[Verification of Definition~\ref{defn:combinatorial_HHS}.\eqref{item:C_0=C}]\label{lem:edges_in_link_quot}
Let $\h \Delta$ be a non-maximal simplex of $\h X$. Suppose that $[a],[b]\in\link(\h \Delta)$ are distinct, non-adjacent vertices which are contained in $\h \W$--adjacent maximal simplices $\h\Sigma_a,\h\Sigma_b$. Then there exist $\h \W$--adjacent maximal simplices $\h\Pi_a,\h\Pi_b$ of $\h X$ such that $\h \Delta\star [a]\subseteq\h\Pi_a$ and $\h \Delta\star [b]\subseteq\h\Pi_b$.
\end{lemma}

The following proof is prototypical of how to use Corollary~\ref{cor:shortpair_for_short}, together with the strong bounded geodesic image Lemma~\ref{lem:strong_bgi_for_short}, in order to lift combinatorial configurations from $\h X^{+\h \W}$ to $X^{+\W}$.

\begin{proof}
Suppose first that $\h p([a])=\h p([b])=[v]$. Then $[a],[b]\in (L_{[v]})^{(0)}$, as they are not $\h X$-adjacent. Let $a,b\in X$ be $\W$-adjacent lifts of $[a]$ and $[b]$; moreover, let $v=p(a)$ and $v'=p(b)$, which are $\W$-adjacent as well and in the same $\N$-orbit. By Corollary~\ref{cor:simplicial_x^w/n}, $\ov X^{+\W}/\N$ is simplicial, so $v$ must be equal to $v'$ or $\ov X^{+\W}/\N$ would have an edge with the same endpoints. This means that $a$ and $b$ belong to the same $(L_{v})^{(0)}$. Let $\Delta$ be a lift of $\h\Delta$ inside $\link_X(a)=\link_X(b)$. Then, as $(X,\W)$ is a combinatorial HHS, there exist $\W$-adjacent maximal simplices $\Pi_a, \Pi_b$ such that $\Delta\star a\subseteq\Pi_a$ and $\Delta\star b\subseteq\Pi_b$. Thus, the required simplices $\h\Pi_a$ and $\h\Pi_b$ are the quotient projections of $\Pi_a$ and $\Pi_b$.
\par\medskip
Thus suppose that $\h p([a])=[w]$ and $\h p([b])=[w']$ are different. In particular $[w]$ and $[w']$ are not $\ov{X}/\N$-adjacent, or $[a]$ and $[b]$ would be joined by an edge of $\h X$. This forces the support of $\h\Delta$ to be a single vertex $[v]$, which is $\ov X/\N$-adjacent to both $[w]$ and $[w']$. Let $[y]=\Sigma_a\cap (L_{[w]})^{(0)}$, so that $[a]$ is either $[y]$ or $[w]$, and $[y']=\Sigma_b\cap (L_{[w']})^{(0)}$. Now take lifts $y$ of $[y]$ and $y'$ of $[y']$ which are $\W$-adjacent. Let $w=p(y)$, $w'=p(y')$, and let $v\in \link_{\ov X}(w')$ be a lift of $[v]$. There exists $n\in \N$ such that $nw\in\link_{\ov X}(v)$. Hence the situation in $X^{+\W}$ is as in Figure~\eqref{fig:lift_fullness} below.

\begin{figure}[htp]
    \centering
    $$\begin{tikzcd}
        nw\ar[r,no head]& v\ar[r,no head]\ar[rd,no head]& w'\ar[r,no head, dashed]\ar[d, no head]\ar[rd,no head, dashed]& w\ar[d, no head]  \ar[lll, bend right, "n"]  \\
        & & y'\ar[r,no head, dashed]\ar[ru,no head, dashed]& y 
    \end{tikzcd}$$
    \caption{The full lines represent $X$-edges, while the dashed lines represent $\W$-edges.}
    \label{fig:lift_fullness}
\end{figure}

Our goal is to show that there exist $\W$-adjacent lifts of $[y]$ and $[y']$ which are also $X$-adjacent to some lift of $[v]$. Then we will lift $\h \Delta$ to some $\Delta$ supported on $v$, and we will conclude as above that $[y]$ and $[y']$ belong to $\h W$-adjacent maximal simplices containing $\h \Delta$.

By Corollary~\ref{cor:shortpair_for_short}, $\N$ is equipped with a well-order, called complexity, whose minimum element is the identity, so we proceed by induction on the complexity of $n$. If $n=1$ then $w=nw$, and both $y$ and $y'$ are already $X$-adjacent to $v$. Otherwise, let $(s, \gamma_s)$ be a shortening pair, as in Corollary~\ref{cor:shortpair_for_short}. Using $\gamma_s$, we want to replace some lifts, without breaking the configuration from Figure~\eqref{fig:lift_fullness}, in such a way that the two new lifts of $[w]$ are $w$ and $\gamma_s n w$. Then we shall conclude by induction, as $\gamma_s n$ has strictly lower complexity than $n$ by the defining properties of a shortening pair.

There are some cases to consider, depending on how $\gamma_s$ acts on our configuration.
\begin{itemize}
    \item If $\gamma_s$ fixes $w$, then we apply $\gamma_s$ to all lifts. As $\gamma_s$ acts by isometries on $X^{+\W}$, the configuration from Figure~\eqref{fig:lift_fullness} is preserved; moreover, every lift is mapped to a lift of the same point, as $\gamma_s\in \N$. Now $\gamma_s nw$ differs from $\gamma_s w=w$ by $\gamma_s n$.
    \item If $\gamma_s$ fixes $nw$ then $nw=\gamma_s n w$ already differs from $w$ by $\gamma_s n$. Then we leave the configuration untouched.
    \item Otherwise, Corollary~\ref{cor:shortpair_for_short}  implies that $\dist_s(w, nw)\ge 100E$. If neither $v$ nor $w'$ belonged to $\operatorname{Star}_{\ov X}(s)$, then $\dist_s(w, w')$, $\dist_s(w', v)$, and $\dist_s(v, nw)$ would all be well-defined, and by triangle inequality at least one of them would be greater than $33E$. Without loss of generality, say $\dist_s(w, w')\ge 33E$. However, this would contradict the strong bounded geodesic image Lemma~\ref{lem:strong_bgi_for_short}, because the edge $\{w, w'\}$ of $\ov X^{+\W}$ would be a geodesic disjoint from $\operatorname{Star}_{\ov X}(s)$. 
    \\
    Thus, suppose first that $\gamma_s$ fixes $v$. If we apply $\gamma_s$ to $nw$ then $v=\gamma_s v$ is again $\ov X$-adjacent to $\gamma_s n w$, so we can replace the lifts without breaking the configuration.
    \\
    If instead $\gamma_s$ fixes $w'$, we apply $\gamma_s$ to both $v$ and $nw$. This way $\gamma_s v$ is still $X$-adjacent to $w'$, and therefore to $y'$. 
\end{itemize}
The proof of Lemma~\ref{lem:edges_in_link_quot} is now complete.
\end{proof}

\subsubsection{Hyperbolicity of augmented links}\label{subsubsec:hyp_link_quot}
Let $\h \Delta$ be a simplex of $\h X$. We want to show that $\link(\h\Delta)^{+\W}$ is uniformly hyperbolic, in order to verify the first half of Definition~\ref{defn:combinatorial_HHS}.\eqref{item:chhs_delta}. As $\h X$ is a blowup graph, we only need to focus on the cases when $\link(\h \Delta)$ is unbounded, listed in Lemma~\ref{cor:bounded_links}.

Firstly, $\link(\emptyset)^{+\h\W}$ retracts onto $(\ov X/\N)^{+\h\W}$, which coincides with $\ov X^{+\W}/\N$ as, by construction, two vertices $[v],[w]\in (\ov X/\N)^{(0)}$ are $\h \W$-adjacent if and only if they have $\W$-adjacent lifts $v,w\in \ov X^{(0)}$. Then Corollary~\ref{cor:hyp_quotients_from_rotating} tells us that $\ov X^{+\W}/\N$ is $E$-hyperbolic. Thus we get:
\begin{lemma} $\link(\emptyset)^{+\h\W}$ is hyperbolic, and the hyperbolicity constant is depth-resistant.
\end{lemma}

We can argue similarly if $\h \Delta=\{([v],[x])\}$ is of edge-type. Indeed, $\link(\h\Delta)^{+\h\W}$ retracts onto $(\link_{\ov X/\N}([v]))^{+\h\W}=(\link_{\ov X}(v)^{+\W})/\N_v$, and the latter is $E$-hyperbolic by Corollary~\ref{cor:local_hyp_quotients_from_rotating}. Hence:

\begin{lemma}[Edge-type] If $\h \Delta$ is of edge-type then $\link(\h\Delta)^{+\h\W}$ is hyperbolic, and the hyperbolicity constant is depth-resistant.
\end{lemma}

We are now left with the triangle-type case, which we split into Lemmas~\ref{lem:l_[w]} and~\ref{lem:hyp_triangle_quotient}. For every $v\in\ov X^{(0)}$ set $L_{[v]}\coloneq L_v/(\N\cap\Stab{G}{v})$. The definition does not depend on the choice of $v\in [v]$, as $nL_v=L_{nv}$ for every $n\in \N$.
\begin{lemma}\label{lem:l_[w]}
    The following holds if $\N$ is deep enough. 
    \begin{itemize}
        \item If $Z_{[v]}$ is finite then $L_{[v]}$ is uniformly bounded, and the bound depends on $\N$.
        \item If instead $Z_{[v]}$ is infinite, then the quotient map $L_v\to L_{[v]}$ is a $\Stab{G}{v}$-equivariant quasi-isometry, whose constants are depth-resistant. As a consequence, $L_{[v]}$ is a quasiline on which $Z_{[v]}$ acts geometrically, while $Z_{[w]}$ acts with uniformly bounded orbits whenever $[w]\in\link_{\ov X/\N}([v])$.
    \end{itemize} 
\end{lemma}

\begin{proof}
    If $Z_v$ is finite then $L_v$ was already uniformly bounded. Thus assume that $Z_v$ is infinite, so that $L_v$ is a quasiline on which $Z_v$ acts geometrically. If $Z_{v}\cap \N\neq\{0\}$, then $L_{[v]}$ is bounded, and the bound is uniform as there are finitely many $G$-orbits of vertices in $\ov X$. Thus suppose instead that $Z_{v}\cap \N=\{0\}$. Recall that, in Subsection~\ref{subsec:preparation}, we constructed a quasimorphism $\psi_v\colon E_v\to \mathbb{R}$ which is trivial on $\N\cap E_v$. As each $E_v$ has finite index in $\Stab{G}{v}$, we can assume that $\N$ is deep enough that every $\Gamma_w$ is contained in $E_v$ whenever $w\in\operatorname{Star}_{\ov X}(v)\cap G\B$, so that $\N\cap \Stab{G}{v}=\N\cap E_v$. Now, by Remark~\ref{rem:weird_flex}, every subgroup of $E_v$ on which $\psi_v$ vanishes (such as $\N\cap \Stab{G}{v}$) acts with uniformly bounded orbits on $L_v$. Thus, the quotient map $L_v\to L_{[v]}$ is a $\Stab{G}{v}$-equivariant quasi-isometry, whose constants are depth-resistant.
\end{proof}

\begin{lemma}[Triangle-type]\label{lem:hyp_triangle_quotient} The following holds if $\N$ is deep enough. Let $\h \Delta=\{([v],[x]), ([w])\}$ be of triangle-type. Then there is a $\Stab{G}{[w]}$-equivariant quasi-isometry $\link(\h\Delta)^{+\h\W}\to L_{[w]},$ whose constants are depth-resistant.
\end{lemma}

\begin{proof} Let  $\Delta=\{(v,x), (w)\}$ be any lift of $\h \Delta$. By \cite[Lemma 3.34]{Short_HHG:I}, for every $w\in \ov X^{(0)}$ there is a $\Stab{G}{w}$-equivariant $(K,K)$-quasi-isometry $L_w\to \link(\Delta)^{+\W}$, for some uniform constant $K\ge 0$. Moreover,  $L_{[w]}=L_w/(\N\cap\Stab{G}{w})$, so it suffices to show that $\link(\h\Delta)^{+\h\W}=(\link(\Delta)^{+\W})/(\N\cap\Stab{G}{w})$, as this shall imply the existence of a $\Stab{G}{[w]}$-equivariant $(K,K)$-quasi-isometry $\link(\h\Delta)^{+\h\W}\to L_{[w]}$. It is clear that the quotient projection of $\link(\Delta)$ is contained in $\link(\h\Delta)$, and that if $y, y'\in \link(\Delta)$ are $\W$-adjacent then their projections $[y], [y']$ are $\h W$-adjacent by construction. Conversely, let $[y],[y']\in \link(\h \Delta)$ be $\h\W$-adjacent. Lift $[y]$ to $y\in \link(\Delta)$, and lift $[y']$ to some $y'$ which is $\W$-adjacent to $y$. Let $n\in \N$ be such that $y'\in nw$. Then we must have that $nw=w$, or $w$ and $nw$ would be $\W$-adjacent and there would be an edge in the simplicial graph $\ov X/\N$ connecting $[w]$ to itself. Thus $y'\in\link(\Delta)$ as well, and we are done. 
\end{proof}

\subsubsection{Quasi-isometric embeddings}\label{subsubsec:qi_emb_quot}
We move on to show that the augmented link of a simplex $\h\Delta$ of $\h X$ is quasi-isometrically embedded in $Y_{\h\Delta}$, thus proving the second part of Axiom~\eqref{item:chhs_delta}. Again, we look at all possible shapes of $\link(\h\Delta)$, according to Lemma~\ref{cor:bounded_links}. If $\link(\h\Delta)$ has diameter at most $2$, or if $\h\Delta=\emptyset$, then clearly $\C(\h\Delta)$ is quasi-isometrically embedded in $Y_{\h\Delta}$. Then we only need to deal with the following cases:
\begin{itemize}
    \item $\h\Delta=\{([v], [x])\}$ of edge-type, where $[v]$ has valence greater than one in $\ov X/\N$;
    \item $\h\Delta=\{([v], [x]), ([w])\}$ of triangle-type.
\end{itemize}

\begin{lemma}[Edge-type] The following holds if $\N$ is deep enough. Let $\h\Delta=\{([v], [x])\}$ be a simplex of edge-type, where $[v]$ has valence greater than one in $\ov X/\N$. Then there exists a coarsely Lipschitz retraction from $ Y_{\h \Delta}=\h p^{-1}(\ov X/\N-\{[v]\})^{+\h\W}$ to $\link(\h\Delta)^{+\h\W}$, whose constants are depth-resistant.
\end{lemma}

\begin{proof}
    The retraction $\h p$ maps $Y_{\h \Delta}$ onto $(\ov X/\N-\{[v]\})^{+\h\W}$ and $\link(\h\Delta)^{+\h\W}$ onto $\link_{\ov X/\N}([v])^{+\h\W}$, so it is enough to build a retraction $$\rho\colon(\ov X/\N-\{[v]\})^{+\h\W}\to \link_{\ov X/\N}([v])^{+\h\W}. $$
    For every $[u]\in(\ov X/\N-\{[v]\})^{(0)}$, pick any geodesic $\ov \gamma$ in $(\ov X/\N)^{+\h W}$ from $[u]$ to a closest point inside $\link_{\ov X/\N}([v])$, and let $\varrho([u])$ be the endpoint of this geodesic. Notice that $[v]$ does not belong to $\ov \gamma$. Indeed, suppose that this is not the case, and let $[t]$ be the vertex of $\ov \gamma$ which comes right before $[v]$. If $[t]$ and $[v]$ are $(\ov X/\N)$-adjacent, then we would contradict the fact that $\gamma$ connects $[u]$ to a closest point in $\link_{\ov X/\N}([v])$. If instead $[t]$ and $[v]$ are $\h\W$-adjacent, then $[t]$ is also $\h\W$-adjacent to some $[w]\in\link_{\ov X/\N}([v])$, and we could find a path from $[u]$ to $[w]$ which is shorter than $\ov\gamma$, again finding a contradiction.
\par\medskip

    Now we want to show that $\varrho$ is both coarsely well-defined and coarsely Lipschitz, for some depth-resistant constants. Let $[u],[u']\in(\ov X/\N-\{[v]\})^{(0)}$ be such that $\dist_{(\ov X/\N)^{+\h\W}}([u],[u'])\le 1$, and let $\ov\gamma $ (resp. $\ov \gamma'$) be a geodesic from $[u]$ (resp. $[u']$) to $\link_{\ov X/\N}([v])$. The configuration in Figure~\eqref{fig:qi_emb_edge} is therefore a geodesic pentagon inside $(\ov X/\N)^{+\h\W}$, which by Corollary~\ref{lem:lifting} we can lift to $\ov X^{+\W}$ if $\N$ is deep enough. 

    \begin{figure}[htp]
        \centering
        $$\begin{tikzcd}
            {[u]}\ar[dd, dash dot, no head] \ar[rrr, dashed, no head, "\ov\gamma"] &&& \varrho([u])\ar[dr, no head]&\\
            &&&& {[v]}\\
             {[u']}\ar[rrr, dashed, no head, "\ov\gamma'"] &&&\varrho([u'])\ar[ur, no head]&
        \end{tikzcd}$$
        \caption{The pentagon inside $(\ov X/\N)^{+\h\W}$, where the full lines represent $\ov X/\N$-edges, the dashed lines are geodesics of $(\ov X/\N)^{+\h\W}$, and the dash-and-dot line means that $\dist_{(\ov X/\N)^{+\h\W}}([u],[u'])\le 1$. The configuration lifts to a pentagon inside $\ov X^{+\W}$, with vertices $v, w, u, u', w'$.}
        \label{fig:qi_emb_edge}
    \end{figure}

    Let $\gamma$ (resp. $\gamma'$) be the lift of $\ov \gamma$ (resp. $\ov\gamma'$), with endpoints $u$ and $w\in\varrho([u])$ (resp. $u'$ and $w'\in\varrho([u'])$. Notice that neither $\gamma$ nor $\gamma'$ contain any lift of $[v]$, as pointed out above. Let $v$ be the lift of $[v]$ which is adjacent to both $w$ and $w'$.
    \par\medskip

    Now we claim that $\dist_{\link_{\ov X}(v)^{+\W}}(w, w')\le 6E$, which then implies that $\varrho([u])$ and $\varrho([u'])$ are $6E$-close in $\link_{\ov X/\N}([v])^{+\h\W}$. Indeed, if $\dist_{\link_{\ov X}(v)^{+\W}}(w, w')> 6E$, then by triangle inequality one between $\dist_{\link_{\ov X}(v)^{+\W}}(w, u)$, $\dist_{\link_{\ov X}(v)^{+\W}}(u, u')$, and $\dist_{\link_{\ov X}(v)^{+\W}}(u', w')$ is at least $2E$. But this contradicts the strong bounded geodesic image, Lemma~\ref{lem:strong_bgi_for_short_edge_type}, as neither $\gamma$ nor $\gamma'$ can pass through $v$.
\end{proof}

The proof in the triangle-type case is similar, but to build the retraction we need something more sophisticated than a geodesic, which we call an \emph{approach path}. Our approach paths are not the same as their homonyms from \cite[Definition 8.36]{BHMS}, although they play a similar role.
\begin{lemma}[Triangle-type]\label{lem:qi_emb_con_approach}
The following holds if $\N$ is deep enough. Let $\h\Delta=\{([v], [x]), ([w])\}$ be a simplex of triangle-type. There exists a coarsely Lipschitz retraction 
$$\varrho_{[w]}\colon Y_{\h \Delta}\to \link(\h\Delta)^{+\h\W},$$
whose constants are depth-resistant.
\end{lemma}

\begin{proof}
First, notice that $Y_{\h \Delta}=\left( (L_{[w]})^{(0)}\cup \h\p^{-1}(\ov X/\N-\operatorname{Star}_{\ov X/\N}([w])\right)^{+\h\W}.$ Now, for every $[y]\in (L_{[w]})^{(0)}$ set $\varrho_{[w]}([y])=[y]$. For every $[u]\in (\ov X/\N-\operatorname{Star}_{\ov X/\N}([w]))^{(0)}$ we define the value of the retraction on $\Squid([u])$ as follows. Pick a geodesic $\ov\gamma$ in $\ov X^{+\W}/\N$ from $[u]$ to $[w]$. Let $[r]\in\ov\gamma$ be the last point before $[w]$.
\begin{enumerate}
    \item Suppose first that $[r]\not\in \link_{\ov X/\N}([w])$, i.e. the last edge of $\ov\gamma$ is a $\h\W$-edge. Then choose any $[y]\in (L_{[w]})^{(0)}$ which is $\h\W$-adjacent to $[r]$, and set $\varrho_{[w]}(\Squid([u]))=[y]$. Moreover, let $\ov\lambda$ be the subpath of $\ov\gamma$ between $[u]$ and $[r]$. What we get is the configuration in Figure~\eqref{fig:approach_W}, which we call an \emph{approach path of type $\W$}.
    \item Suppose instead that $[r]\in \link_{\ov X/\N}([w])$, and let $[t]$ be the last point of $\ov\gamma$ before $[r]$. There exists $[a]\in \link_{\ov X/\N}([r])$ which is within distance $1$ from $[t]$ inside $\ov X^{+\W}/\N$ ($[a]$ might be $[t]$ itself, if the edge of $\ov \gamma$ between $[t]$ and $[r]$ comes from $\ov X/\N$). Notice that $[a]\neq [w]$, because otherwise $[t]$ would be $\ov X^{+\W}/\N$-adjacent to $[w]$ and this would contradict the fact that $\ov\gamma$ is a geodesic. Now pick a geodesic from $[a]$ to $[w]$ inside $\link_{\ov X/\N}([r])^{+\W}$, and let $[b]$ be the second-to-last point of such geodesic. Then $[b]$ is $\h\W$-adjacent to $[w]$, and therefore also to some $[y]\in (L_{[w]})^{(0)}$, so we set $\varrho_{[w]}(\Squid([u]))=[y]$. For further reference, let $\ov\eta_1$ be the subpath of $\ov\gamma$ from $[u]$ to $[t]$, and let $\ov\eta_2$ be the subgeodesic from $[a]$ to $[b]$ inside $\link_{\ov X/\N}([r])^{+\W}$. We get the configuration in Figure~\eqref{fig:approach_ov_X}, which we call an \emph{approach path of type $\ov X$}.
\end{enumerate}

\begin{figure}[htp]
    \centering
    $$\begin{tikzcd}
        {[w]} \ar[r, dashed, no head] \ar[d, no head]& {[r]} \ar[rrr, dash dot, no head, "\ov\lambda"]&&&{[u]}\\
        {[y]=\varrho_{[w]}([u])}\ar[ru, dashed, no head] & &&&
    \end{tikzcd}$$
    \caption{An approach path of type $\W$. Here the full arc is an edge of $\h X$, the dashed lines are $\h\W$-edges, and the dash-and-dot line is a geodesic inside $\ov X^{+\W}/\N$, which does not intersect $\operatorname{Star}_{\ov X/\N}([w])$.}
    \label{fig:approach_W}
\end{figure}

\begin{figure}[htp]
    \centering
    $$\begin{tikzcd}
        {[y]=\varrho_{[w]}([u])}\ar[rd, dashed, no head]
        &&{[r]}\ar[dll, no head] \ar[ll, no head] \ar[dl, no head]  \ar[dr, no head] \ar[drr, no head, dash dot]\\
        {[w]} \ar[r, dashed, no head] \ar[u, no head]& {[b]} \ar[rr, dash dot, red, no head, "\ov\eta_2"]&&{[a]}\ar[r, no head, dash dot]& {[t]}\ar[rr, dash dot, no head, "\ov\eta_1"] && {[u]}
    \end{tikzcd}$$
    \caption{An approach path of type $\ov X$. Here the full arcs represent edges of $\h X$; the dashed lines are $\h\W$-edges; the black dash-and-dot lines are geodesics inside $\ov X^{+\W}/\N$; and the red dash-and-dot line is a geodesic inside $\link_{\ov X/\N}([r])^{+\W}$. It will be relevant that $\ov \eta_1$ does not intersect $\operatorname{Star}_{\ov X/\N}([w])$.}
    \label{fig:approach_ov_X}
\end{figure}

First, we prove that both types of approach paths lift to $X^{+\W}$, meaning that all vertices and geodesics involved in the definition admit lifts which are arranged in the same configuration.
\begin{claim}\label{claim:approach_W_lifts}
    An approach path of type $\W$ lifts to $X^{+\W}$.
\end{claim}

\begin{claimproof}[Proof of Claim~\ref{claim:approach_W_lifts}]
    Lift $\ov \gamma$ to a geodesic $\gamma\in \ov X^{+\W}$, with endpoints $u\in [u]$ and $r\in [r]$. Then there exists $y\in [y]$ which is $\W$-adjacent to $r$, by how $\h\W$-edges are defined, and set $w=p(y)$. Then the configuration is as in Figure~\eqref{fig:approach_W} (notice that $w$ and $r$ are not $\ov X$-adjacent, or their projections $[w]$ and $[r]$ would be adjacent as well).
\end{claimproof}

\begin{claim}\label{claim:approach_ov_X_lifts}
    An approach path of type $\ov X$ lifts to $X^{+\W}$. 
\end{claim}

\begin{claimproof}[Proof of Claim~\ref{claim:approach_ov_X_lifts}]
    Let $\eta_1\subset \ov X^{+\W}$ be a lift of $\ov \eta_1$, with endpoints $u\in [u]$ and $t\in [t]$. By Corollary~\ref{lem:lifting}, we can lift the triangle with vertices $[r], [a], [t]$ to a triangle in $\ov X^{+\W}$ with vertices $r\in [r], a\in [a], t'\in [t]$, and up to the action of $\N$ we can assume that $t'=t$. Now, by Lemma~\ref{lem:lift_geodesics_for_short} we can lift $\ov\eta_2$ to a geodesic inside $\link_{\ov X}(r)$, with endpoints $a\in [a]$ and $b\in [b]$. Let $y\in [y]$ be $\W$-adjacent to $b$, let $w=p(y)$, and let $r'=nr\in \link_{\ov X}(w)$ be a lift of $[r]$, for some $n\in \N$. The whole configuration is as in Figure~\eqref{fig:lift_approach_ov_X}. 
    \begin{figure}[htp]
    \centering
    $$\begin{tikzcd}
        r'=nr\ar[r, no head] \ar[dr, no head] &{y}\ar[rd, dashed, no head]
        &&{r}\ar[lll, bend right, "n"] \ar[dl, no head]  \ar[dr, no head] \ar[drr, no head, dash dot]\\
        &{w} \ar[r, dashed, no head] \ar[u, no head]& {b} \ar[rr, dash dot, red, no head, "\eta_2"]&&{a}\ar[r, no head, dash dot]& {t}\ar[rr, dash dot, no head, "\eta_1"] && {u}
    \end{tikzcd}$$
    \caption{The (unclosed) lift of an approach path of type $\ov X$.}
    \label{fig:lift_approach_ov_X}
\end{figure}

    If $n=1$, then $r=r'$, and the configuration we get is a lift of the subgraph in Figure~\eqref{fig:approach_ov_X}. Otherwise, we proceed by induction on the complexity of $n$. Let $(s, \gamma_s)$ be a shortening pair, as in Corollary~\ref{cor:shortpair_for_short}. We want to replace some lifts from Figure~\eqref{fig:approach_ov_X}, without breaking the configuration, in such a way that the new lifts of $[r]$ will differ by $\gamma_s n$, which has strictly less complexity than $n$. 
    \begin{itemize}
        \item If $\gamma_s$ fixes $r$ then we apply $\gamma_s$ to all lifts from Figure~\eqref{fig:lift_approach_ov_X}.
        \item If $\gamma_s$ fixes $r'$ then $r'=\gamma_s nr$ already differs from $r$ by $\gamma_sn$, so we do nothing.
        \item Otherwise, $\dist_s(r, r')\ge 100E$. Arguing as in Lemma~\ref{lem:edges_in_link_quot}, the strong bounded geodesic image~\ref{lem:strong_bgi_for_short} tells us that $\gamma_s$ must fix either $b$ or $w$. If $\gamma_s w=w$, then we just replace $r'$ by $\gamma_s r'$, which is still $X$-adjacent to $w$ and therefore to $y$. If instead $\gamma_s b=b$, then we apply $\gamma_s$ to $r'$, $ w$, and $y$. Notice that $\gamma_s y$ is still $\W$-adjacent to $b=\gamma_s b$. In both cases, after the replacement $r$ and $\gamma_sr'$ differ by $\gamma_s n$.
    \end{itemize}
    Proceeding by induction, we can find lifts such that $r=r'$, as required.
\end{claimproof}

Finally, we shall prove with a single argument that the map $\varrho$ is both coarsely well-defined and coarsely Lipschitz with depth-resistant constants. Let $[u], [u']\in (\ov X/\N-\operatorname{Star}_{\ov X/\N}([w]))^{(0)}$ be such that $\dist_{(\ov X/\N)^{+\h\W}}([u], [u'])\le 1$, and consider two approach paths, one from $[u]$ to $[w]$ and one from $[u']$ to $[w]$. Let $u\in [u]$ and $u'\in [u']$ be such that $\dist_{\ov X^{+\W}}(u, u')\le 1$. Now lift both approach paths, starting from $u$ and $u'$, respectively, to get the configuration from Figure~\eqref{fig:2_lifted_approach_paths}, where both $w$ and $w'=nw$ belong to $[w]$ and $n\in \N$. To illustrate the process, we assume that the path starting at $u$ is of type $\ov X$, while the path starting at $u'$ is of type $\W$ (the two other cases are dealt with analogously). 

\begin{figure}[htp]
    \centering
    $$\begin{tikzcd}
        {y}\ar[rd, dashed, no head]
        &&{r}\ar[dll, no head] \ar[ll, no head] \ar[dl, no head]  \ar[dr, no head] \ar[drr, no head, dash dot]\\
        {w} \ar[r, dashed, no head] \ar[u, no head]\ar[d, bend right=90, "n"] & {b} \ar[rr, dash dot, red, no head, "\eta_2"]&&{a}\ar[r, no head, dash dot]& {t}\ar[rr, dash dot, no head, "\eta_1", blue] && {u}\ar[d, dash dot, no head, blue]\\
        {w'=n w} \ar[r, dashed, no head] \ar[d, no head]& {r'} \ar[rrrrr, dash dot, no head, "\lambda", blue]&&&&& u'\\
        {y'}\ar[ru, dashed, no head]
    \end{tikzcd}$$
    \caption{Two lifts of approach paths starting at $\ov X^{+\W}$-adjacent vertices. The blue path is a concatenation of three geodesics of $\ov X^{+\W}$.}
    \label{fig:2_lifted_approach_paths}
\end{figure}

We split the argument into two steps.
\par\medskip
\textbf{Step 1: Gluing $w$ to $w'$.}
We first prove that we can change the lifts, without breaking the configuration from Figure~\eqref{fig:2_lifted_approach_paths}, until $w=w'$. This will again be a combination of Corollary~\ref{cor:shortpair_for_short}  and the strong bounded geodesic image Lemma~\ref{lem:strong_bgi_for_short}.
\par\medskip
We proceed by induction on the complexity of $n$. If $n=1$ then we have nothing to prove; otherwise, let $(s, \gamma_s)$ be a shortening pair, as in Corollary~\ref{cor:shortpair_for_short}, so that $\gamma_s n$ has strictly less complexity than $n$. We want to replace the lifts in such a way that the two new lifts of $[w]$ differ by $\gamma_s n$, in order to conclude by induction. 

If $\gamma_s$ fixes $w$, we apply $\gamma_s$ to the whole diagram. If $\gamma_s$ fixes $w'$ then we do nothing, as $w'=\gamma_s w'$ already differs from $w$ by $\gamma_s n$. 

Otherwise, we have that $\dist_s(w, w')\ge 100E$. Now look at the blue path from Figure~\eqref{fig:2_lifted_approach_paths}, which is a concatenation of three geodesics of $\ov X^{+\W}$. If there is a point $z$ on the blue path such that $\dist_{\ov X}(z, s)\le 1$, then $\gamma_s z=z$ and we can apply $\gamma_s$ to $w'$, $y'$, and the subpath of the blue path between $z$ and $r'$. The new blue path is again a concatenation of three geodesics, since we did not change its projection to $\ov X/\N$ and therefore is still a lift of three geodesics. Then we can conclude by induction.
\par\medskip
We are left with the case when $\dist_s(w, w')\ge 100E$, but no point on the blue path belongs to $\operatorname{Star}_{\ov X}(s)$. In particular, by the strong bounded geodesic image Lemma~\ref{lem:strong_bgi_for_short}, the distances $\dist_s(w',r'),\dist_s(r',u'),\dist_s(u',u),\dist_s(u,t)$ are all well-defined and bounded above by $2E$. Moreover, the triangle inequality yields
$$\dist_s(w, t)\ge \dist_s(w, w')-\dist_s(w',r')-\dist_s(r',u')-\dist_s(u',u)-\dist_s(u,t) \ge 92E.$$

If $r\not\in \operatorname{Star}_{\ov X}(s)$, then by the triangle inequality, at least one of $\dist_s(w, r)$ and $\dist_s(r, t)$ would be at least $46E>2E$, again contradicting the strong bounded geodesic image Lemma~\ref{lem:strong_bgi_for_short}. Hence $\gamma_s$ must fix $r$. Furthermore, if no point on the red path $\eta_2$ belongs to $\operatorname{Star}_{\ov X}(s)$, then by triangle inequality one between $\dist_s(w, a)$, $\dist_s(a,b)$, and $\dist_s(b, t)$ would be greater than $30E>2E$. This would again contradict the strong bounded geodesic image Lemma~\ref{lem:strong_bgi_for_short}, either inside $\link_{\ov X}(r)^{+\W}$ (in the first two cases) or inside $\ov X^{+\W}$ (in the last case). Then let $k\in \eta_2$ be fixed by $\gamma_s$. If we apply $\gamma_s$ to everything beyond $r$ and $k$ (meaning, to $w'$, $y'$, the blue path, and the subpath of $\eta_2$ between $k$ and $b$), then we do not break the configuration, and we can conclude by induction.

\par\medskip
\textbf{Step 2: Bounding $\dist_{L_w}(y,y')$.} After the previous step, our configuration looks as in Figure~\eqref{fig:2_glued_approach_paths}:

\begin{figure}[htp]
    \centering
    $$\begin{tikzcd}
        {y}\ar[rd, dashed, no head]
        &&{r}\ar[dll, no head] \ar[ll, no head] \ar[dl, no head]  \ar[dr, no head] \ar[drr, no head, dash dot]\\
        {w}\ar[d, no head]\ar[dr, dashed, no head] \ar[r, dashed, no head] \ar[u, no head] & {b} \ar[rr, dash dot, red, no head, "\eta_2"]&&{a}\ar[r, no head, dash dot]& {t}\ar[rr, dash dot, no head, "\eta_1", blue] && {u}\ar[d, dash dot, no head, blue]\\
        {y'}  \ar[r, dashed, no head]& {r'} \ar[rrrrr, dash dot, no head, "\lambda", blue]&&&&& u'
    \end{tikzcd}$$
    \caption{Now the lifts of both approach paths terminate at $w$.}
    \label{fig:2_glued_approach_paths}
\end{figure}

Our final goal is to show that $\dist_{L_w}(y,y')$ is bounded in terms of $E$. This will then imply that $\dist_{L_{[w]}}(y,y')$ as well is bounded in terms of the depth-resistant constant $E$, concluding the proof.
\par\medskip

Firstly, we argue that $y'\subseteq \rho^{r'}_w$. Indeed, with our Notation~\ref{notation:rho}, 
$\rho^{r'}_w$ was defined as $\rho^{[\Delta']}_{[\Delta]}$, where $\Delta'=\{(s', x'), (r')\}$ is any simplex of triangle-type containing $r'$ but no point in $(L_{r'})^{(0)}$. Moreover, by Definition~\ref{defn:projections}, $\rho^{[\Delta']}_{[\Delta]}\coloneq p(\Sat(\Delta')\cap Y_\Delta)\supseteq p(r')$ is obtained by applying the coarse closest point projection $p\colon Y_{\Delta}\to \C(\Delta)$ to $r'$, which is $\W$-adjacent (that is, adjacent in $Y_{\Delta}$) to $y'$. Similarly, $y\subseteq \rho^{b}_w$, so it suffices to bound the distance $\dist_{w}(b,r')$. 

Now, notice that no point on the red path $\eta_2$ is $\ov X$-adjacent to $w$, as $\eta_2\subset\link_{\ov X}(r)$ and $\ov X$ is triangle-free. Moreover, recall that by construction no point on the blue path belongs to $\operatorname{Star}_{\ov X}(w)$. Thus both the red and the blue path have well-defined projections on $L_w$. Now, by the triangle inequality
$$\dist_{w}(b,r')\le \dist_{w}(b,a)+\dist_{w}(a,t)+\dist_{w}(t,u)+\dist_{w}(u, u')+\dist_{w}(u,r').$$
The first term is at most $2E$, by the strong bounded geodesic image Lemma~\ref{lem:strong_bgi_for_short}, applied inside $\link_{\ov X}(r)^{+\W}$. All other terms are at most $2E$ each, again by strong BGI applied inside $\ov X^{+\W}$. Thus $\dist_{w}(b,r')\le 10E$, and this concludes the proof of Lemma~\ref{lem:qi_emb_con_approach}.
\end{proof}

\subsubsection{$G/\N$-action}\label{rem:g_action_quot}
    The $G$-action on $X$ induces a $G/\N$ action on $\h X$, which has finitely many $G/\N$-orbits of links of simplices, and this action extends to $\h W$, as each edge of $\h W$ lifts to some edge of $\W$. 
    
    Moreover, if one fixes a generating set $S$ for $G$, there is a $G$-equivariant $(K,K)$-quasi-isometry $f\coloneq \Cay{G}{S}\to \W$, for some $K$ depending on $S$. By taking the quotient by $\N$, we get a $G/\N$-equivariant map $$\h f\coloneq \Cay{G/\N}{S\N}\to \h\W,$$
    which is again a $(K,K)$-quasi-isometry (notice that $K$ is depth-resistant). Then $(\h X, \h W)$ is a combinatorial HHG structure for $G/\N$, in view of the “moreover” part of Theorem~\ref{thm:hhs_links}.

\subsection{The quotient is short(er)}
To conclude the proof of Theorem \ref{thm:quotient_short_HHG}, we finally check that the combinatorial structure for $G/\N$, coming from the action on $(\h X,\h\W)$, is short:
\begin{lemma}\label{lem:short_hhg_structure_for_quot}
    $G/\N$ admits a short HHG structure $(G/\N,\ov X/\N)$, whose extensions are defined as in Lemma~\ref{lem:squid_for_quot_ext.1}.
\end{lemma}

\begin{proof}
    Axiom~\eqref{short_axiom:graph} is clear, as $\h X$ is a blowup of $\ov X/\N$ by Lemma~\ref{lem:blowup_for_quotient}, and by Lemma~\ref{lem:ovX/N_is_good} the latter is triangle- and square-free with no connected components which are points. Moreover, Axiom~\eqref{short_axiom:extension} is a combination of Lemmas~\ref{lem:squid_for_quot_ext.1} and~\ref{lem:Hv/Nv_rel_hyp}. Regarding Axiom~\eqref{short_axiom:cobounded}, the properties of the action on $\C\ell_{[v]}\cong L_{[v]}$ were proved in Lemmas~\ref{lem:l_[w]} and~\ref{lem:hyp_triangle_quotient}.
\end{proof}

\subsection{Residual hyperbolicity}\label{subsec:reshyp}
\begin{defn}
    We say that $\N=\langle\langle \Gamma_1,\ldots,\Gamma_k\rangle\rangle$ is a \emph{full} kernel if, for every $v\in \ov X^{(0)}$, there exists $i$ such that $\Gamma_i$ is conjugated into $Z_v$. 
\end{defn}

\begin{cor}\label{cor:deep_full_is_hyp}
    If $\N$ is a full, deep-enough kernel, then $G/\N$ is hyperbolic. Furthermore, if $G$ is not virtually cyclic and the main coordinate space $\C S$ is unbounded, then $G/\N$ is also non-elementary hyperbolic.
\end{cor}

\begin{proof}
    Since $\N$ is full, every cyclic direction $Z_{[v]}$ for the quotient is bounded, and therefore so is every $\C\ell_{[v]}$. Then Remark \ref{rem:unbounded_dom_short_hhg} shows that no two orthogonal domains in the structure have unbounded coordinate spaces, so $G/\N$ is hyperbolic by \cite[Corollary 2.16]{quasiflats}. 
    
    In the “furthermore” setting, \cite[Corollary 14.4]{HHS_I} implies that $G$ acts non-elementarily on $\C S$, and therefore on $\ov X^{+\W}$. Then Lemma~\ref{cor:hyp_quotients_from_rotating} tells us that $G/\N$ still acts non-elementarily on $\ov X^{+\W}/\N$, whenever $\N$ is deep enough.
\end{proof}

Recall that a group $G$ is \emph{fully residually $P$} for some property $P$ if, for every finite subset $F\subset G$, there exists a quotient $G\to \ov G$ where $F$ injects, and such that $\ov G$ enjoys $P$.

\begin{cor}\label{cor:fullres_hyp}
    A short HHG $G$ is fully residually hyperbolic. If moreover $G$ is not virtually cyclic and the main coordinate space $\C S$ is unbounded, then $G$ is fully residually non-elementary hyperbolic.
\end{cor}

\begin{proof}
    Fix a finite set $F\subset G-\{1\}$, and let $\N$ be a full kernel. By Corollary~\ref{cor:separating_filling}, we can choose $\N$ to be deep enough that $F$ injects in $G/\N$. Furthermore, since $\N$ is full, by Corollary~\ref{cor:deep_full_is_hyp} we have that $G/\N$ is hyperbolic (and non-elementary in the “moreover” setting).
\end{proof}

\begin{cor}\label{cor:if_hyp_resfin}
    If all hyperbolic groups are residually finite then all short HHG are residually finite.
\end{cor}

\section{Hopf property from central quotients}\label{sec:hopf}
Recall that a group $G$ is \emph{Hopfian}, or has the \emph{Hopf property}, if every surjective homomorphism $\phi\colon G\to G$ is an isomorphism. In this Section we develop some tools to study self-epimorphisms of short HHG, with the aim of then proving the Hopf property for most large hyperbolic type Artin groups. We expect that one could treat other short HHG similarly, and indeed in Subsection~\ref{subsec:HNN} we prove the Hopf property of certain HNN extensions of free groups.

\subsection{A criterion}
Firstly, we state a simple criterion for a group to be Hopfian.

\begin{defn}
\label{defn:enough}
    We say that a group $G$ has \emph{enough Hopfian quotients} if for every surjective homomorphism $\phi: G\to G$ and non-trivial $g_0\in G$ there exists a quotient $H$ of $G$, say with quotient map $q$, and $n\geq 1$ such that:
\begin{itemize}
 \item $q(g_0)\neq 1$,
 \item $H$ is Hopfian,
 \item $\phi^n$ induces a homomorphism $\psi:H\to H$, which is necessarily surjective.
\end{itemize}
\end{defn}

\begin{rem}
    Note that the third bullet holds if and only if $\phi^n(\ker(q))\le\ker(q)$.
\end{rem}

\begin{lemma}
\label{lem:enough}
    If $G$ has enough Hopfian quotients then it is Hopfian.
\end{lemma}

\begin{proof}
    Given a surjective homomorphism $\phi$ and $g_0\neq 1$ we have to argue that $\phi(g_0)$ is non-trivial. In the setting of Definition \ref{defn:enough}, this will follow if we show that $\phi^n(g_0)$ is non-trivial. But we have $\psi(q(g_0))\neq 1$, and therefore $\phi^n(g_0)\neq 1$, as required. 
\end{proof}

\begin{rem}\label{rem:almost_char_hopf_quotients}
   We stress that, despite the name "enough Hopfian quotients", the property from Definition~\ref{defn:enough} is stronger than being residually Hopfian, and indeed there are famous examples of residually finite, non-finitely-generated groups which do not satisfy the conclusion of Lemma~\ref{lem:enough}. The key point of Definition~\ref{defn:enough} is the third bullet, which requires the quotient $H$ to be ``almost characteristic'' in some sense.
\end{rem}

\subsection{Preliminary lemmas on central extensions}
\begin{lemma}
\label{lem:central_ext_peripheral}
    Let $G$ be a short HHG, and let $H\le G$ be a subgroup isomorphic to a $\mathbb Z$-central extension $1\to Z\to H\to K\to 1$.
    \begin{enumerate}
        \item If $K$ is infinite then $H$ is virtually contained in $\Stab{G}{v}$ for some $v\in \ov{X}^{(0)}$.
        \item If moreover $K$ is not virtually cyclic then $Z$ is virtually contained in $Z_v$.
    \end{enumerate}
\end{lemma}

\begin{proof}
    We first notice that, if $K$ is infinite, then the centraliser of any element of $H$ is not virtually cyclic, and therefore, $H$ cannot contain any element acting loxodromically on the top-level coordinate space for $G$ (since this action is acylindrical, by \cite[Corollary 14.4]{HHS_I}). 

    Now, by the Omnibus subgroup theorem \cite[Theorem 9.20]{DHS}, together with \cite[Theorem 7.1]{DHS}, there exists $h\in H$ and a finite collection $\mathcal V$ of pairwise orthogonal, unbounded domains, such that $H$ permutes $\mathcal V$, and a power of $h$ acts loxodromically on $\C V$ for every $V\in \mathcal V$. By the above discussion, $\mathcal V$ does not coincide with the $\nest$-maximal element. Hence, by our description of unbounded domains in a short HHG (Remark~\ref{rem:unbounded_dom_short_hhg}), we see that an index-two subgroup of $H$ must fix a domain $V\in \mathcal V$ of the form $\ell_v$ or $\U_v$. In the first case $H$ virtually fixes $v$, so (1) follows. In the second scenario, since $\C \U_v$ is unbounded, $v$ is the only vertex $w\in \ov X^{(0)}$ such that $\ell_w\orth \U_v$, and again we get (1).
\par\medskip

    Towards proving (2), let $H_0$ be a finite-index subgroup of $H$ contained in $\Stab{G}{v}$. The group $H_1=H_0/(H_0\cap Z_v)$ embeds in a hyperbolic group, and either $Z$ is virtually contained in $Z_v$, or $H_1$ has infinite centre. The latter can only happen if $H_1$ is virtually cyclic, but then $K$ would also be virtually cyclic, which is not possible under our assumption. Therefore $Z$ is virtually contained in $Z_v$, as required.
\end{proof}

We will also need the following support lemma. Recall that a group extension $1\to Z\to H\to H/Z\to 1$ is \emph{virtually trivial} if there exists a finite-index subgroup $H'\le H$ and a group retraction $H'\to H'\cap Z$, which we call a \emph{virtual retraction} from $H$ to $Z$.

\begin{lemma}
 \label{lem:non_trivial_quot_ext}
 Let $1\to Z\to H\to H/Z\to 1$ be a non-virtually-trivial extension, and let $\phi:H\to K$ be a surjective homomorphism whose kernel intersects $Z$ trivially. Then the extension $1\to \phi(Z)\to K\to K/\phi(Z)\to 1$ is non-virtually-trivial.
\end{lemma}

\begin{proof} Towards a contradiction, assume that the latter central extension is virtually trivial, so that there exist a finite-index subgroup $K'\le K$ and a group retraction $\rho\colon K'\to K'\cap \phi(Z)$. Let $\psi\colon \phi(Z)\to Z$ be an inverse of $\phi$. Then the composition $\psi\circ \rho\circ \phi$, defined on $\phi^{-1}(K')\le H$, is a virtual retraction from $H$ to $Z$, against the hypothesis that the first extension is non-virtually-trivial.
\end{proof}

\subsection{Certain relatively hyperbolic groups are Hopfian}
\begin{thm}
\label{thm:rel_hyp_hopf}
    Let $G$ be hyperbolic relative to $\Z$-central extensions of hyperbolic groups (including the case that $G$ itself is such an extension). Then $G$ is Hopfian.
\end{thm}

\begin{proof}
First, we note that if a group $H$ is hyperbolic relative to subgroups which are virtually a direct product of $\mathbb Z$ and a hyperbolic group (for short, virtual products), then it is Hopfian. Indeed, the peripheral subgroups are equationally Noetherian by \cite[Corollary 6.13]{ReinfeldWeidmann} for hyperbolic groups and \cite[Theorem E]{Val18} plus \cite[Theorem 1]{BMR97} for finite extensions of direct products of hyperbolic groups, so that $H$ is Hopfian by \cite[Corollary 3.14 and Theorem D]{Groves_Hull:eq_noetherian}.

We now proceed by induction on the number $k$ of peripheral subgroups which are not virtual products (for short, \emph{twisted}), the case $k=0$ being what we discussed above.

Suppose that the statement holds when there are at most $k$ twisted peripheral subgroups, and consider $G$ having $k+1$ twisted peripheral subgroups. Fix a self-epimorphism $\phi$ of $G$ and $g_0\neq 1$. We will use the criterion provided by Lemma \ref{lem:enough}, constructing a quotient $H$ which will be hyperbolic relative to $\Z$-central extensions of hyperbolic groups with $k$ twisted peripheral subgroups.
\par\medskip

We first claim that, up to passing to a power of $\phi$, there is some twisted peripheral $P$, with cyclic direction generated by $z_P$, and some positive integer $N_P$ such that $\phi(z_P^{N_P})$ is conjugated into $\langle z_P^{N_P}\rangle$. Indeed, let $P$ be any twisted peripheral. If $\phi^n(z_P)$ is a torsion element for some $n\in \mathbb{N}$ then we can take $N_P$ to be its order; thus we can assume that, for any twisted peripheral $P$ and any $n\in \mathbb{N}$, we have that $\phi^n(z_P)$ has infinite order. In this case, $\phi(P)$ is a $\mathbb Z$-central extension, and we claim that it cannot be an extension of a virtually cyclic group. Indeed, any such extension is virtually trivial, so this would contradict Lemma \ref{lem:non_trivial_quot_ext}. Since $G$ is a short HHG by \cite[Proposition 6.3]{Short_HHG:I}, we are now in a position to apply Lemma \ref{lem:central_ext_peripheral} to $\phi(P)$, and conclude that it is virtually contained in a conjugate $P'$ of some peripheral subgroup, which must be twisted itself (again as a consequence of Lemma \ref{lem:non_trivial_quot_ext}). Moreover, $P'\cap \phi(z_P) P' \phi(z_P)^{-1}$ contains the infinite subgroup $\langle\phi(z_P)\rangle \cap P'$, therefore $\phi(z_P)\in P'$ or we would contradict almost malnormality of peripheral subgroups. Notice also that the centraliser of $\phi(z_P)$ in $P'$, which contains $\phi(P)\cap P'$, cannot be virtually Abelian; hence $\phi(z_P)$ must be contained in the centre of $P'$. 

 Considering the directed graph with vertices the twisted peripherals and a directed edge from $P$ to $P'$ with $\phi(P)$ virtually contained in a conjugate of $P'$, we see that, up to passing to an iterated of $\phi$ (which is allowed by Lemma \ref{lem:enough}) there exists a twisted peripheral $P$ such that $\phi(P)$ is virtually contained in a conjugate of $P$. Moreover, $\phi(z_P)$ is virtually contained in the relevant conjugate of the centre of $P$.

 \par\medskip
 Therefore, by the relatively hyperbolic Dehn filling theorem, there exists a non-trivial multiple $N=N(g_0)$ of $N_P$ such that:
 \begin{itemize}
     \item the group $H=G/\langle\langle z_P^{N}\rangle\rangle$ is hyperbolic relative to
    \begin{enumerate}
        \item virtual products and $k$ twisted peripheral subgroups (coming from peripherals of $G$), and
        \item\label{item:hyp_peripheral} the hyperbolic group $P/\langle\langle z_P^{N}\rangle\rangle$.
    \end{enumerate}
    \item the image of $g_0$ in $H$ in non-trivial.
 \end{itemize}

Note that we can drop the subgroup from Item~\ref{item:hyp_peripheral} from the list of peripherals, so that $H$ is hyperbolic relative to virtual products and at most $k$ twisted peripheral subgroups, and it is therefore Hopfian by induction. Furthermore, the fact that $\phi(z_P^{N_P})$ is conjugate into $\langle z_P^{N_P}\rangle$ ensures that $\phi$ induces a homomorphism of $H$, so that we checked all conditions from Lemma \ref{lem:enough}, and the proof is complete.
\end{proof}

\subsection{The product region graph}\label{subsec:pr_graph}
\begin{defn}
    Let $G$ be a short HHG with support graph $\ov{X}$. Let $\ov{X}'$ be the full subgraph of $\ov{X}$ spanned by all vertices $v$ with $Z_v$ infinite. The \emph{product region graph} of $G$, denoted by $\mathcal{PR}(G)$, is the simplicial graph whose vertex set is $(\ov X')^{(0)}/G$, and where two vertices are adjacent if and only if they admit adjacent representatives in $\ov X'$.
\end{defn}

\begin{rem}
    The product region graph has the following interpretation. The vertices are conjugacy classes of vertex stabilisers (which are HHS product regions), and two vertices are adjacent if there exist conjugacy representatives that intersect along an edge group.
\end{rem}

\begin{defn}\label{defn:central_dir}
    A short HHG $(G, \ov X, \W)$ has \emph{central cyclic directions} if, for every $v\in\ov X^{(0)}$, either $Z_v$ is finite or it lies in the centre of $\Stab{G}{v}$ whenever $v\in\ov X^{(0)}$. In other words, whenever $Z_v$ is infinite, $\Stab{G}{v}$ is a $\Z$-central extension of a hyperbolic group.
\end{defn}

\begin{lemma}
\label{lem:rel_hyp_check}
    If a short HHG $(G, \ov X, \W)$ has discrete product region graph then it is hyperbolic relative to $\{\Stab{G}{v}\}_{v\in V}$, where $V$ is a collection of representatives for some $G$-orbits of vertices with infinite cyclic directions. In particular, if $G$ furthermore has central cyclic directions then it is Hopfian by Theorem~\ref{thm:rel_hyp_hopf}.
\end{lemma}

\begin{proof}
    The product region graph of $(G, \ov  X)$ being discrete is equivalent to no two vertices $v$ of the support graph $\ov  X$ with infinite $Z_v$ being connected to each other. Let $(X,\W)$ be a combinatorial HHG structure for $G$, where $X$ is a blowup of $\ov X$. With the aim of using \cite[Corollary 5.2]{russell}, we now modify the HHS structure, by removing various bounded domains. Namely, we only keep the following:
    \begin{itemize}
        \item The maximal domain $S$;
        \item For every $v\in \ov X^{(0)}$ of valence greater than one, the domain $\U_v=\link(\{(v,x)\})$, where $x$ is any point in $(L_v)^{(0)}$;
        \item For every $v\in \ov X^{(0)}$ for which $L_v$ is infinite, the domain $\ell_v$;
        \item For every $v\in \ov X^{(0)}$ for which $L_v$ is infinite, the domain $I_v\coloneq[\Sigma]$ corresponding to the link of the simplex $\Sigma=\{(v)\}$.  
    \end{itemize}
    Let $\frakS_{keep}\subset \frakS$ be the $G$-invariant subset containing the above domains. Notice that $I_v$ contains both $\ell_v$ and $\U_v$, and has no orthogonal domain in $\frakS_{keep}$. 
    
    Now, in view of Remark~\ref{rem:unbounded_dom_short_hhg}, every $U\in \frakS-\frakS_{keep}$ has bounded coordinate space. Hence, by inspection of the definition of a HHS (see e.g. \cite[Definition 1.1]{HHS_II}), removing these domains can only affect the existence of containers and the validity of the large link axiom, so we must check that both still hold in $\frakS_{keep}$.
\par\medskip
    \textbf{Containers}: By inspection of Remark~\ref{rem:unbounded_dom_short_hhg}, combined with the fact that the product region graph is discrete, the only pairs of orthogonal domains in $\frakS_{keep}$ are of the form $\U_v$ and $\ell_v$, for $v\in \ov X^{(0)}$ of valence greater than one and with infinite cyclic direction. Thus, whenever $U$ and $V$ are both nested in some $T\in \frakS_{keep}$, the container for $U$ inside $T$ is $V$, and vice versa.
\par\medskip
    \textbf{Large links}: Let $U\in \frakS_{keep}$ which is not $\nest$-minimal, and let $z,z'\in G$. We want to prove that, if one sets $N=2E\dist_U(z,z')+2E$, there exist $\{T_1,\ldots, T_{\lceil N\rceil}\}\subseteq \frakS_{keep}$ properly nested in $U$ and such that, whenever $V\in \frakS_{keep}$ is properly nested in $U$ and $\dist_V(z,z')> E$, then $V\nest T_i$ for some $i$. 
    \par\medskip

    We first notice that, if $U$ contains only finitely many domains of $\frakS_{keep}$ with unbounded coordinate spaces, then the large link axiom holds trivially (possibly after enlarging the HHS constant $E$). In particular, this happens if $U=I_v$, as it only contains $\ell_v$ and possibly $\U_v$.
\par\medskip

    Moreover, suppose that $U$ only contained $\nest$-minimal domains, already in $\frakS$ (this is the case if $U=\U_v$). Then the large link axiom for $\frakS$ produces a collection $\{T_1,\ldots, T_{\lceil N\rceil}\}\subseteq \frakS$, and one can simply intersect such collection with $\frakS_{keep}$ to get the required property.
\par\medskip

    The only case which is not covered by the above is when $U=S$ is the maximal domain. Let $\mathcal{T}=\{T_1,\ldots, T_{\lceil N\rceil}\}\subseteq \frakS$ the collection granted by the large link axiom for $\frakS$. If $\mathcal{T}\subseteq \frakS_{keep}$ we have nothing to prove; otherwise let $T\in \mathcal{T}-\frakS_{keep}$. If no $V\in\frakS_{keep}$ is properly nested in $T$ we can simply remove $T$ from the collection; otherwise we need to replace $T$ with some finite collection inside $\frakS_{keep}$. Suppose that $T=[\Delta]$ for some simplex $\Delta\subseteq X$. There are several cases to consider, according to the shape of $\Delta$ for which $[\Delta]\not\in \frakS_{keep}$. In what follows, let $\{v,w\}$ be an edge of $\ov X$ containing $\ov \Delta$, and let $x\in (L_v)^{(0)}$ and $y\in (L_w)^{(0)}$.
    \begin{itemize}
        \item Suppose that $\Delta=\{(v)\}$, where $v$ has finite cyclic direction. If there exists $V\in\frakS_{keep}$ which is nested in $T$, then $V\nest \U_v$. If $\U_v\in \frakS_{keep}$ then we replace $T$ by $\U_v$; otherwise $\link_{\ov X}(v)=\{w\}$, so that $V$ can only be $\ell_w$, and we replace $T$ by $\ell_w$. 
        \item Suppose that $\Delta=\{(x)\}$. Again, if there exists $V\in\frakS_{keep}$ which is nested in $T$, then $V\nest \U_v$, and we can argue as above.
        \item Suppose that $\Delta=\{(v,x)\}$, so that $T=\U_v$. As $T\not\in \frakS_{keep}$, we must have that $\link_{\ov X}(v)=\{w\}$, so the only $V\in\frakS_{keep}$ which is nested in $T$ is $V=\ell_w$. Thus we replace $T$ by $\ell_w$.
        \item Suppose that $\Delta=\{(v,w)\}$. The only unbounded domains which are nested in $T$ can be $\ell_v$ and $\ell_w$, so we can replace $T$ by $\{\ell_v, \ell_w\}\cap \frakS_{keep}$.
        \item Suppose that $\Delta=\{(v,y)\}$. The only unbounded domain which is nested in $T$ can be $\ell_v$, so we can replace $T$ by $\ell_v$.
        \item Finally, if $\Delta=\{(x,y)\}$ then no $V\in \frakS$ is nested in $T$, so we can simply remove the latter.
    \end{itemize}
    This concludes the verification of the large link axiom.
\par\medskip 
It is now readily seen that the HHG structure $(G,\frakS_{keep})$ has isolated orthogonality in the sense of \cite{russell}, specifically isolated by the set of domains $\{I_v\}$ as above. Then \cite[Theorem 4.3]{russell} states that $G$, equipped with a fixed word metric, is hyperbolic relative to the product regions $\{P_{I_v}\}_{v\in \ov X^{(0)}}$. Furthermore, since $G$ permutes the product regions, \cite[Proposition 5.1]{drutu_relhyp} yields that $G$ is hyperbolic  relative to a collection $\mathcal{Q}$ of subgroups, where for every $Q\in \mathcal Q$ there exists a unique $P_{I_v}$ within finite Hausdorff distance. Notice that $P_{I_v}$ is exactly $P_{\ell_v}$, so it lies within finite Hausdorff distance from a coset of a vertex stabiliser; hence there are $x\in G$ and a vertex $w\in \ov X^{(0)}$ with infinite cyclic direction such that the Hausdorff distance between $Q$ and $x\Stab{G}{w}$ is finite. In turn, $\dist_{Haus}(x\Stab{G}{w},\Stab{G}{xw})\le |x|$, where $|x|$ is the norm of $x$ in the fixed word metric; thus let $w'=xw$, and let $R=\dist_{Haus}(\Stab{G}{w'},Q)$.

We now claim that $\Stab{G}{w'}=Q$, thus completing the proof. To see this, let $K=Q\cap \Stab{G}{w'}$, and notice that, by \cite[Lemma 4.5]{Hruska_Wise}, there exists a constant $R'>0$ such that
$$Q\subseteq Q\cap N_R(\Stab{G}{w'})\subseteq N_{R'}(Q\cap \Stab{G}{w'})=N_{R'}(K)$$
so $K$ has finite index in $Q$. The same argument shows that $K$ has finite index in $\Stab{G}{w'}$.

Now, given $g\in \Stab{G}{w'}$, notice that $K\cap gKg^{-1}$ has finite index in $K$, as it is the intersection of two finite-index subgroups of $\Stab{G}{w'}$. Hence $K\cap gKg^{-1}$ has finite index in $Q$ as well, since $K$ has finite index in $Q$. Thus $Q\cap gQg^{-1}$, which contains $K\cap gKg^{-1}$, has finite index in $Q$, and since $Q$ is infinite and almost malnormal we must have that $g\in Q$. As $g$ was any element of $\Stab{G}{w'}$, this proves that $\Stab{G}{w'}\le Q$. 

For the reverse inclusion let $g\in Q$. The same argument shows that $\Stab{G}{w'}\cap g\Stab{G}{w'}g^{-1}=\Stab{G}{w'}\cap \Stab{G}{gw'}$ has finite-index in $\Stab{G}{w'}$. In particular some non-trivial subgroup of $Z_{gw'}$ fixes $w'$, which means that $\dist_{\ov X}(w',gw')\le 1$. But since $G$ has discrete product region graph and $Z_{gw'}$ is infinite, $gw'$ must actually coincide with $w'$, i.e., $g\in \Stab{G}{w'}$, as required. 
\end{proof}

\begin{defn}\label{defn:clean_int}
    A short HHG $(G, \ov X, \W)$ has \emph{clean intersections} if, for every $\ov X$-adjacent vertices $v,w$, the edge group $\Stab{G}{v}\cap \Stab{G}{w}$ coincides with $\langle Z_v, Z_w\rangle$. 
\end{defn}

\begin{defn}\label{defn:stable_pr}
    We say that a short HHG has \emph{stable product regions} if the following strengthening of Lemma \ref{lem:central_ext_peripheral} holds for $G$. 
    Let $H\le G$ be a subgroup isomorphic to a $\mathbb Z$-central extension $1\to Z\to H\to K\to 1$, and suppose that $H$ is virtually contained in some $\Stab{G}{v}$.
    \begin{enumerate}
        \item \label{item_stable_1} If $K$ is infinite then $\Stab{G}{v}$ actually contains $H$. If, in addition, $K$ is not virtually cyclic then $Z$ is contained in $Z_v$.
        \item \label{item_stable_2} There exists $I$, depending on $G$ only, such that if $K$ is finite then $H$ has an index-$\leq I$ subgroup contained in $\Stab{G}{v}$.
    \end{enumerate}
\end{defn}

\newcommand{\good}{always restrained}
\begin{defn}\label{defn:good_pr}
    Let $(G, \ov X,\W)$ be a short HHG, let $P=\Stab{G}{v}$ for some $v\in \ov X^{(0)}$, and let $\phi\colon G\to G$ be a homomorphism. $P$ is \emph{\good{}} with respect to $\phi$ if, for every $k\in\mathbb{N}$, either:
    \begin{itemize}
        \item $\phi^k(P)$ is a $\Z$-central extension of a non-elementary hyperbolic group, or
        \item $\phi^k(P)$ is virtually $\Z^2$, and virtually contained in an edge group.
    \end{itemize}
\end{defn}

\begin{lemma}\label{lem:quotient_for_good_stuff}
    Let $(G, \ov X,\W)$ be a colourable short HHG with stable product regions, central cyclic directions, and clean intersections. Let $g_0\in G-\{1\}$, let $\phi\colon G\to G$ be a homomorphism, and let $P_i=\Stab{G}{v_i}$ be \good{} vertex stabilisers with respect to $\phi$, for $i=1,\ldots, r$. Then there exists a kernel $\N$, as in Notation~\ref{notation:kernel}, such that:
    \begin{itemize}
        \item $\phi^M(\N)\le \N$ for some $M\in \mathbb{N}_{>0}$;
        \item $Z_{v_i}\cap \N\neq \{1\}$ for every $i=1,\ldots, r$;
        \item $g_0\not\in \N$;
        \item $G/\N$ is a colourable short HHG.
    \end{itemize}
\end{lemma}

\begin{rem}\label{rem:hopfian_from_good_stuff}
    Notice that, in the above Lemma, the product region graph of $G/\N$ injects in the graph obtained from $\mathcal{PR}(G)$ after removing the open stars of the vertices corresponding to $v_1,\ldots, v_r$. 
    
    Now, suppose that $G$ has central cyclic directions. In view of the above discussion, if removing all \good{} stabilisers makes $\mathcal{PR}(G)$ discrete, then $G/\N$ is Hopfian by Lemma \ref{lem:rel_hyp_check}. In other words, the quotient $G/\N$ satisfies all requirements of our criterion, Lemma \ref{lem:enough}.
\end{rem}

\begin{proof}[Proof of Lemma \ref{lem:quotient_for_good_stuff}] We proceed by induction on $r$, the base case $r=0$ being trivial. For the inductive step, let $\N'$ be a kernel satisfying the statement, for the collection $\{P_1,\ldots,P_{r-1}\}$. Up to replacing $\phi$ by a power, we can assume that $\phi(\N')\le \N'$. Set $P=P_r$, and choose a generator $z$ of $Z_{r}$. As $P$ is \good{}, we are in one of the three situations below.

\par\medskip
\textbf{(1)} Suppose first that $\phi^k(z^n)\in \N'$ for some $k,n\in\mathbb{N}$ and $n\neq 0$. Then set $\N=\langle\langle \N', z^n\rangle \rangle$, which is preserved by $\phi^k$. Working in the short HHG $G/\N'$, we see that, up to replacing $n$ by a non-trivial multiple, we can assume that $g_0\not\in \N$, and that $G/\N$ is again a colourable short HHG.
\par\medskip
\textbf{(2)} Suppose now that, for every $k\in \mathbb{N}$, $\phi^k(P)$ is a $\Z$-central extension of a hyperbolic group. By stability of product regions, this means that $\phi^k(P)$ is conjugated into some vertex stabiliser $Q_k$; moreover $\phi(Q_k)$ is again a $\Z$-central extension of a non-elementary hyperbolic group, as it contains $\phi^{k+1}(P)$, and the stability assumption implies that the $\phi$-image of the centre of $Q_k$ is conjugated into the centre of $Q_{k+1}$. Now, there are finitely many cyclic directions up to conjugation, so we can find $n\in\mathbb{N}_{>0}$ and a cyclic direction $Z'$ such that $\phi^n(\langle z\rangle)\le Z'$, and $\phi^n(Z')$ is conjugated inside $Z'$. Notice that neither $\langle z\rangle$ nor $Z'$ intersect $\N'$ trivially, as this case was covered by point \textbf{(1)}. Then set $\N=\langle\langle \N', z^t,tZ'\rangle\rangle$, which is preserved by $\phi^n$ for any choice of $t\in\mathbb{N}_{>0}$. Again, one can choose $t$ in such a way that $g_0\not\in \N$, and that $G/\N$ is a colourable short HHG.
\par\medskip
\textbf{(3)} Finally, suppose that there exists $k_0\in\mathbb{N}_{>0}$ such that, for every $k\ge k_0$, $\phi^k(P)$ is virtually $\Z^2$ and is virtually contained in some edge group $E_k$. Without loss of generality, we can replace $\phi$ by $\phi^{k_0}$ and assume that $k_0=1$. By definition, each $E_k$ is the intersection of two vertex stabilisers, so the stability of product regions implies that $\phi^k(P)$ is actually a finite-index subgroup of $E_k$. 

Notice that central cyclic directions and clean intersection imply that each $E_k$ is isomorphic to $\Z^2$, so that $\phi(E_k)$ is also Abelian. Moreover, $\phi(E_k)$ is a finite-index overgroup of $\phi^{k+1}(P)$, which is virtually $\Z^2$ and is contained inside $ E_{k+1}$; hence, by stability of product regions applied to the central extension $\phi(E_k)$, we must have that $\phi(E_k)\le E_{k+1}$.

As in point \textbf{(2)}, the existence of finitely many edge groups allows one to find $n\in\mathbb{N}_{>0}$ and an edge group $E'=\Stab{G}{v}\cap\Stab{G}{w}$ such that both $\phi^n(P)$ and $\phi^n(E')$ are conjugated inside $E'$. As $G$ has clean intersections, $E'=\langle Z_v, Z_w\rangle$, so for every $t\in \mathbb{N}_{>0}$ the subgroup $tE'=\langle tZ_v, tZ_w\rangle$ is preserved by $\phi^n$ (up to conjugation). Then set $\N=\langle\langle \N, z^t, tE'\rangle\rangle$, which is preserved by $\phi^n$ (notice that $\phi^n(z^t)$ is conjugated inside $tE'$). If, say, $Z_v$ already intersected $\N'$, we choose $t$ in such a way that $tZ_v\le \N'$, and similarly for $Z_w$. Then again a suitable choice of $t$ grants the required properties of the quotient.
\end{proof}

\subsection{Hopf property for admissible HNN extensions}\label{subsec:HNN}
Before focusing on Artin groups, we provide an easy example of how one can establish the Hopf property for certain short HHG, which will serve as a blueprint for many arguments in the next Section. The additional hypotheses we will assume on the short HHG rule out the difficulties that appear for Artin groups. We start with a general Lemma.

\begin{lemma}
\label{lem:abel_hopf}
    Let $G$ be a finitely generated group, and let $\phi\colon G\to G$ be a surjective homomorphism. Then $\phi$ induces an automorphism of the abelianisation $G^{ab}$ of $G$.
\end{lemma}

\begin{proof} The commutator subgroup, which is the kernel of the abelianisation, is preserved by any endomorphism, so $\phi$ descends to a self-epimorphism of $G^{ab}$. Furthermore, the latter is a finitely generated Abelian group, so it is Hopfian (as it is residually finite, for instance). 
\end{proof}

\begin{prop}\label{prop:example_hopf}
    Let $(G, \ov X, \W)$ be a short HHG with central cyclic directions. Suppose that:
    \begin{itemize}
        \item All vertex stabilisers are conjugate;
        \item The image of a vertex stabiliser in $G^{ab}$ has torsion-free rank at least $3$;
        \item The image of a cyclic direction in $G^{ab}$ has infinite order.
    \end{itemize}
    Then $G$ is Hopfian.
\end{prop}

\begin{proof}
     Let $\phi\colon G\to G$ be a surjective homomorphism, and let $g_0\in G-\{1\}$. In order to apply our criterion, Lemma~\ref{lem:enough}, we must produce a Hopfian quotient $H$ of $G$, such that the image of $g_0$ is non-trivial, and that some iterate of $\phi$ induces a self-epimorphism of $H$.
     
     Let $P=\Stab{G}{v}$ for some $v\in \ov X^{(0)}$, and let $\langle z\rangle$ be its cyclic direction. As $z$ has infinite order in $G^{ab}$, $\langle\phi(z)\rangle$ is infinite cyclic; furthermore, since the image of $P$ in $G^{ab}$ has torsion-free rank at least three, the same must be true for $\phi(P)$, which means that the latter must be a $\Z$-central extension of a non-elementary hyperbolic group. In turn, Lemma~\ref{lem:central_ext_peripheral} implies that $\langle\phi(z)\rangle$ is virtually contained in some cyclic direction, which is conjugate to $\langle z\rangle$ by assumption. Hence, up to post-composing $\phi$ by an inner automorphism, we can assume the existence of some $n\in \mathbb{N}_{>0}$ and $m\in \mathbb{Z}-\{0\}$ such that $\phi(z^n)=z^m$. 
     
     We now claim that $n$ divides $m$. Let $\ov\phi\colon G^{ab}\to G^{ab}$ be the induced map, and let $\ov z$ be the image of $z$ in the abelianisation. Since $\ov \phi$ maps the torsion subgroup to itself, and the latter does not contain $\ov z$, up to taking a further quotient we can assume that $G^{ab}\cong \Z^r$ for some $r\in\mathbb{N}_{>0}$. Choose a base $e_1,\ldots, e_r$ of $\Z^r$ such that $\ov z=ke_1$ for some $k\in\mathbb{N}_{>0}$. Then $mke_1=m\ov z=\ov\phi(n\ov z)\le nk \Z^r$, hence $n$ divides $m$.

    The above discussion implies that $\phi(z^n)\le \langle z^n\rangle$, so $\phi$ induces an automorphism of $H\coloneq G/\langle\langle z^{nK}\rangle\rangle$ for every $K\in\mathbb{N}$. Since there is a unique conjugacy class of cyclic directions, Corollary~\ref{cor:fullres_hyp} grants the existence of some $K$ such that $H$ is hyperbolic (hence Hopfian), and $g_0$ survives in $H$, as required. 
\end{proof}

Just to give a concrete example for the Proposition:

\begin{ex}\label{example:zxF}
    Let $G_0=\Z\times F_4$, where the centre is generated by $z\in \Z$ and the free factor has a basis $e_1,\ldots, e_4$. Let $H_1=\langle z, e_1\rangle$ and $H_2=\langle z, e_2\rangle$. Consider the HNN extension $G\colon G_0 *_\psi$, where $\psi\colon H_1\to H_2$ maps $z$ to $e_2$ and $e_1$ to $z$. The HNN extension is an admissible graph of groups, in the sense of \cite[Definition 3.1]{croke-Kleiner}, and therefore it admits a short HHG structure with support graph the Bass-Serre tree, as argued in \cite[Subsection 2.3.3]{Short_HHG:I}. In particular, all vertex stabilisers are conjugates of $G_0$. Furthermore, $G^{ab}$ is the free Abelian group generated by the images of $z, e_3, e_4,t$; in particular, the image of $G_0$ in the abelianisation is isomorphic to $\Z^3$, and the image of $z$ has infinite order. Hence $G$ satisfies the requirements of Proposition~\ref{prop:example_hopf}, and is therefore Hopfian.
\end{ex}

\section{Hopf property for Artin groups}\label{subsec:example_artin}

The goal of this section is to prove Theorem \ref{thm:XL_Hopfian} about the Hopf property for Artin groups. We will use our "Dehn filling" quotients to do so. We start with some basic definitions.

\begin{defn}
Let $\Gamma$ be a simplicial graph, with edge set $E$, and let $m\colon E\to \mathbb{N}_{\ge 2}$ be a labelling of the edges of $\Gamma$ with positive integers greater than or equal to $2$. Recall that the \emph{Artin group} $A_\Gamma$ is the group with the following presentation:
$$A_{\Gamma}=\langle \Gamma^{(0)}\,|\, \mathrm{prod}(a,b,m_{ab})=\mathrm{prod}(b,a,m_{ab}) \, \forall \{a,b\}\in E\rangle,$$
where $m_{ab}\coloneq m(\{a,b\})$ and $\mathrm{prod}(u,v,n)$ denotes the prefix  of length $n$ of the infinite alternating word $uvuvuv\dots$. 

An Artin group $A_\Gamma$ is of \emph{large type} if all edge labels are at least $3$, and it is of \emph{hyperbolic type} if, for every triangle with vertices $a,b,c$ inside $\Gamma$, the sum of the inverses of the edge labels is strictly less than one:
$$\frac{1}{m_{ab}}+\frac{1}{m_{bc}}+\frac{1}{m_{ac}}<1.$$
\end{defn}

\begin{defn}[Odd components]
Given a labelled graph $\Gamma$, we say that two vertices $a,b$ are in the same \emph{odd component} if there exists a combinatorial path between them, all whose edges have odd labels (i.e. if $a$ and $b$ are in the same connected component after we remove all even edges). The \emph{odd component graph}, denoted $\Gamma_{OC}$, is the simplicial graph whose vertices are odd components, and where two odd components $C,C'$ are adjacent if there exist vertices $a\in C$, $a'\in C'$ which are joined by an even edge.
\end{defn}

\begin{rem}
    A result of Paris \cite[Corollary 4.2]{paris_artin_groups} states that two standard generators are conjugate if and only if they lie in the same odd component, and this is why odd components will be relevant.
\end{rem}

\begin{rem}[Short HHG structure]\label{rem:example_Artingroup}
In \cite{ELTAG_HHS}, the authors produce a combinatorial HHG structure $(X,\W)$ for $A_\Gamma$ which, as noticed in \cite{Short_HHG:I}, is a short HHG structure. Here we point out some of its properties.
\begin{itemize}
    \item Fix a representative vertex for every odd component, and let $V$ be the union of such vertices. Let $\mathcal H$ be the collection of all cyclic subgroups generated by either a standard generator $s\in V$, or by the centre $z_{ab}$ of a standard dihedral subgroup $A_{ab}\coloneq\langle a, b\rangle$, for every two $\Gamma$-adjacent $a,b$. For every $H\in \mathcal H$, let $N(H)$ be its normaliser. Then $X$ is a blowup of the \emph{commutation graph} $Y$, whose vertices are the cosets of the $N(H)$, and two cosets $gN(H)$ and $hN(H')$ are adjacent if and only if $gHg^{-1}$ commutes with $hH'h^{-1}$. By \cite[Lemma 3.10]{ELTAG_HHS}, $Y$ is connected if the defining graph $\Gamma$ is. 
    \item By construction, $Y$ is bipartite, as two different conjugates of the standard generators never commute, nor do two different conjugates of centres of dihedral subgroups. This gives an $A_\Gamma$-invariant colouring of $Y$.
    \item In \cite[Lemmas 2.27 and 2.28]{ELTAG_HHS}, the authors describe $N(H)$ as follows: if $H=\langle a\rangle$ then $N(H)=C(a)$ is the centraliser of $a$, which is the direct product of $\langle a\rangle$ and a finitely generated free group; if $H=\langle z_{ab}\rangle$ then $N(H)=A_{ab}$ is the corresponding dihedral subgroup, which is a central extension with kernel $\langle z_{ab}\rangle$ and quotient a free product of two cyclic groups.
\end{itemize}
\end{rem}

For the next definition, recall that a \emph{leaf} of a simplicial graph is an edge with an endpoint of valence one, which we call the \emph{tip} of the leaf. A leaf in $\Gamma$ is said to be even or odd according to its edge label.

\begin{defn}[Hanging component]
    A \emph{hanging component} is an odd component $C$ which is a leaf of $\Gamma_{OC}$. A hanging component $C$ is \emph{broad} if $|C|>1$. A hanging component is a \emph{needle} if $C=\{v\}$ is a single vertex, and $v$ is a (necessarily even) leaf of $\Gamma$.
\end{defn}

We devote the rest of the Section to the proof of the following:
\begin{thm}\label{thm:XL_Hopfian}
    Let $A_\Gamma$ be an Artin group of large and hyperbolic type, such that every hanging component is either broad or a needle. 
    Then every surjective homomorphism $\phi\colon A_\Gamma\to A_\Gamma$ is an isomorphism.
\end{thm}

\begin{proof}[Outline of the proof]
    Firstly, it is enough to consider the case where $\Gamma$ is connected, since a free product of (finitely many, finitely generated) Hopfian groups is Hopfian by \cite[Theorem 1.1]{DeyNeumann_free_prod_Hopf}. Thus we are in the setting of Subsection~\ref{subsec:example_artin}, and $A_\Gamma$ is a short HHG. We can also assume that $|\Gamma|\ge 3$, as $\Z$ and dihedral Artin groups are known to be residually finite and therefore Hopfian. The proof is then split between Propositions~\ref{prop:odd} to~\ref{prop:2_odd_broad}, depending on the number of odd components of $\Gamma$.
\end{proof}

\begin{figure}[htp]
    \centering
    \includegraphics[width=\textwidth, alt={examples of hanging components}]{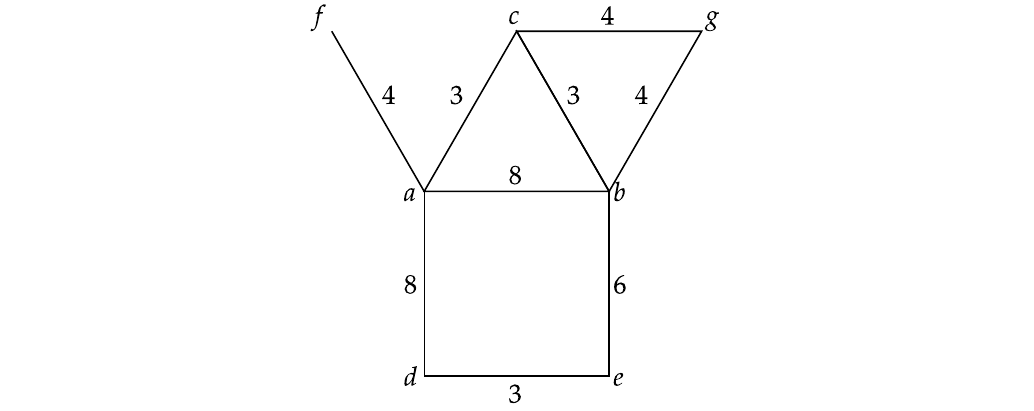}
    \caption{There are three hanging components in this graph: $\{d,e\}$ is broad, $\{f\}$ is a needle, and $\{g\}$ is what we forbid in Theorem \ref{thm:XL_Hopfian}.}
    \label{fig:needles_&_broad}
\end{figure}

\subsection{Pruning leaves}
Recall that, if $p\in\Gamma^{(0)}$ is the tip of an even leaf, then its centraliser is the $\Z^2$ subgroup generated by $p$ and the centre of the dihedral $A_{pq}$ corresponding to the leaf (see e.g. \cite[Corollary 34]{CMV:parabolics}). In particular $C(p)\le A_{pq}$, so the product region associated to $p$ is somewhat redundant. This is made clearer in the next Lemma.
\begin{lemma}
    Let $A_\Gamma$ be an Artin group of large hyperbolic type. Suppose that $\Gamma$ is connected and has at least three vertices. There exists a short HHG structure $(A_\Gamma, \ov{X})$, where $\ov X$ is the full subgraph of the commutation graph whose vertices are cosets of normalisers of
    \begin{itemize}
        \item centres of standard dihedral parabolics, or
        \item cyclic subgroups generated by standard generators \emph{which are not the tips of even leaves}.
    \end{itemize}
\end{lemma}
\begin{proof}
Let $Y$ be the commutation graph, let $p$ be the tip of an even leaf $\{p,q\}$ of $\Gamma$. For every $g\in A_\Gamma$, the coset $g N(\langle p\rangle)$ is only adjacent to $gA_{pq}$ in $Y$, and is therefore a vertex of valence one of the commutation graph. Now let $\ov X$ be the full, $A_\Gamma$-invariant subgraph of $Y$ defined above, which is still triangle- and square-free, as so is $Y$, and none of its connected components is a point. 

Now, by Proposition~\ref{prop:short_is_squid}, $A_\Gamma$ admits blowup materials with support graph $Y$, which we can restrict to $\ov X$ by forgetting the data associated to the cosets $gN(\langle p\rangle)$. It is easily seen that the restriction gives blowup materials for $A_\Gamma$, as all the requirements of Definition~\ref{defn:squid_material} are already satisfied in the bigger graph $Y$. The only non-trivial observation is that point \eqref{squid_material:big_papa} still holds. Indeed, $A_\Gamma$ is weakly hyperbolic relative to the collection
$$\{A_{ab}\}_{\{a,b\}\in \Gamma^{(1)}}\cup\{N(c)\}_{c\in\Gamma^{(0)}}.$$
However $N(\langle p\rangle)$ is contained inside $A_{pq}$, so $A_\Gamma$ is also weakly hyperbolic relative to 
$$\{A_{ab}\}_{\{a,b\}\in \Gamma^{(1)}}\cup\{N(c)\}_{c\in\Gamma^{(0)}, |\link_\Gamma(c)|>1}.$$
Then Theorem~\ref{thm:squidification} yields the required short HHG structure $(A_\Gamma, \ov X)$.
\end{proof}

\subsection{Some properties to check}
We now argue that the short structure $(A_\Gamma, \ov{X})$ defined above fits the framework of Subsection~\ref{subsec:pr_graph}. For the rest of the Section, by \emph{vertex stabiliser} we will always mean the stabiliser of a vertex of $\ov X$, with respect to the action of $A_\Gamma$. Firstly, by inspection of vertex stabilisers, we see that cyclic directions are central, Definition~\ref{defn:central_dir}. Next, an easy observation, which we prove for completeness:

\begin{lemma}
    $(A_\Gamma, \ov{X})$ has clean intersections, in the sense of Definition~\ref{defn:clean_int}.
\end{lemma}
\begin{proof}
    Let $\{v,w\}$ be an edge of $\ov X$. Up to the action of the group, we can assume that $v=A_{ab}$ is a standard dihedral and $w=C(a)$ is the centraliser of $a$. To prove that $C(a)\cap A_{ab}=\langle a, z_{ab}\rangle$, it is enough to notice that $(C(a)\cap A_{ab})/\langle a\rangle$ must centralise the non-trivial projection of $z_{ab}$ to the free group $C(a)/\langle a\rangle$.
\end{proof}

Finally, we move to stability of product regions:

\begin{lemma}
\label{lem:central_ext_peripheral_plus}
$(A_\Gamma, \ov X)$ has stable product regions, in the sense of Definition~\ref{defn:stable_pr}.
\end{lemma}

\begin{proof}
Let $H\le A_\Gamma$ be a subgroup isomorphic to a $\mathbb Z$-central extension of the form $1\to Z\to H\to K\to 1$, and suppose that $H$ is virtually contained in $\Stab{A_\Gamma}{v}$ for some $v\in \ov X^{(0)}$.
\par\medskip
\eqref{item_stable_1} First, we assume that $K$ is infinite, and we want to show that $\Stab{A_\Gamma}{v}$ actually contains $H$. Indeed, up to conjugation, $\Stab{A_\Gamma}{v}$ is either a standard dihedral $A_{ab}$ or the centraliser $C(a)$ of a standard generator $a$. In the first case, $H$ is contained in $\Stab{A_\Gamma}{v}$ because parabolics are root-closed by \cite[Theorem D]{CMV:parabolics}. In the second case, any element $h$ of $H$ is contained in some subgroup $H'$ of $H$ isomorphic to $\mathbb Z^2$, since $A_\Gamma$ is torsion-free (parabolics being root-closed implies this), and $H'$ has a finite-index subgroup $H'_0$ contained in $C(a)$. We have that $H'_0$ needs to contain a non-trivial power $a^k$ of $a$, for otherwise it would embed in the free group $C(a)/\langle a\rangle$. Since $H'$ is Abelian, it is contained in the centraliser $C(a^k)$. By \cite[Corollary 5.3]{paris_artin_groups}, this coincides with $C(a)$, so that $H'$, whence $h$, is contained in $C(a)$.

Now assume in addition that $K$ is not virtually cyclic, so that Lemma~\ref{lem:central_ext_peripheral} tells us that $Z$ is virtually contained in $Z_v$, and we claim that $Z\le Z_v$. If $\Stab{A_\Gamma}{v}=C(a)$ then $H_1\coloneq H/(\langle a\rangle \cap H)$ embeds in a free group and is not virtually cyclic (as otherwise $H$ would be virtually $\Z^2$). This means that $H_1$ must have trivial centre, that is, $Z$ must be contained in $\langle a\rangle$. If instead $\Stab{A_\Gamma}{v}=A_{ab}$, then \cite[Remark 3.6.(2)]{MV:centralisers} tells us that a proper root of $z_{ab}$ has cyclic centraliser, and then again we must have that $Z\le \langle z_{ab}\rangle$.

\par\medskip
\eqref{item_stable_2} Finally, suppose that $K$ is finite, and we claim that there exist $I\in\mathbb{N}_{>0}$, only depending on $\Gamma$, such that $H$ has an index-$\leq I$ subgroup contained in $\Stab{A_\Gamma}{v}$. As the ambient group $A_\Gamma$ is torsion-free, $H$ must be infinite cyclic, say generated by $g$ (see e.g. \cite[Lemma 3.2]{Macpherson_virtcyclic}). If $\Stab{A_\Gamma}{v}$ is conjugate to a dihedral parabolic, then $g\in \Stab{A_\Gamma}{v}$, because parabolics are root-closed; hence we are left to consider the case where $\Stab{A_\Gamma}{v}=C(a)$ is the centraliser of some standard generator $a$. Let $n\in\mathbb{N}_{>0}$ be such that $g^n$ is contained in $C(a)$, which in turn means that $a\in C(g^n)$. Notice that, if $C(g^n)$ coincides with $C(g)$, then $g\in C(a)$ and we are done; so suppose that this is not the case. Centralisers of elements of large-type Artin groups are analysed in detail in \cite[Section 3]{MV:centralisers}, see in particular \cite[Remark 3.6]{MV:centralisers}. An inspection of all the various possibilities reveals that, if $C(g^n)$ strictly contains $C(g)$, then $g$ lies in a conjugate of a dihedral subgroup $A_{bc}$, and $g^n$ belongs to the centre of such conjugate. For simplicity, we can assume that $g$ is contained in $A_{bc}$, as opposed to a conjugate. We now argue that in this case $C(g^n)=A_{bc}$ coincides with $C(g^I)$ for some uniform $I$, which suffices for our purposes.  We have that $g$ maps to a torsion element of $A_{bc}/Z(A_{bc})$, which has bounded torsion as it is a free product of cyclic groups. Therefore, a uniform power of $g$ maps to the trivial element of $A_{bc}/Z(A_{bc})$, that is, said uniform power is contained in the centre of $A_{bc}$ and its centraliser is $A_{bc}$, as required.
\end{proof}

\subsection{Proof of Theorem~\ref{thm:XL_Hopfian}}
We finally move to the core of the argument, which we split into three subcases, depending on whether $\Gamma$ has one, two, or at least three odd components. We start with an observation.

\begin{rem}[Abelianisation of an Artin group]\label{rem:abel_artin}
    Let $A_\Gamma$ be an Artin group. The abelianisation $A_\Gamma^{ab}$ of $A_\Gamma$ is the free Abelian group with one generator for every odd component, and the abelianisation map sends each standard generator to its component. In particular, both standard generators and centres of dihedrals have non-trivial image in the abelianisation, so Lemma \ref{lem:abel_hopf} implies that, for every epimorphism $\phi\colon A_\Gamma\to A_\Gamma$, their $\phi$-images must have infinite order.
    
Now, let $P$ be a vertex stabiliser, and we look at its image inside $A_\Gamma^{ab}$. If $P$ is conjugate to a standard dihedral subgroup $A_{bc}$,  then its image has rank 1 if $b$ and $c$ are in the same odd component, and 2 otherwise. If instead $P$ is conjugate to the centraliser $C(a)$ of some standard generator, then the rank of $C(a)$ in the abelianisation is
\begin{itemize}
\item 2 if $a$ belongs to a hanging component;
\item at least $3$ otherwise.
\end{itemize}
This is because $C(a)$ contains some conjugate of $z_{bc}$ for every dihedral $A_{bc}$ where $b$ is in the same odd component as $a$ and $c$ is not. 
\end{rem}

\subsubsection{One odd component}
\begin{prop}
\label{prop:odd}
    Let $A_\Gamma$ be a large Artin group of hyperbolic type. Assume further that $\Gamma$ is a connected graph on at least three vertices, and has a single odd component. Then $A_\Gamma$ is Hopfian.
\end{prop}

\begin{proof}
Let $\phi$ be an epimorphism, and let $g_0\in \ker\phi-\{1\}$. Our goal is to produce a Hopfian quotient $A_\gamma\to G$ where the image of $g_0$ is non-trivial, and such that $\phi$ induces a map on $G$. Then we will conclude by Lemma~\ref{lem:enough}. The crucial feature of this case is that all standard generators are conjugate. This means that, given any $a\in \Gamma^{(0)}$, removing the class of $a$ makes the product region graph $\mathcal{PR}(A_\gamma)$ discrete.

Since $\phi^n(a)$ has infinite order for every $n\in\mathbb{N}$, $\phi^n(C(a))$ is a $\mathbb Z$-central extension, say with base $B_n$. There are three possibilities, \textbf{A}, \textbf{B}, and \textbf{C} below, depending on the isomorphism type of $B_n$. The second scenario is furthermore split into two possibilities, \textbf{B1} and \textbf{B2}.
   \par\medskip
\textbf{A.} Suppose first that $B_n$ is always non-elementary. In this case $C(a)$ is \good{} with respect to $\phi$, and by Remark~\ref{rem:hopfian_from_good_stuff} we get a Hopfian quotient $G$ satisfying the requirements.
   \par\medskip
\textbf{B1.} Suppose now that $B_n$ is definitely virtually cyclic. By Lemma \ref{lem:central_ext_peripheral_plus}, $\phi^n(C(a))$ must be contained in some $P=\Stab{A_\Gamma}{v}$. If $\phi^n(C(a))$ is always virtually contained in an edge group then $C(a)$ is \good{}, and we conclude as above. Otherwise, up to replacing $\phi$ by an iterated, assume that $\phi(C(a))$ is virtually $\Z^2$ but is not contained in an edge group, so there exists a unique $P$ containing $\phi(C(a))$.
\par\medskip
Suppose first that $P$ is conjugate to $C(a)$. Up to composition with an inner automorphism, we can actually assume that $H\coloneq \phi(C(a))\subseteq C(a)$. We have $\phi(H)\le\phi(C(a))\le H$. For $\psi=\phi|_H\colon H\to H$ and $H_0$ a finite-index subgroup of $H$ isomorphic to $\mathbb Z^2$ let $H_1=\bigcap_{i\geq 0} \psi^{-i}(H_0)$. It is readily checked that $H_1$ is $\psi$-invariant, whence $\phi$-invariant. We claim that $H_1$ is a finite-index subgroup of $H_0$, which in turn implies that $H_1$ is isomorphic to $\mathbb Z^2$. This is because the index of each $\psi^{-i}(H_0)$ in $H$ is at most the index of $H_0$ in $H$, and since there are only finitely many subgroups of $H$ of index bounded by a given constant, the intersection defining $H_1$ is actually equal to a finite intersection of finite-index subgroups. Furthermore, $H_1$ needs to contain a non-trivial power $a^k$ (for otherwise it would embed in the hyperbolic group $C(a)/\langle a\rangle$); pick $k>0$ minimal. We can find $g$ such that $\{g,a^k\}$ is a basis of $H_1\cong\mathbb Z^2$. Now, every virtually cyclic subgroup of the free group $C(a)/\langle a\rangle$ is cyclic, so $g$ has no hidden symmetries in $C(a)$ by \cite[Remark 5.4]{Short_HHG:I}. Then by \cite[Proposition 5.2]{Short_HHG:I}, there exist $p\in \mathbb{N}_{>0}$ and a short HHG structure in which $g'\coloneq g^p$ spans a cyclic direction. 

 We have that $g^{kp}$ lies in $H_1$, and together with $a^{k^2p}$ it generates $kpH_1$. Since $\phi(H_1)\le H_1$, we have that $\phi(\langle\langle g^{Mkp},a^{Mk^2p}\rangle\rangle)\le \langle\langle g^{Mkp},a^{Mk^2p}\rangle\rangle$ for any integer $M$. By Theorem \ref{thm:quotient_short_HHG}, for a suitable $M$ the quotient $G\coloneq A_\Gamma/\langle\langle g^{Mkp},a^{Mk^2p}\rangle\rangle$ is a colourable short HHG, where the image of $g_0$ is non-trivial. Moreover, as $\mathcal{PR}(G)$ is obtained by removing the class of $a$ from $\mathcal{PR}(A_\Gamma)$, $G$ is Hopfian by Lemma~\ref{lem:rel_hyp_check}, so $G$ satisfies the requirements.
   \par\medskip
\textbf{B2.}
In the same setting as above, suppose now that $\phi(C(a))$ is conjugated into a (unique) dihedral parabolic $P$. Up to conjugation, we can assume that $P=A_{bc}$ for some $\Gamma$-adjacent generators $b,c$. We now argue that $\phi(A_\Gamma)$ is contained in $A_{bc}$, thus contradicting the surjectivity of $\phi$. Say that $C(a)$ is the stabiliser of the vertex $v\in \ov{X}$. For any vertex $w$ of $\ov{X}$ adjacent to $v$ we have that $P_w=\Stab{A_\Gamma}{w}$ is a conjugate of a dihedral group, and the centraliser of a generator $z_w$ of its centre. Since $z_w$ has infinite-order image in the abelianisation, it maps under $\phi$ to a non-trivial element of $A_{bc}$.

Notice that $\phi(z_w)$ is not conjugate to any power of either $b$ or $c$. If this was not true, say without loss of generality that $\phi(z_w)=b^k$ for some non-trivial $k\in\mathbb{Z}$. Then $\phi(a)$ would lie in the centraliser of $b^k$ inside $A_{bc}$, which is the edge group $\langle b, z_{bc}\rangle$. If $\phi(a^r)\in \langle b\rangle$ for some $r\ge 0$, then the whole $\phi(C(a))$ would lie in $\langle b, z_{bc}\rangle$, as every element of $\phi(C(a))$ must centralise $\phi(a)$. If instead $\phi(a)\not\in\langle b\rangle$, then $\langle \phi(a), \phi(z_w)\rangle\cong \Z^2$ would coarsely coincide with $\phi(C(a))$. In both cases, one would contradict the fact that $\phi(C(a))$ is not virtually contained in an edge group. 

As $\phi(z_w)$ is not conjugate to any power of either $b$ or $c$, by \cite[Remark 3.6]{MV:centralisers}) its centraliser is entirely contained in $A_{bc}$. Hence $\phi(P_w)$, which centralises $\phi(z_w)$, is contained in $A_{bc}$. Now, given a vertex $v'$ of $\ov X$ adjacent to $w$, $P_{v'}=\Stab{A_\Gamma}{v'}$ contains $z_w$, so $\phi(P_{v'})$ contains an element in $A_{bc}$ which is not conjugate to any power of either $b$ or $c$. Moreover, $P_{v'}$ is a conjugate of $C(a)$, and as such its image must be contained in a conjugate of $A_{bc}$. By \cite[Theorem A]{CMV:parabolics}, the intersection of $A_{bc}$ and one of its conjugates is either trivial, the whole $A_{bc}$, or is conjugate to either $\langle b\rangle$ or $\langle c\rangle$; hence we must have that in fact $\phi(P_{v'})$ is contained in $A_{bc}$. We can then proceed inductively on the distance in $\ov{X}$ from $a$, and as the support graph $\ov X$ is connected (see Subsection~\ref{subsec:example_artin}) we eventually get that $\phi$ maps every vertex stabiliser $Q$ into $A_{bc}$, as required.
\par\medskip
\textbf{C.} Suppose finally that $B_n$ is definitely finite. Up to replacing $\phi$ by a power, we can assume that $\phi(C(a))$ is virtually infinite cyclic. We first claim that, for every two $\Gamma$-adjacent vertices $c$ and $d$, $\phi(a)$ and $\phi(z_{cd})$ have a non-trivial common power. Indeed, first consider $b\in \link_\gamma(a)$. The fact that $\phi(C(a))$ is virtually $\Z$ ensures that $\phi(a)$ and $\phi(z_{ab})$ have a non-trivial common power. In turn, as $\phi(C(b))$ is conjugate to $\phi(C(a))$ and contains $\phi(z_{ab})$, the same must hold for $\phi(b)$ and $\phi(z_{ab})$. Iterating this procedure, we eventually get that, for each $z_{cd}$, there exist $N_{cd},M_{cd}$ such that $\phi(z_{cd}^{N_{cd}})=\phi(a^{M_{cd}})$, that is, $\phi(z_{cd}^{N_{cd}}a^{-M_{cd}})=1$.  We can in fact take multiples to ensure that all $M_{cd}$ coincide, say $M_{cd}=M$. For $\mathcal N_K=\langle\langle z_{cd}^{KN_{cd}}a^{-KM}\rangle\rangle$ we have $\phi(\mathcal N_K)\le\mathcal N_K$, so $\phi$ induces a homomorphism of $A(K)=A_\Gamma/\mathcal N_K$. We claim that we can choose $K$ so that $A(K)$ is a $\Z$-central extension of a hyperbolic group (hence Hopfian by Theorem \ref{thm:rel_hyp_hopf}) and the image of $g_0$ in it is non-trivial. If this is true then $G=A(K)$ satisfies all requirements.

In order to do so, we consider the auxiliary group $A'(K)=A_\Gamma/\langle\mathcal N_K, a^{KM}\rangle$. Since $\langle\mathcal N_K, a^{KM}\rangle=\langle\langle z_{cd}^{KN_{cd}}, a^{KM} \rangle\rangle$, by Theorem \ref{thm:quotient_short_HHG}, for suitable values of $K$, $A'(K)$ is a HHG; furthermore, from the description of its structure, it is clear that $A'(K)$ has bounded orthogonality, and is therefore a hyperbolic group by e.g. \cite[Corollary 2.14]{quasiflats}. By Corollary~\ref{cor:separating_filling} we can further arrange that the image of $g_0$ is non-trivial in $A'(K)$, and therefore also in its extension $A(K)$. We are left to prove that the natural projection $A(K)\to A'(K)$ is a $\Z$-central extension. By construction, the kernel is normally generated by the image of $a^{KM}$ in $A(K)$, so in turn it suffices to prove that said image commutes with a generating set of $A(K)$. The reason for this is that $A(K)$ is obtained from $A_\Gamma$ by imposing the relations $z_{cd}^{KN_{cd}}=a^{KM}$. As $z_{cd}$ commutes with $c$, the relations make $a^{KM}$ commute with all the standard generators, as required.
\end{proof}

\subsubsection{At least three odd components}
We consider now the general case, postponing the study of Artin groups with two odd components as it is more involved and reuses some techniques from this paragraph.

\begin{prop}\label{prop:3+_odd}
    Let $A_\Gamma$ be a large Artin group of hyperbolic type, where $\Gamma$ is a connected graph on at least three vertices. Suppose that $\Gamma$ has at least three odd components, and every hanging component is either broad or a needle. Then $A_\Gamma$ is Hopfian.
\end{prop}

\begin{proof}
As usual, given an epimorphism $\phi$, we want to find a collection of \good{} vertex stabilisers whose removal makes $\mathcal{PG}(A_\Gamma, \ov X)$ discrete, and then conclude by Lemma~\ref{lem:quotient_for_good_stuff}. Let $C_1,\ldots, C_k$ be the hanging components of $\Gamma$, and let $\Gamma_{core}$ be the subgraph of $\Gamma$ spanned by all other odd components. 
\par\medskip

In view of Remark~\ref{rem:abel_artin}, if $P$ is a vertex stabiliser for the action on $\ov X$, then its image in the abelianisation of $A_\Gamma$ has rank at least $3$ if and only if $P$ is conjugate to $C(a)$ for some $a\in \Gamma_{core}$. For any such $P$, $\phi(P)$ must be a $\Z$-central extension of a non-elementary hyperbolic group, as its projection to $A_\Gamma^{ab}$ must have rank at least 3; moreover $\phi(P)\le Q$ for some stabiliser $Q$, and the projection of $Q$ to $A_\Gamma^{ab}$ must have rank at least 3 as well. This, together with stability of product regions, implies that, for every $a\in \Gamma_{core}^{(0)}$, $\phi(a)$ is conjugated into $\langle b\rangle$ for some $b\in \Gamma_{core}^{(0)}$. Notice that the above argument also tells us that $C(a)$ is \good{} for every $a\in \Gamma_{core}^{(0)}$.
\par\medskip

Now consider the retraction
$$r\coloneq A_\Gamma\to A_{C_1}*\ldots*A_{C_k},$$
defined by mapping every generator in $\Gamma_{core}$ to the identity and all other generators to themselves. This map is well-defined, because any edge connecting $\Gamma_{core}$ to any hanging component is even; furthermore, $\phi$ induces a self-epimorphism $\ov\phi$ of the quotient, because $\ker(r)$ is normally generated by $\Gamma_{core}^{(0)}$. Notice that, by Proposition~\ref{prop:odd} plus the aforementioned \cite[Theorem 1.1]{DeyNeumann_free_prod_Hopf}, the quotient is Hopfian, so $\ov\phi$ is an isomorphism. Now, $r$ is injective on every dihedral $A_{bc}$ where $b$ and $c$ belong to the same hanging component, thus $\phi(A_{bc})$ must be a $\Z$-central extension of a non-elementary hyperbolic group because its $r$-projection $r(\phi(A_{bc}))=\ov\phi(r(A_{bc}))$ is again isomorphic to $A_{bc}$. As this argument works for every iterate of $\phi$, we get that every dihedral in a hanging component is \good{}.

At this point, removing centralisers of core vertices and dihedrals contained in hanging components is still not enough to make the product region graph discrete, so we need to find more \good{} stabilisers. By inspection of $\mathcal{PR}(A_\Gamma, \ov X)$, it is enough to prove the following: 

\begin{claim}\label{claim:ponte_tibetano}
Let $a,b,c\in \Gamma^{(0)}$ be such that $a\in \Gamma_{core}^{(0)}$, $b,c$ are adjacent vertices in the same hanging component, and $a$ is adjacent to $b$. Then $A_{ab}$ is \good{}. 
\end{claim}

\begin{claimproof}[Proof of Claim~\ref{claim:ponte_tibetano}]
By contradiction, up to passing to an iterated of $\phi$, assume that $\phi(A_{ab})$ is virtually $\Z^2$ and is not contained in an edge group (notice that $\phi(A_{ab})$ cannot be virtually cyclic, as the image of $A_{ab}$ in the abelianisation is isomorphic to $\Z^2$). By stability of product regions, there exists some vertex stabiliser $P$, say with centre generated by $z$, such that $\phi(A_{ab})\le P$. Since we already know that $\phi(a)$ is contained in the centre of some vertex stabiliser, we must have that $\phi(a)\in \langle z\rangle$. Moreover, $\phi(b)$ cannot belong to any edge group $E\le P$, as otherwise $\phi(A_{ab})=\langle\phi(a),\phi(b)\rangle\le E$ as well. Thus $\phi(z_{cb})$, which commutes with $\phi(b)$ and is contained in a centre, must belong to $\langle z\rangle$. This contradicts the fact that $a$ and $z_{bc}$ are non-commensurable in the abelianisation, so their $\phi$-images cannot lie in the same cyclic subgroup.
\end{claimproof}
The proof of Proposition~\ref{prop:3+_odd} is now done.
\end{proof}

\subsubsection{Two odd components}
We finally move to the case where $\Gamma$ has two odd components, which are therefore both hanging components. According to their shapes, we split the Proposition into two sub-lemmas.
\begin{prop}\label{prop:2_odd_needle}
    Let $A_\Gamma$ be a large Artin group of hyperbolic type, where $\Gamma$ is a connected graph on at least three vertices. Suppose that $\Gamma$ has two odd components, one of which is a needle. Then $A_\Gamma$ is Hopfian.
\end{prop}

\begin{proof} Pick a self-epimorphism $\phi$, and let $g_0\in \ker\phi-\{1\}$. Again, the goal is to find a quotient $A_\gamma\to G$ satisfying the requirements of Lemma~\ref{lem:enough}. Let $C$ and $C'$ be the two odd components, and assume without loss of generality that $|C|>1$. As $C'$ is a needle, the only vertex of $C'$, call it $b$, is adjacent to a unique vertex $a\in C$. Now, if $C(a)$ is \good{}, then Lemma~\ref{lem:quotient_for_good_stuff} produces a Hopfian quotient with the required properties (notice that, as $b$ is a leaf, its centraliser is not a vertex of $\ov X$, so removing $C(a)$ makes $\mathcal{PR}(A_\Gamma, \ov X)$ discrete). 
\par\medskip

Thus suppose that $C(a)$ is not \good{}. Up to replacing $\phi$ with a power, we can assume that $\phi(C(a))$ is virtually $\Z^2$ but not contained in an edge group (notice that $\phi(C(a))$ cannot be virtually cyclic, as the image of $C(a)$ in the abelianisation is isomorphic to $\Z^2$). By stability of product regions, $\phi(C(a))$ must be contained in a unique vertex stabiliser $P$. Furthermore, we claim that the whole component $C$ is mapped inside $P$. Indeed, let $a'\in C$ be connected to $a$ by an odd edge. Then $\phi(A_{aa'})\le P$, as it must centralise $\phi(z_{aa'})$ which lies in a non-edge $\Z^2$ subgroup of $P$. But then, since $a$ is conjugate to $a'$ by an element of $A_{aa'}$, we get that $\phi(C(a'))\le P$ as well, and it is again a virtually $\Z^2$ subgroup not contained in any edge group. As $C$ is an odd component, any two vertices are connected by a path with odd labels, so we get that $\phi(a'')\in P$ for every $a''\in C$.

Similarly, notice that $\phi(A_{ab})\le P$ as well, as it must be contained in the centraliser of $\phi(z_{ab})\le \phi(C(a))$. But this violates surjectivity, as then $\phi(A_\Gamma)$ is totally contained in $P$.
\end{proof}

\begin{prop}\label{prop:2_odd_broad}
    Let $A_\Gamma$ be a large Artin group of hyperbolic type, where $\Gamma$ is a connected graph on at least three vertices. Suppose that $\Gamma$ has two odd components, which are both broad. Then $A_\Gamma$ is Hopfian.
\end{prop}

\begin{proof}
    Let $C=\{a_1,\ldots, a_k\}$ and $C'=\{b_1,\ldots, b_r\}$ be the odd components, let $\phi$ be an epimorphism, and let $g_0\in\ker\phi-\{1\}$. Again, if both $C(a_1)$ and $C(b_1)$ are \good{}, then we can invoke Remark~\ref{rem:hopfian_from_good_stuff} and conclude. So suppose that $C(a_1)$ is not \good{}, so there exists a vertex stabiliser $P$, say with centre generated by $z$, such that $\phi(C(a_1))$ is a virtually $\Z^2$ subgroup of $P$ not contained in any edge group. Arguing as in Proposition~\ref{prop:2_odd_needle}, one gets that $\phi(C)\le P$, and $\phi(A_{a_ib_j})\le P$ for every $i,j$ such that $a_i$ and $b_j$ are joined by an even edge. We now consider three scenarios, A, B, and C, depending on the shape of $\phi(C(b_1))$. 
    \par\medskip
    \textbf{A.}
    Suppose first that $\phi(C(b_1))$ is a $\Z$-central extension of a non-elementary hyperbolic group, so $b_1$ must be sent inside some centre. As standard generators have primitive image in the abelianisation while centres of dihedrals do not, we can assume up to conjugation that $\phi(b_1)=b_1$ or $\phi(b_1)=a_1$. In the former case, $\phi$ preserves the normal closure of $C'$, so it induces a self-map $\ov\phi$ of the quotient $A_C$, obtained by retracting onto $C$. Then one can run the proof of Proposition~\ref{prop:3+_odd} verbatim, to get that $A_\Gamma$ has “enough” \good{} vertex stabilisers. 
\par\medskip

    In the latter case, $\phi^2(C(b_1))\le \phi(C(a_1))\le P$. It now matters what $P$ is, taking into account that it cannot be an odd dihedral as $\phi(C(a_1))$ must map to $\Z^2$ in the abelianisation. If $P$ is a conjugate of $C(b_1)$ then we can argue as above, with $\phi^2$ replacing $\phi$. If instead $P$ is conjugated to $Q\in\{C(a_1),A_{a_ib_j}\}$ by some $g\in G$, up to composing $\phi$ with the conjugation by $g^{-1}$ we can assume that $\phi(Q)\le Q$. Then we have that:
    \begin{itemize}
        \item $\phi^2(C)\le \phi(Q)\le Q$;
        \item $\phi^2(C')\le Q$, because $\phi^2(C(b_1))\le \phi(C(a_1))$ is contained in a $\Z^2$ subgroup which is not an edge group. 
    \end{itemize}
    Hence $\phi^2(A_\Gamma)\le Q$, violating surjectivity.

\par\medskip
    \textbf{B.}
    Suppose now that $\phi(C(b_1))$ is virtually $\Z^2$, but not contained in an edge group. Since $\phi(C(b_1))$ contains $\phi(z_{a_1b_1})$, which belongs to the non-edge $\Z^2$ subgroup $\phi(C(a_1))$ of $P$, we must have that $\phi(C(b_1))\le P$. Then again, using that $\phi(C(b_1))$ is not contained in an edge group, we get that $\phi(C')\le P$, which combined with $\phi(C)\le P$ violates surjectivity.
    \par\medskip
    \textbf{C.}
    Suppose finally that $\phi(C(b_1))$ is a finite-index subgroup of an edge group. As above, since $\phi(C(b_1))$ contains $\phi(z_{a_1b_1})$ we must have that $\phi(C(b_1))\le P$. We now consider the possible conjugacy types of $P$.
    \begin{itemize}
        \item Suppose first that $P$ is conjugate to $C(a_1)$, and up to composing $\phi$ with an inner automorphism we can indeed assume that $P=C(a_1)$. Then $\phi^2(C(b))\le \phi(C(a))\le P$ is a non-edge, virtually $\Z^2$ subgroup, and again this implies that $\phi(C')\le P$, contradicting surjectivity.
        \item Suppose now that $P$ is conjugate to some dihedral, which must be of the form $A_{a_ib_j}$ (again because any other dihedral has cyclic image in $A_\Gamma^{ab}$). Pick $b_2\in C'$ which is connected to $b_1$ by an odd edge. Then $\phi(z_{b_1b_2})\in P$, and as parabolics are root closed we must have that $\phi(b_1b_2)\in P$. Hence $\phi(C(b_2))$, which is obtained by conjugating $\phi(C(b_1))$ by a power of $\phi(b_1b_2)$, is also nested in the subgroup $P$. Proceeding this way, we eventually get that $\phi(C')\le P$, and once more $\phi$ could not be surjective. 
        \item The only case left is when $P$ is conjugate to $C(b_1)$, and again we can indeed assume that $P=C(b_1)$ up to composing $\phi$ with a conjugation. Say $C(b_1)=\Stab{A_\Gamma}{v}$ for some $v\in \ov X^{(0)}$. As $\phi(C(b_1))$ lies in some edge group, there must be some $w\in\link_{\ov X}(v)$ such that, if we set $Q=\Stab{A_\Gamma}{w}$, then $\phi(C(b_1))\le C(b_1)\cap Q$. Since $\ov X$ is bipartite, $Q$ must be conjugate to a dihedral, so the same trick as above shows that $\phi(C')\le Q$. If we show that $\phi(Q)\le C(b_1)$ then $\phi^2(C')\le C(b_1)$, and $\phi^2(C)\le \phi(C(b_1))\le C(b_1)$. This would then again contradict surjectivity. 
        \\
        To prove that $\phi(Q)\le C(b_1)$, we first notice that the image of $Q$ in the abelianisation must have rank 2, or it could not contain $\phi(C(b_1))$; hence there exists $i,j$ such that $Q$ is conjugate to $A_{a_ib_j}$. In turn, since $\phi(A_{a_ib_j})\le C(b_1)$, there must be some $P'$, which is a conjugate of $C(b_1)$, such that $\phi(Q)\le P'$. But $\phi(Q)$ contains $\phi^2(C(b_1))$, which is virtually $\Z^2$ and lies inside $\phi(C(b_1))\le C(b_1)$. As $\ov X$ is bipartite, any two different conjugates of $C(b_1)$ intersect along a virtually cyclic subgroup, so we must have that $P'=C(b_1)$, as required.
    \end{itemize}
This concludes the proof of Proposition~\ref{prop:2_odd_broad}, and in turn of Theorem~\ref{thm:XL_Hopfian}. \end{proof}

\subsection{Comments and previous results}
\begin{rem}[Generic Artin groups are Hopfian]\label{rem:generic_artin}
In \cite{GV_random}, Goldsborough and Vaskou devised a model of random Artin groups, where, given a complete graph on $n$ vertices, each edge label is chosen with uniform probability from the set $\{\infty, 2, \ldots, f(n)\}$, for some non-decreasing divergent function $f\colon \mathbb{N}\to \mathbb{N}$. A property of Artin groups is \emph{generic} if there exists a function $f_0\colon \mathbb{N}\to \mathbb{N}$ such that, for every choice of function $f\ge f_0$, the property holds with probability approaching $1$ as $n\to +\infty$. In the same paper, the authors prove that the class of extra-large Artin groups is generic. Moreover, let $p(n)=e(n)/f(n)$, where $e(n)$ is the cardinality of odd numbers in the set $\{\infty, 2, \ldots, f(n)\}$. For any choice of $f$, the probability that a random Artin group has a single odd component is the same as the probability that a random (unlabelled) graph on $n$ vertices, where each edge exists with probability $p(n)$, is connected. Such probability is known to approach $1$ as $n\to +\infty$ (see e.g. \cite{erdosrenyi}); hence, as our Theorem \ref{thm:XL_Hopfian} applies to XL Artin groups with a single odd component, we get that a generic Artin group is Hopfian, thus proving Corollary~\ref{corintro:generic} from the Introduction.
\end{rem}

\begin{rem}[Other generic classes]\label{rem:bmv}
   In \cite{BMV}, Blufstein, Martin, and Vaskou established the Hopf property for large hyperbolic type Artin groups which are either \emph{free-of-infinity} (the defining graph is complete) or \emph{XXXL} (all edge labels are at least 6). Both classes are generic, in the sense of Remark \ref{rem:generic_artin}. Figure~\ref{fig:comparison} provides examples of Artin groups covered by our result, by theirs, and by none of them. We stress that the techniques from \cite{BMV} are very different from ours, as they involve a full description of all homomorphisms between groups in their families; this also allows them to determine when such groups are co-Hopfian (every \emph{injective} homomorphism is an isomorphism).
\end{rem}

\begin{figure}[htp]
    \centering
    \includegraphics[width=0.75\textwidth, alt={examples of Artin groups which belong to Hopfian classes, by the aforementioned results}]{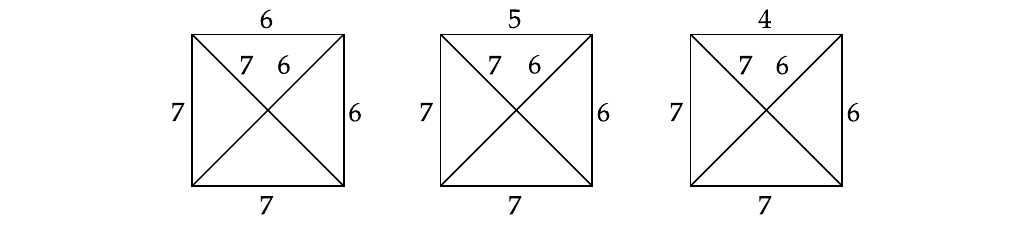}
    \caption{From left to right, an Artin group which is Hopfian by \cite{BMV} (it is XXXL), an Artin group which is Hopfian by our Theorem~\ref{thm:XL_Hopfian} (it has a single odd component), and an Artin group which is not covered by either methods. Notice that none of these Artin groups is known to be residually finite (see Remark~\ref{rmk:resfin} below).}
    \label{fig:comparison}
\end{figure}

\begin{rem}[Explicit residual hyperbolicity]\label{rmk:Shephard}
    As a special case of Corollary~\ref{cor:if_hyp_resfin}, if all hyperbolic groups are residually finite, then every Artin group $A_\Gamma$ of large and hyperbolic type is residually finite, hence Hopfian. On the one hand, it is common belief that there exists a non-residually finite hyperbolic group. On the other, the proof of Corollary~\ref{cor:fullres_hyp} shows that it would suffice that “enough” hyperbolic quotients of $A_\Gamma$ are residually finite, namely those with the following presentation, for a suitable choice of $N$: 
    $$\langle \Gamma^{(0)}\,|\,\forall c\in \Gamma^{(0)},\,\forall \{a,b\}\in E,\, \mathrm{prod}(a,b,m_{ab})=\mathrm{prod}(b,a,m_{ab}), \, c^N=(ab)^{m_{ab}N}=1\rangle.$$
    An intermediate quotient, falling in the family of \emph{Shephard groups}, is the following:
    $$S_\Gamma^N\coloneq \langle \Gamma^{(0)}\,|\,\,\forall c\in \Gamma^{(0)},\,\forall \{a,b\}\in E,\, \mathrm{prod}(a,b,m_{ab})=\mathrm{prod}(b,a,m_{ab}), \, c^N=1\rangle.$$
    The latter groups were studied in \cite{goldman20242dimensionalshephardgroups}, where the author proved that, if $\Gamma$ is triangle-free and large type, then $S_\Gamma^N$ is residually finite for all large enough $N$ (see \cite[Corollary F]{goldman20242dimensionalshephardgroups}). In turn, this is used in \cite[Theorem G]{goldman20242dimensionalshephardgroups} to prove that the corresponding Artin group $A_\Gamma$ is residually finite (notice that this now follows easily from the fact that every $g\in A_\Gamma$ survives in some $S_\Gamma^N$, as a consequence of  Corollary~\ref{cor:fullres_hyp}).
\end{rem}

\begin{rem}[Overview on residual finiteness for Artin groups]\label{rmk:resfin}
    Few classes of Artin groups are known to be residually finite, among which:
    \begin{itemize}
        \item Most Artin groups on three generators \cite{Squier_3gen_artin1987,Resfin_2D, Finite_stature_Artin,WuYe_splittings_triangle_Artin2025,Meyer_resfin_triangle_artin};
        \item Artin groups whose defining graph is triangle-free and contains no square whose edge labels are all 2 (this is the full statement of the aforementioned \cite[Theorem G]{goldman20242dimensionalshephardgroups});
        \item Even Artin groups of FC type (including RAAGs, see \cite{eafc_polyfree});
        \item Artin groups whose defining graphs are forests (because they are fundamental groups of graph manifolds with boundary \cite{Brunner1992, HermillerMeier_artin_grps_and_3mfd1999} and therefore they are virtually special \cite{Liu_virtual_cubulation_graph_mfd2013, PrzytyckiWise_graph_mfd_with_boundary_virtspecial2013});
        \item Spherical Artin groups (because they are linear, by e.g. \cite{CohenWales_linear_spherical} or \cite{digne});
        \item Certain 2-dimensional Artin groups, among which even XXXL Artin groups on graphs admitting a “partial orientation” (see \cite{Resfin_2D});
        \item “Forests” of residually finite parabolic subgroups (see \cite{Blascoetal} for details);
        \item Amalgams of residually finite Artin groups along parabolic subgroups which are retracts \cite{BolerEvans_amalgam_along_retract_is_resfin1973}.
    \end{itemize}
    Remarkably, none of the above families is generic in the sense of \cite{GV_random}.
\end{rem}

\begin{rem}[Equational Noetherianity]\label{rem:no_we_cant_barak}
Barak \cite{barak2024equational} recently established that, if $G$ is a colourable,  \emph{strictly acylindrical} HHG, then $G$ is \emph{equationally Noetherian}, hence Hopfian by e.g. \cite[Corollary 3.14 and Theorem D]{Groves_Hull:eq_noetherian}. For our purposes, the only consequence of strict acylindricity to keep in mind is that, for every $U\in \frakS$, its stabiliser acts acylindrically on $\C U$, as a corollary of \cite[Theorem 6.3]{bbfs}; in particular, if $\C U$ is unbounded, then $\Stab{G}U$ is either virtually cyclic or acylindrically hyperbolic.

In our setting, it is clear that an Artin group $A_\Gamma$ of large and hyperbolic type is colourable, as it is colourable as a short HHG. However, if $\Gamma$ has at least one edge, then the short HHG structure is not strictly acylindrical, as the centraliser of a vertex has infinite centre and is therefore not acylindrically hyperbolic.
\end{rem}

\section{Quotients of the five-holed sphere mapping class group}\label{sec:mcg_quot}
Let $S=S_{0,5}$ be a five-punctured sphere, and let $\MCG(S)$ be its extended mapping class group, which is a short HHG as pointed out in \cite[Subsection 2.3.1]{Short_HHG:I}. Our last theorem proves that almost every way of adding finitely many relations to $\MCG(S)$ results in a hierarchically hyperbolic group, up to stabilising by taking sufficient powers of the relators. This provides an almost complete answer to \cite[Question 3]{rigidity_mcg_mod_dt}, in the case of a five-punctured sphere.

\begin{thm}\label{thm:S5_quotient}
    Let $S=S_{0,5}$, and let $g_1,\dots,g_l\in \MCG(S)$. 
    Suppose the collection $\{g_i\}$ contains no pseudo-Anosovs with hidden symmetries, in the sense of Definition~\ref{defn:hiddensymm_general}. Then there exists $N\in\mathbb{N}-\{0\}$ such that, for all $K_1,\dots, K_l\in \mathbb{Z}-\{0\}$, we have that $\MCG(S)/\langle\langle \{g_i^{K_i N}\}\rangle\rangle$ is hierarchically hyperbolic.
\end{thm}

\begin{proof} Fix a curve $\gamma$ on $S$. Given $g_1,\dots,g_l\in \MCG(S)$, it suffices to prove the statement replacing each $g_i$ with $g_i^K$ for some $K\neq 0$. Therefore, as a consequence of Nielsen-Thurston classification (see \cite[Corollary 13.3]{FarbMargalit}), up to conjugation we can assume that each $g_i$ is of one of the following types:
 \begin{enumerate}
  \item A power of the Dehn twist $\tau_\gamma$ around $\gamma$;
  \item A power of a multitwist with associated multicurve $\{\gamma,\beta\}$, for a fixed curve $\beta$ which is disjoint from $\gamma$;
  \item A partial pseudo-Anosov without hidden symmetries, supported on the unique component $Y$ of $S-\gamma$ which is homeomorphic to $S_{0,4}$;
  \item A pseudo-Anosov.
 \end{enumerate}
In particular, every $g_i$ has infinite order. Up to taking further powers, we can make the following modifications to the collection of elements under consideration:
 \begin{itemize}
     \item  Suppose that two elements are commensurable up to conjugation, say for simplicity of notation $g_1$ and $g_2$. Then, up to taking a common power of all the elements in our collection, we can find $g\in \MCG(S)$ and integers $a,b$ such that $g_1=g^a$ and $g_2$ is conjugate to $g^b$. We can then replace $g_1$ and $g_2$ with $g^{gcd(|a|,|b|)}$, without changing the subgroup normally generated by the collection.
     \item  Similarly, suppose that two elements of type (3) are not commensurable in $\MCG(S)$ but have commensurable images in $\MCG(Y)$. Again, denote them by $g_1$ and $g_2$. Up to taking powers of all elements, and replacing $g_2$ with a conjugate, this means that there exists $g\in \MCG(S)$ of type (3) such that $g_1=g^a$ and $g_2=g^b\tau_\gamma^c$, for some integers $a,b,c$. The subgroup $\langle g_1,g_2\rangle$ coincides with $\langle\tau_\gamma^r, g^s\tau_\gamma^t\rangle$ for some integers $r,s,t$, since any subgroup of $\mathbb Z^2\cong \langle \tau_\gamma,g\rangle$ is of that form. We can then replace $g_1,g_2$ with $\tau_\gamma^r$ and $g^s\gamma^t$, and then possibly repeat the procedure in the previous bullet if the new collection contains commensurable elements.
 \end{itemize}
Assume first that there are no elements of type (2). By \cite[Proposition~5.1]{Short_HHG:I}, there exists a colourable short HHG structure for $\MCG(S)$ where every $g_i$ of type (4) generates a cyclic direction, up to taking a suitable power. Furthermore, by \cite[Proposition 5.2]{Short_HHG:I}, we can also assume that every element of type (3) has a power that generates a cyclic direction. Then Theorem~\ref{thm:quotient_short_HHG} ensures that we can find some integer $N>0$ such that $\MCG(S)/\langle\langle \{g_i^{K_i N}\}\rangle\rangle$ is a short HHG, and in particular hierarchically hyperbolic.
\par\medskip
    Now assume that there are elements of type (2), say $g_1=\tau_\gamma^{a_1}\tau_\beta^{b_1}, \ldots, g_r=\tau_\gamma^{a_r}\tau_\beta^{b_r}$, where all $a_i$ and $b_i$ are non-zero integers. Consider some choice of integers $N\neq 0$ and $K_i$. If some element, say $g_{r+1}$, is of the form $\tau_\gamma^c$, set $$d=gcd(a_1K_1,b_1K_1,\ldots, a_rK_r,b_rK_r,cK_{r+1}),$$ otherwise set $$d=gcd(a_1K_1,b_1K_1,\ldots, a_rK_r,b_rK_r).$$ Since $\tau_\beta$ is conjugate to $\tau_\gamma$, we have that $\mathcal N=\langle\langle g_i^{N K_i}\rangle\rangle \le \langle\langle \tau_\gamma^{dN}, g_{r+1}^{NK_{r+1}},\dots,g_l^{NK_l}\rangle\rangle$. Let $G_1$ and $G_2$ be the quotients of $\MCG(S)$ by the first and second group respectively, so that, similarly to above, $G_2$ is a short HHG if we choose $N$ suitably. In fact, in the short HHG structure for $\MCG(S)$ given by \cite[Propositions 5.1 and 5.2]{Short_HHG:I}, the cyclic directions are conjugate to powers of $\tau_\gamma, g_{r+1},\dots,g_l$, so $G_2$ is hyperbolic because we modded out all cyclic directions. Also, by the inclusion of kernels, we have a surjective homomorphism $\phi\colon G_1\to G_2$, and the kernel is normally generated by the image $\tau$ of $\tau_\gamma^{dN}$ in $G_1$. 
\par\medskip
We claim that $\tau$ is a central element of $G_1$. To see this, using that all Dehn twists are conjugate (in the mapping class groups of $S_{0,4}$, as well as of $S_{0,5}$), we first notice that there exists an integer $d'$ such that $\mathcal N$ contains $\tau_\gamma^{dN}\tau^{d'}_\alpha$, for any curve $\alpha$ disjoint from $\gamma$. In terms of $G_1$, this implies that $\tau$ coincides with the image of $\tau^{-d'}_\alpha$, and in particular $\tau$ commutes with the image of the half-twist around $\alpha$ or any curve disjoint from $\alpha$. Varying $\alpha$, the images of said half-twists generate $G_1$, so that $\tau$ commutes with a generating set of $G_1$, and is therefore central. 
Hence $G_1$ is a cyclic central extension of a hyperbolic group, and therefore it is a HHG by \cite[Corollary 4.3]{HRSS_3manifold}, as required.
\end{proof}

\begin{ex}\label{ex:symmetry}
We now exhibit a partial pseudo-Anosov without hidden symmetries, in order to show that the hypothesis of Theorem~\ref{thm:S5_quotient} is not vacuous. Let $\gamma$ be a curve on $S=S_{0,5}$, let $\alpha,\beta$ be disjoint from $\gamma$, and let $\tau_\alpha,\tau_\beta,\tau_\gamma$ be the associated Dehn twists. Let $Y$ be the connected component of $S-\gamma$ which is homeomorphic to $S_{0,4}$, let $p$ be the puncture of $Y$ coming from $\gamma$, and let $q$ be the puncture which is separated from $p$ by both $\alpha$ and $\beta$. By e.g. \cite[Proposition 3.19]{FarbMargalit}, the quotient $\Stab{\MCG(S)}{\gamma}/\langle \tau_\gamma\rangle$ is an index two overgroup of $\MCG(Y, \{p\})$, that is, the subgroup of $\MCG(Y)$ spanned by all elements that fix the puncture $p$. Let $i\in\Stab{\MCG(S)}{\gamma}$ be an-orientation-preserving mapping class whose image
$\ov i\in \MCG(Y, \{p\})$ swaps $\alpha$ and $\beta$ (for example, $i$ could be a rotation of angle $\pi$ around the axis passing through $q$). Let $g=\tau_\gamma\tau_\alpha\tau_\beta^{-1}$, which has a hidden symmetry because $\ov i\ov g\ov i^{-1}=\ov \tau_\beta\ov \tau_\alpha^{-1}=\ov g^{-1}$ while $igi^{-1}=\tau_\gamma\tau_\beta\tau_\alpha^{-1}=\tau_\gamma^2g^{-1}$.

This example should make the terminology "hidden symmetry" clearer, as an axis for $\ov g$ is “flipped” by the conjugation by $\ov i$ but this symmetry is not witnessed in the extension. Furthermore, the problem persists even if we only consider the orientable mapping class group $\mathcal{MCG}(S)$, as both $i$ and $g$ are orientation-preserving.
\end{ex}

\bibliography{biblio}

@article {HHS_II,
    AUTHOR = {Behrstock, Jason and Hagen, Mark and Sisto, Alessandro},
     TITLE = {Hierarchically hyperbolic spaces {II}: {C}ombination theorems
              and the distance formula},
   JOURNAL = {Pacific J. Math.},
  FJOURNAL = {Pacific Journal of Mathematics},
    VOLUME = {299},
      YEAR = {2019},
    NUMBER = {2},
     PAGES = {257--338},
      ISSN = {0030-8730,1945-5844},
   MRCLASS = {20F36 (20F65 20F67)},
  MRNUMBER = {3956144},
MRREVIEWER = {Jiming\ Ma},
       DOI = {10.2140/pjm.2019.299.257},
       URL = {https://doi.org/10.2140/pjm.2019.299.257},
}

@article {CMV:parabolics,
    AUTHOR = {Cumplido, Mar\'{\i}a and Martin, Alexandre and Vaskou,
              Nicolas},
     TITLE = {Parabolic subgroups of large-type {A}rtin groups},
   JOURNAL = {Math. Proc. Cambridge Philos. Soc.},
  FJOURNAL = {Mathematical Proceedings of the Cambridge Philosophical
              Society},
    VOLUME = {174},
      YEAR = {2023},
    NUMBER = {2},
     PAGES = {393--414},
      ISSN = {0305-0041,1469-8064},
   MRCLASS = {20F65 (20F36)},
  MRNUMBER = {4545212},
MRREVIEWER = {Michael\ Dougherty},
       DOI = {10.1017/S0305004122000342},
       URL = {https://doi.org/10.1017/S0305004122000342},
}

@article{shephard,
author = {Shephard, G. C.},
title = {Regular Complex Polytopes},
journal = {Proceedings of the London Mathematical Society},
volume = {s3-2},
number = {1},
pages = {82-97},
doi = {https://doi.org/10.1112/plms/s3-2.1.82},
url = {https://londmathsoc.onlinelibrary.wiley.com/doi/abs/10.1112/plms/s3-2.1.82},
eprint = {https://londmathsoc.onlinelibrary.wiley.com/doi/pdf/10.1112/plms/s3-2.1.82},
year = {1952}
}

@misc{goldman20242dimensionalshephardgroups,
      title={2-dimensional {Shephard} groups}, 
      author={Katherine Goldman},
      year={2024},
      eprint={2411.15434},
      archivePrefix={arXiv},
      primaryClass={math.GR},
      url={https://arxiv.org/abs/2411.15434}, 
}

@article {quasiflats,
    AUTHOR = {Behrstock, Jason and Hagen, Mark F. and Sisto, Alessandro},
     TITLE = {Quasiflats in hierarchically hyperbolic spaces},
   JOURNAL = {Duke Math. J.},
  FJOURNAL = {Duke Mathematical Journal},
    VOLUME = {170},
      YEAR = {2021},
    NUMBER = {5},
     PAGES = {909--996},
      ISSN = {0012-7094,1547-7398},
   MRCLASS = {20F67 (20F36 20F55 30F60 53C23)},
  MRNUMBER = {4255047},
MRREVIEWER = {Yasushi\ Yamashita},
       DOI = {10.1215/00127094-2020-0056},
       URL = {https://doi.org/10.1215/00127094-2020-0056},
}

@article {DHS,
    AUTHOR = {Durham, Matthew Gentry and Hagen, Mark F. and Sisto,
              Alessandro},
     TITLE = {Boundaries and automorphisms of hierarchically hyperbolic
              spaces},
   JOURNAL = {Geom. Topol.},
  FJOURNAL = {Geometry \& Topology},
    VOLUME = {21},
      YEAR = {2017},
    NUMBER = {6},
     PAGES = {3659--3758},
      ISSN = {1465-3060},
   MRCLASS = {20F65 (20F67 30F60)},
  MRNUMBER = {3693574},
MRREVIEWER = {Nadia Benakli},
       DOI = {10.2140/gt.2017.21.3659},
       URL = {https://doi.org/10.2140/gt.2017.21.3659},
}

@article {BHMS,
    AUTHOR = {Behrstock, Jason and Hagen, Mark and Martin, Alexandre and
              Sisto, Alessandro},
     TITLE = {A combinatorial take on hierarchical hyperbolicity and
              applications to quotients of mapping class groups},
   JOURNAL = {J. Topol.},
  FJOURNAL = {Journal of Topology},
    VOLUME = {17},
      YEAR = {2024},
    NUMBER = {3},
     PAGES = {Paper No. e12351, 94},
      ISSN = {1753-8416,1753-8424},
   MRCLASS = {20 (57K20 57M07)},
  MRNUMBER = {4822919},
       DOI = {10.1112/topo.12351},
       URL = {https://doi.org/10.1112/topo.12351},
}

@article{converse,
  title={A Combinatorial Structure for Many Hierarchically Hyperbolic Spaces},
  author={Hagen, Mark and Mangioni, Giorgio and Sisto, Alessandro},
  journal={Pacific J. Math.},
  number={2},
  volume={341},
  year={2026},
  pages={305–377},
  doi={https://doi.org/10.2140/pjm.2026.341.305},
}

@article{ELTAG_HHS,
  title={Extra-large type {Artin} groups are hierarchically hyperbolic},
  author={Hagen, Mark and Martin, Alexandre and Sisto, Alessandro},
  journal={Mathematische Annalen},
  pages={1--72},
  year={2022},
  publisher={Springer}
}

@article {russell,
    AUTHOR = {Russell, Jacob},
     TITLE = {From hierarchical to relative hyperbolicity},
   JOURNAL = {Int. Math. Res. Not. IMRN},
  FJOURNAL = {International Mathematics Research Notices. IMRN},
      YEAR = {2022},
    VOLUME = {1},
     PAGES = {575--624},
      ISSN = {1073-7928,1687-0247},
   MRCLASS = {57K20 (20F65 20F67)},
  MRNUMBER = {4366027},
MRREVIEWER = {Vassilis\ Metaftsis},
       DOI = {10.1093/imrn/rnaa141},
       URL = {https://doi.org/10.1093/imrn/rnaa141},
}

@article {dahmani:rotating,
    AUTHOR = {Dahmani, Fran\c{c}ois},
     TITLE = {The normal closure of big {D}ehn twists and plate spinning
              with rotating families},
   JOURNAL = {Geom. Topol.},
  FJOURNAL = {Geometry \& Topology},
    VOLUME = {22},
      YEAR = {2018},
    NUMBER = {7},
     PAGES = {4113--4144},
      ISSN = {1465-3060},
   MRCLASS = {20F65 (20E07)},
  MRNUMBER = {3890772},
MRREVIEWER = {Michael Hull},
       DOI = {10.2140/gt.2018.22.4113},
       URL = {https://doi.org/10.2140/gt.2018.22.4113},
}

@article {dfdt,
    AUTHOR = {Dahmani, Fran\c{c}ois and Hagen, Mark F. and Sisto, Alessandro},
     TITLE = {Dehn filling {D}ehn twists},
   JOURNAL = {Proc. Roy. Soc. Edinburgh Sect. A},
  FJOURNAL = {Proceedings of the Royal Society of Edinburgh. Section A.
              Mathematics},
    VOLUME = {151},
      YEAR = {2021},
    NUMBER = {1},
     PAGES = {28--51},
      ISSN = {0308-2105},
   MRCLASS = {20F65 (20F67 57K20)},
  MRNUMBER = {4202630},
MRREVIEWER = {Jens Harlander},
       DOI = {10.1017/prm.2020.1},
       URL = {https://doi.org/10.1017/prm.2020.1},
}

@article {hhs_asdim,
    AUTHOR = {Behrstock, Jason and Hagen, Mark F. and Sisto, Alessandro},
     TITLE = {Asymptotic dimension and small-cancellation for hierarchically
              hyperbolic spaces and groups},
   JOURNAL = {Proc. Lond. Math. Soc. (3)},
  FJOURNAL = {Proceedings of the London Mathematical Society. Third Series},
    VOLUME = {114},
      YEAR = {2017},
    NUMBER = {5},
     PAGES = {890--926},
      ISSN = {0024-6115,1460-244X},
   MRCLASS = {20F65 (20F67 57M07)},
  MRNUMBER = {3653249},
MRREVIEWER = {Mahan\ Mj},
       DOI = {10.1112/plms.12026},
       URL = {https://doi.org/10.1112/plms.12026},
}

@article {DeyNeumann_free_prod_Hopf,
    AUTHOR = {Dey, I. M. S. and Neumann, Hanna},
     TITLE = {The {H}opf property of free products},
   JOURNAL = {Math. Z.},
  FJOURNAL = {Mathematische Zeitschrift},
    VOLUME = {117},
      YEAR = {1970},
     PAGES = {325--339},
      ISSN = {0025-5874,1432-1823},
   MRCLASS = {20.52},
  MRNUMBER = {276352},
MRREVIEWER = {R.\ Bryce},
       DOI = {10.1007/BF01109851},
       URL = {https://doi.org/10.1007/BF01109851},
}

@article {DGO,
    AUTHOR = {Dahmani, F. and Guirardel, V. and Osin, D.},
     TITLE = {Hyperbolically embedded subgroups and rotating families in
              groups acting on hyperbolic spaces},
   JOURNAL = {Mem. Amer. Math. Soc.},
  FJOURNAL = {Memoirs of the American Mathematical Society},
    VOLUME = {245},
      YEAR = {2017},
    NUMBER = {1156},
     PAGES = {v+152},
      ISSN = {0065-9266,1947-6221},
      ISBN = {978-1-4704-2194-6; 978-1-4704-3601-8},
   MRCLASS = {20F65 (20F06 20F34 20F67 57M07)},
  MRNUMBER = {3589159},
MRREVIEWER = {Dominik\ Gruber},
       DOI = {10.1090/memo/1156},
       URL = {https://doi.org/10.1090/memo/1156},
}

@article {Osin_Dehn_Fill,
    AUTHOR = {Osin, Denis V.},
     TITLE = {Peripheral fillings of relatively hyperbolic groups},
   JOURNAL = {Invent. Math.},
  FJOURNAL = {Inventiones Mathematicae},
    VOLUME = {167},
      YEAR = {2007},
    NUMBER = {2},
     PAGES = {295--326},
      ISSN = {0020-9910,1432-1297},
   MRCLASS = {20F67 (20E26 20F06 20F65 57M27)},
  MRNUMBER = {2270456},
MRREVIEWER = {Ilya\ Kapovich},
       DOI = {10.1007/s00222-006-0012-3},
       URL = {https://doi.org/10.1007/s00222-006-0012-3},
}

@article {HHS_I,
    AUTHOR = {Behrstock, Jason and Hagen, Mark F. and Sisto, Alessandro},
     TITLE = {Hierarchically hyperbolic spaces, {I}: {C}urve complexes for
              cubical groups},
   JOURNAL = {Geom. Topol.},
  FJOURNAL = {Geometry \& Topology},
    VOLUME = {21},
      YEAR = {2017},
    NUMBER = {3},
     PAGES = {1731--1804},
      ISSN = {1465-3060},
   MRCLASS = {20F36 (20F55 20F65)},
  MRNUMBER = {3650081},
MRREVIEWER = {Nadia Benakli},
       DOI = {10.2140/gt.2017.21.1731},
       URL = {https://doi.org/10.2140/gt.2017.21.1731},
}

@article {BBF,
    AUTHOR = {Bestvina, Mladen and Bromberg, Ken and Fujiwara, Koji},
     TITLE = {Constructing group actions on quasi-trees and applications to
              mapping class groups},
   JOURNAL = {Publ. Math. Inst. Hautes \'{E}tudes Sci.},
  FJOURNAL = {Publications Math\'{e}matiques. Institut de Hautes \'{E}tudes
              Scientifiques},
    VOLUME = {122},
      YEAR = {2015},
     PAGES = {1--64},
      ISSN = {0073-8301,1618-1913},
   MRCLASS = {20E08 (20F34 20F65 20F69)},
  MRNUMBER = {3415065},
MRREVIEWER = {Dmytro\ M.\ Savchuk},
       DOI = {10.1007/s10240-014-0067-4},
       URL = {https://doi.org/10.1007/s10240-014-0067-4},
}

@article {bbfs,
    AUTHOR = {Bestvina, Mladen and Bromberg, Ken and Fujiwara, Koji and
              Sisto, Alessandro},
     TITLE = {Acylindrical actions on projection complexes},
   JOURNAL = {Enseign. Math.},
  FJOURNAL = {L'Enseignement Math\'ematique},
    VOLUME = {65},
      YEAR = {2019},
    NUMBER = {1-2},
     PAGES = {1--32},
      ISSN = {0013-8584,2309-4672},
   MRCLASS = {57M07 (20E08 20F65)},
  MRNUMBER = {4057354},
MRREVIEWER = {Jingyin\ Huang},
       DOI = {10.4171/lem/65-1/2-1},
       URL = {https://doi.org/10.4171/lem/65-1/2-1},
}

@article {Groves_Hull:eq_noetherian,
    AUTHOR = {Groves, D. and Hull, M.},
     TITLE = {Homomorphisms to acylindrically hyperbolic groups {I}:
              {E}quationally noetherian groups and families},
   JOURNAL = {Trans. Amer. Math. Soc.},
  FJOURNAL = {Transactions of the American Mathematical Society},
    VOLUME = {372},
      YEAR = {2019},
    NUMBER = {10},
     PAGES = {7141--7190},
      ISSN = {0002-9947,1088-6850},
   MRCLASS = {20F65 (20F67)},
  MRNUMBER = {4024550},
MRREVIEWER = {Thomas\ Koberda},
       DOI = {10.1090/tran/7789},
       URL = {https://doi.org/10.1090/tran/7789},
}

@article {Hruska_Wise,
    AUTHOR = {Hruska, G. Christopher and Wise, Daniel T.},
     TITLE = {Packing subgroups in relatively hyperbolic groups},
   JOURNAL = {Geom. Topol.},
  FJOURNAL = {Geometry \& Topology},
    VOLUME = {13},
      YEAR = {2009},
    NUMBER = {4},
     PAGES = {1945--1988},
      ISSN = {1465-3060,1364-0380},
   MRCLASS = {20F67 (20F65 20F69)},
  MRNUMBER = {2497315},
MRREVIEWER = {Fran\c{c}ois\ Dahmani},
       DOI = {10.2140/gt.2009.13.1945},
       URL = {https://doi.org/10.2140/gt.2009.13.1945},
}

@article {BMR97,
    AUTHOR = {Baumslag, Gilbert and Myasnikov, Alexei and Roman'kov,
              Vitaly},
     TITLE = {Two theorems about equationally {N}oetherian groups},
   JOURNAL = {J. Algebra},
  FJOURNAL = {Journal of Algebra},
    VOLUME = {194},
      YEAR = {1997},
    NUMBER = {2},
     PAGES = {654--664},
      ISSN = {0021-8693,1090-266X},
   MRCLASS = {20F99},
  MRNUMBER = {1467171},
MRREVIEWER = {R.\ M.\ Bryant},
       DOI = {10.1006/jabr.1997.7025},
       URL = {https://doi.org/10.1006/jabr.1997.7025},
}

@article {Val18,
    AUTHOR = {Valiunas, Motiejus},
     TITLE = {Acylindrical hyperbolicity of groups acting on quasi-median
              graphs and equations in graph products},
   JOURNAL = {Groups Geom. Dyn.},
  FJOURNAL = {Groups, Geometry, and Dynamics},
    VOLUME = {15},
      YEAR = {2021},
    NUMBER = {1},
     PAGES = {143--195},
      ISSN = {1661-7207,1661-7215},
   MRCLASS = {20F65 (20E06 20F67 20F70)},
  MRNUMBER = {4235751},
       DOI = {10.4171/ggd/595},
       URL = {https://doi.org/10.4171/ggd/595},
}

@article {ReinfeldWeidmann,
    AUTHOR = {Weidmann, Richard and Reinfeldt, Cornelius},
     TITLE = {Makanin-{R}azborov diagrams for hyperbolic groups},
   JOURNAL = {Ann. Math. Blaise Pascal},
  FJOURNAL = {Annales Math\'{e}matiques Blaise Pascal},
    VOLUME = {26},
      YEAR = {2019},
    NUMBER = {2},
     PAGES = {119--208},
      ISSN = {1259-1734,2118-7436},
   MRCLASS = {20F67 (20F65)},
  MRNUMBER = {4140867},
MRREVIEWER = {Denis\ E.\ Serbin},
       URL = {http://ambp.cedram.org/item?id=AMBP_2019__26_2_119_0},
}

@article {Blascoetal,
    AUTHOR = {Blasco-Garc\'ia, Rub\'en and Juh\'asz, Arye and Paris, Luis},
     TITLE = {Note on the residual finiteness of {A}rtin groups},
   JOURNAL = {J. Group Theory},
  FJOURNAL = {Journal of Group Theory},
    VOLUME = {21},
      YEAR = {2018},
    NUMBER = {3},
     PAGES = {531--537},
      ISSN = {1433-5883,1435-4446},
   MRCLASS = {20F36 (20F55)},
  MRNUMBER = {3794930},
MRREVIEWER = {Jon\ McCammond},
       DOI = {10.1515/jgth-2017-0049},
       URL = {https://doi.org/10.1515/jgth-2017-0049},
}

@article{HRSS_3manifold,
    AUTHOR = {Hagen, Mark and Russell, Jacob and Sisto, Alessandro and
              Spriano, Davide},
     TITLE = {Equivariant hierarchically hyperbolic structures for
              3-manifold groups via quasimorphisms},
   JOURNAL = {Ann. Inst. Fourier (Grenoble)},
  FJOURNAL = {Universit\'e{} de Grenoble. Annales de l'Institut Fourier},
    VOLUME = {75},
      YEAR = {2025},
    NUMBER = {2},
     PAGES = {769--828},
      ISSN = {0373-0956,1777-5310},
   MRCLASS = {20F67 (57K35)},
  MRNUMBER = {4921365},
       DOI = {10.5802/aif.3654},
       URL = {https://doi.org/10.5802/aif.3654},
}

@misc{inpreparation,
 author = {Fournier-Facio, Francesco and Mangioni, Giorgio and Sisto, Alessandro},
 title = {Bounded cohomology, quotient extensions, and hierarchical hyperbolicity},
 year = {2025},
 howpublished = {Preprint, {arXiv}:2505.20462 [math.{GR}] (2025)},
 url = {https://arxiv.org/abs/2505.20462},
 arXiv = {arXiv:2505.20462}
}

@article {paris_artin_groups,
    AUTHOR = {Paris, Luis},
     TITLE = {Parabolic subgroups of {A}rtin groups},
   JOURNAL = {J. Algebra},
  FJOURNAL = {Journal of Algebra},
    VOLUME = {196},
      YEAR = {1997},
    NUMBER = {2},
     PAGES = {369--399},
      ISSN = {0021-8693,1090-266X},
   MRCLASS = {20F36 (20F55)},
  MRNUMBER = {1475116},
MRREVIEWER = {Robert\ B.\ Howlett},
       DOI = {10.1006/jabr.1997.7098},
       URL = {https://doi.org/10.1006/jabr.1997.7098},
}

@article {eafc_polyfree,
    AUTHOR = {Blasco-Garc\'ia, Rub\'en and Mart\'inez-P\'erez, Conchita and
              Paris, Luis},
     TITLE = {Poly-freeness of even {A}rtin groups of {FC} type},
   JOURNAL = {Groups Geom. Dyn.},
  FJOURNAL = {Groups, Geometry, and Dynamics},
    VOLUME = {13},
      YEAR = {2019},
    NUMBER = {1},
     PAGES = {309--325},
      ISSN = {1661-7207,1661-7215},
   MRCLASS = {20F36 (20F24)},
  MRNUMBER = {3900773},
MRREVIEWER = {Volker\ Gebhardt},
       DOI = {10.4171/GGD/486},
       URL = {https://doi.org/10.4171/GGD/486},
}

@article {Resfin_2D,
    AUTHOR = {Jankiewicz, Kasia},
     TITLE = {Residual finiteness of certain 2-dimensional {A}rtin groups},
   JOURNAL = {Adv. Math.},
  FJOURNAL = {Advances in Mathematics},
    VOLUME = {405},
      YEAR = {2022},
     PAGES = {Paper No. 108487, 37},
      ISSN = {0001-8708,1090-2082},
   MRCLASS = {20F36 (20E26 20F65)},
  MRNUMBER = {4437605},
MRREVIEWER = {Valeriy\ G.\ Bardakov},
       DOI = {10.1016/j.aim.2022.108487},
       URL = {https://doi.org/10.1016/j.aim.2022.108487},
}

@book {FarbMargalit,
    AUTHOR = {Farb, Benson and Margalit, Dan},
     TITLE = {A primer on mapping class groups},
    SERIES = {Princeton Mathematical Series},
    VOLUME = {49},
 PUBLISHER = {Princeton University Press, Princeton, NJ},
      YEAR = {2012},
     PAGES = {xiv+472},
      ISBN = {978-0-691-14794-9},
   MRCLASS = {57M50 (20F36 20F65 57M07 57N05)},
  MRNUMBER = {2850125},
MRREVIEWER = {Stephen P. Humphries},
}

@article {croke-Kleiner,
    AUTHOR = {Croke, C. B. and Kleiner, B.},
     TITLE = {The geodesic flow of a nonpositively curved graph manifold},
   JOURNAL = {Geom. Funct. Anal.},
  FJOURNAL = {Geometric and Functional Analysis},
    VOLUME = {12},
      YEAR = {2002},
    NUMBER = {3},
     PAGES = {479--545},
      ISSN = {1016-443X,1420-8970},
   MRCLASS = {53C24 (37D40 53C23 53D25)},
  MRNUMBER = {1924370},
MRREVIEWER = {David\ Michael\ Fisher},
       DOI = {10.1007/s00039-002-8255-7},
       URL = {https://doi.org/10.1007/s00039-002-8255-7},
}

@article{MV:centralisers,
    AUTHOR = {Martin, Alexandre and Vaskou, Nicolas},
     TITLE = {Characterising large-type {A}rtin groups},
   JOURNAL = {Bull. Lond. Math. Soc.},
  FJOURNAL = {Bulletin of the London Mathematical Society},
    VOLUME = {56},
      YEAR = {2024},
    NUMBER = {11},
     PAGES = {3346--3357},
      ISSN = {0024-6093,1469-2120},
   MRCLASS = {20F36 (20F65)},
  MRNUMBER = {4828019},
       DOI = {10.1112/blms.13136},
       URL = {https://doi.org/10.1112/blms.13136},
}

@book {Thurston_topology3m,
    AUTHOR = {Thurston, William P.},
     TITLE = {Three-dimensional geometry and topology. {V}ol. 1},
    SERIES = {Princeton Mathematical Series},
    VOLUME = {35},
      NOTE = {Edited by Silvio Levy},
 PUBLISHER = {Princeton University Press, Princeton, NJ},
      YEAR = {1997},
     PAGES = {x+311},
      ISBN = {0-691-08304-5},
   MRCLASS = {57M50 (53A35 57M25 57M60 57N10)},
  MRNUMBER = {1435975},
MRREVIEWER = {Athanase Papadopoulos},
}

@incollection {Macpherson_virtcyclic,
    AUTHOR = {Macpherson, Dugald},
     TITLE = {Permutation groups whose subgroups have just finitely many
              orbits},
 BOOKTITLE = {Ordered groups and infinite permutation groups},
    SERIES = {Math. Appl.},
    VOLUME = {354},
     PAGES = {221--229},
 PUBLISHER = {Kluwer Acad. Publ., Dordrecht},
      YEAR = {1996},
      ISBN = {0-7923-3853-7},
   MRCLASS = {20B07},
  MRNUMBER = {1486202},
MRREVIEWER = {Cheryl\ E.\ Praeger},
}

@article {CohenWales_linear_spherical,
    AUTHOR = {Cohen, Arjeh M. and Wales, David B.},
     TITLE = {Linearity of {A}rtin groups of finite type},
   JOURNAL = {Israel J. Math.},
  FJOURNAL = {Israel Journal of Mathematics},
    VOLUME = {131},
      YEAR = {2002},
     PAGES = {101--123},
      ISSN = {0021-2172,1565-8511},
   MRCLASS = {20F36},
  MRNUMBER = {1942303},
MRREVIEWER = {Luis\ Paris},
       DOI = {10.1007/BF02785852},
       URL = {https://doi.org/10.1007/BF02785852},
}

@article {digne,
    AUTHOR = {Digne, Fran\c{c}ois},
     TITLE = {On the linearity of {A}rtin braid groups},
   JOURNAL = {J. Algebra},
  FJOURNAL = {Journal of Algebra},
    VOLUME = {268},
      YEAR = {2003},
    NUMBER = {1},
     PAGES = {39--57},
      ISSN = {0021-8693,1090-266X},
   MRCLASS = {20F36},
  MRNUMBER = {2004479},
       DOI = {10.1016/S0021-8693(03)00327-2},
       URL = {https://doi.org/10.1016/S0021-8693(03)00327-2},
}

@article{barak2024equational,
  title={On Equational Noetherianity of Colorable Hierarchically Hyperbolic Groups},
  author={Barak, Ohana},
  journal={arXiv preprint arXiv:2410.00977},
  year={2024}
}

@article{short_HHG:I,
  title={Short hierarchically hyperbolic groups {I}: uncountably many coarse median structures},
  author={Mangioni, Giorgio},
  journal={arXiv preprint arXiv:2410.09232},
  year={2024}
}

@article {neumannreeves,
    AUTHOR = {Neumann, Walter D. and Reeves, Lawrence},
     TITLE = {Central extensions of word hyperbolic groups},
   JOURNAL = {Ann. of Math. (2)},
  FJOURNAL = {Annals of Mathematics. Second Series},
    VOLUME = {145},
      YEAR = {1997},
    NUMBER = {1},
     PAGES = {183--192},
      ISSN = {0003-486X},
   MRCLASS = {20F32 (20J05)},
  MRNUMBER = {1432040},
MRREVIEWER = {John Meier},
       DOI = {10.2307/2951827},
       URL = {https://doi.org/10.2307/2951827},
}

@article{rigidity_mcg_mod_dt,
    author = {Mangioni, Giorgio and Sisto, Alessandro},
    title = {Rigidity of mapping class groups mod powers of twists},
    journal = {Proceedings of the Royal Society of Edinburgh: Section A Mathematics},
    year = {2025},
    DOI ={10.1017/prm.2025.23},
pages = {1-71}
}

@article{GV_random,
    AUTHOR = {Goldsborough, Antoine and Vaskou, Nicolas},
     TITLE = {Random {A}rtin groups},
   JOURNAL = {Algebr. Geom. Topol.},
  FJOURNAL = {Algebraic \& Geometric Topology},
    VOLUME = {25},
      YEAR = {2025},
    NUMBER = {3},
     PAGES = {1523--1544},
      ISSN = {1472-2747,1472-2739},
   MRCLASS = {20F36 (20F65 20F67 20F69 20P05)},
  MRNUMBER = {4930570},
       DOI = {10.2140/agt.2025.25.1523},
       URL = {https://doi.org/10.2140/agt.2025.25.1523},
}

@article {erdosrenyi,
    AUTHOR = {Erd\H{o}s, P. and R\'{e}nyi, A.},
     TITLE = {On the evolution of random graphs},
   JOURNAL = {Bull. Inst. Internat. Statist.},
  FJOURNAL = {Bulletin de l'Institut International de Statistique},
    VOLUME = {38},
      YEAR = {1961},
     PAGES = {343--347},
   MRCLASS = {05.40 (55.10)},
  MRNUMBER = {148055},
}

@misc{BMV,
      title={Homomorphisms between large-type {Artin} groups}, 
      author={Martín Blufstein and Alexandre Martin and Nicolas Vaskou},
      year={2024},
      eprint={2410.19091},
      archivePrefix={arXiv},
      primaryClass={math.GR},
      url={https://arxiv.org/abs/2410.19091}, 
}

@article {GM,
    AUTHOR = {Groves, Daniel and Manning, Jason Fox},
     TITLE = {Dehn filling in relatively hyperbolic groups},
   JOURNAL = {Israel J. Math.},
  FJOURNAL = {Israel Journal of Mathematics},
    VOLUME = {168},
      YEAR = {2008},
     PAGES = {317--429},
      ISSN = {0021-2172,1565-8511},
   MRCLASS = {57M50},
  MRNUMBER = {2448064},
MRREVIEWER = {Colin\ C.\ Adams},
       DOI = {10.1007/s11856-008-1070-6},
       URL = {https://doi.org/10.1007/s11856-008-1070-6},
}

@misc{mangioni2023rigidityresultslargedisplacement,
      title={Random quotients of mapping class groups are quasi-isometrically rigid}, 
      author={Giorgio Mangioni},
      year={2023},
      eprint={2312.00701},
      archivePrefix={arXiv},
      primaryClass={math.GR},
      url={https://arxiv.org/abs/2312.00701}, 
}

@misc{Wilton,
      title={The congruence subgroup property for mapping class groups and the residual finiteness of hyperbolic groups}, 
      author={Henry Wilton and Alessandro Sisto},
      year={2024},
      eprint={2410.00556},
      archivePrefix={arXiv},
      primaryClass={math.GR},
      url={https://arxiv.org/abs/2410.00556}, 
}

@Article{drutu_relhyp,
 Author = {Dru{\c{t}}u, Cornelia},
 Title = {Relatively hyperbolic groups: geometry and quasi-isometric invariance.},
 FJournal = {Commentarii Mathematici Helvetici},
 Journal = {Comment. Math. Helv.},
 ISSN = {0010-2571},
 Volume = {84},
 Number = {3},
 Pages = {503--546},
 Year = {2009},
 Language = {English},
 DOI = {10.4171/CMH/171},
 zbMATH = {5569658},
 Zbl = {1175.20032}
}

@proceedings{aim,
    title={AimPL: Geometry and topology of {Artin} groups},
    editor ={Katherine Goldman},
    note = {Available at http://aimpl.org/geomartingp}
}

@misc{Finite_stature_Artin,
 author = {Kasia Jankiewicz},
 title = {Finite {Stature} in {Artin} groups},
 year = {2026},
 howpublished = {Preprint, {arXiv}:2307.15209 [math.{GR}] (2026)},
 url = {https://arxiv.org/abs/2307.15209},
 arXiv = {arXiv:2307.15209},
note={To appear in \emph{Michigan Math. Journal}}
}

@article{Squier_3gen_artin1987,
  author = {Squier, Craig C.},
  title = {On Certain 3-Generator {Artin} Groups},
  journal = {Transactions of the American Mathematical Society},
  year = {1987},
  volume = {302},
  issn = {0002-9947},
  url = {http://dx.doi.org/10.2307/2000900},
  doi = {10.2307/2000900},
  number = {1},
  publisher = {JSTOR},
  month = July,
  pages = {117},
}

@article{WuYe_splittings_triangle_Artin2025,
  author = {Wu, Xiaolei and Ye, Shengkui},
  title = {Splittings and poly-freeness of triangle {Artin} groups},
  journal = {Transactions of the American Mathematical Society},
  year = {2025},
  volume = {378},
  issn = {0002-9947},
  url = {http://dx.doi.org/10.1090/tran/9408},
  doi = {10.1090/tran/9408},
  number = {7},
  publisher = {American Mathematical Society (AMS)},
  month = Apr,
  pages = {4707–4737},
}

@article{Meyer_resfin_triangle_artin,
  author = {Meyer, Greyson},
  title = {Residual finiteness of {{\(A_{2,3,2n}\)}} triangle {Artin} groups},
  fjournal = {The New York Journal of Mathematics},
  journal = {New York J. Math.},
  issn = {1076-9803},
  volume = {32},
  pages = {808--836},
  year = {2026},
  language = {English},
  url = {nyjm.albany.edu/j/2026/32-34.html},
  zbmath = {8210281},
}

@article{BolerEvans_amalgam_along_retract_is_resfin1973,
  author = {Boler, James and Evans, Benny},
  title = {The free product of residually finite groups amalgamated along retracts is residually finite},
  journal = {Proceedings of the American Mathematical Society},
  year = {1973},
  volume = {37},
  issn = {0002-9939},
  url = {http://dx.doi.org/10.1090/S0002-9939-1973-0306329-3},
  doi = {10.1090/s0002-9939-1973-0306329-3},
  number = {1},
  publisher = {American Mathematical Society (AMS)},
  month = Jan,
  pages = {50–52},
}

@article{Brunner1992,
  title = {Geometric quotients of link groups},
  volume = {48},
  issn = {0166-8641},
  url = {http://dx.doi.org/10.1016/0166-8641(92)90145-P},
  doi = {10.1016/0166-8641(92)90145-p},
  number = {3},
  journal = {Topology and its Applications},
  publisher = {Elsevier BV},
  author = {Brunner, A.M.},
  year = {1992},
  month = Dec,
  pages = {245–262},
}

@article{HermillerMeier_artin_grps_and_3mfd1999,
  author = {Hermiller, Susan M. and Meier, John},
  title = {Artin groups, rewriting systems and three-manifolds},
  journal = {Journal of Pure and Applied Algebra},
  year = {1999},
  volume = {136},
  issn = {0022-4049},
  url = {http://dx.doi.org/10.1016/S0022-4049(97)00171-0},
  doi = {10.1016/s0022-4049(97)00171-0},
  number = {2},
  publisher = {Elsevier BV},
  month = Mar,
  pages = {141–156},
}

@article{Liu_virtual_cubulation_graph_mfd2013,
  author = {Liu, Yi},
  title = {Virtual cubulation of nonpositively curved graph manifolds},
  journal = {Journal of Topology},
  year = {2013},
  volume = {6},
  issn = {1753-8416},
  url = {http://dx.doi.org/10.1112/jtopol/jtt010},
  doi = {10.1112/jtopol/jtt010},
  number = {4},
  publisher = {Wiley},
  month = Apr,
  pages = {793–822},
}

@article{PrzytyckiWise_graph_mfd_with_boundary_virtspecial2013,
  author = {Przytycki, Piotr and Wise, Daniel T.},
  title = {Graph manifolds with boundary are virtually special},
  journal = {Journal of Topology},
  year = {2013},
  volume = {7},
  issn = {1753-8416},
  url = {http://dx.doi.org/10.1112/jtopol/jtt009},
  doi = {10.1112/jtopol/jtt009},
  number = {2},
  publisher = {Wiley},
  month = Sept,
  pages = {419–435},
}
\bibliographystyle{alpha}
\end{document}